\documentclass[a4paper,english]{article}
\usepackage{babel}
\usepackage{zibtitlepage}

\usepackage[utf8]{inputenc}

\usepackage{dsfont}
\usepackage{fp}

\usepackage{mathtools}
\usepackage{amsfonts}
\usepackage{amssymb}
\usepackage{amsmath}

\usepackage{algorithmic}

\renewcommand{\algorithmiccomment}[1]{\bgroup\hfill~\texttt{/* #1 */}\egroup}

\usepackage[ruled,linesnumbered,vlined]{algorithm2e}
\SetAlgoCaptionSeparator{:}
\SetKwInOut{Input}{in}
\SetKwInOut{Output}{out}
\SetKwRepeat{Do}{do}{while}

\DontPrintSemicolon

\usepackage{tabularx,supertabular,booktabs,array}
\usepackage{pdflscape}
\usepackage{longtable}

\usepackage{tikz}
\usepackage{etex}
\usepackage{pgfplots}
\usetikzlibrary{shapes.gates.logic.US,shapes.geometric}
\usetikzlibrary{shapes,arrows,backgrounds,topaths}
\usetikzlibrary{arrows,decorations.markings}
\usetikzlibrary{matrix,snakes}
\usetikzlibrary{plotmarks}
\usetikzlibrary{decorations.text}
\usetikzlibrary{patterns}
\usetikzlibrary{calc}

\usepackage{url}

\usepackage{xcolor}
\usepackage{xspace}
\usepackage{subcaption}
\usepackage{xargs}
\usepackage[colorinlistoftodos,prependcaption,textsize=tiny]{todonotes}

\usepackage{enumitem}
\newlist{abbrv}{itemize}{1}
\setlist[abbrv,1]{label=,labelwidth=1in,align=parleft,itemsep=0.1\baselineskip,leftmargin=!}
 \newcommand{\mytitle}{On Generalized Surrogate Duality in Mixed-Integer Nonlinear Programming}

\newcommand{\name}[1]{\mbox{#1}\xspace}
\newcommand{\scip}{\name{SCIP}}

\newcommand{\cplex}{\name{CPLEX}}
\newcommand{\cppad}{\name{CppAD}}
\newcommand{\ipopt}{\name{Ipopt}}
\newcommand{\mumps}{\name{Mumps}}
\newcommand{\bliss}{\name{bliss}}
\newcommand{\globallib}{\name{GLOBALLib}}
\newcommand{\minlplib}{\name{MINLPLib}}

\newcommand{\MINLP}{\name{MINLP}}
\newcommand{\MINLPs}{\name{MINLPs}}

\newcommand{\MILP}{\name{MILP}}
\newcommand{\MILPs}{\name{MILPs}}
\newcommand{\NLP}{\name{NLP}}

\newcommand{\LP}{\name{LP}}
\newcommand{\LPs}{\name{LPs}}

\newcommand{\RLT}{\name{RLT}}
\newcommand{\SDP}{\name{SDP}}

\newcommand{\T}{\mathsf{T}}

\newcommand{\mnorm}[1]{\left\|#1\right\|}

\newcommand{\R}{\mathbb{R}}
\newcommand{\Z}{\mathbb{Z}}
\newcommand{\N}{\mathbb{N}}

\newcommand{\dash}{---}

\newcommand{\st}{\text{s.t.}}
\newcommand{\fa}{\text{ for all }}

\newcommand{\tabledefline}[2]{\multicolumn{1}{l}{\rlap{#1\ \dash\ #2}}\\}

\definecolor{plum}{rgb}{0.8, 0.6, 0.8}
\definecolor{applegreen}{rgb}{0.55, 0.71, 0.0}
\definecolor{apricot}{rgb}{0.98, 0.81, 0.69}
\definecolor{amber}{rgb}{1.0, 0.49, 0.0}
\definecolor{americanrose}{rgb}{1.0, 0.01, 0.24}

\newcommandx{\change}[2][1=]{\todo[linecolor=americanrose,backgroundcolor=americanrose!25,bordercolor=americanrose,#1]{#2}\,}
\newcommandx{\improvement}[2][1=]{\todo[linecolor=amber,backgroundcolor=amber!25,bordercolor=amber,#1]{#2}\,}
\newcommandx{\unsure}[2][1=]{\todo[linecolor=apricot,backgroundcolor=apricot!25,bordercolor=apricot,#1]{#2}\,}
\newcommandx{\info}[2][1=]{\todo[linecolor=applegreen,backgroundcolor=applegreen!25,bordercolor=applegreen,#1]{#2}\,}
\newcommandx{\missing}[2][1=]{\todo[linecolor=blue,backgroundcolor=blue!25,bordercolor=blue,#1]{#2}\,}

\newcommand{\psd}{\succeq}

\DeclareMathOperator*{\argmin}{argmin}

\newcommand{\norm}[1]{\left\|#1\right\|}

\newcommand{\nlconssidx}{\mathcal{M}}

\newcommand{\nvars}{n}
\newcommand{\nintvars}{p}
\newcommand{\nnlconss}{m}

\newcommand{\minlpfeasset}{\mathcal{F}}
\newcommand{\minlprelaxset}{X}
\newcommand{\sfct}{S}
\newcommand{\sset}{S}
\newcommand{\gsfct}[1]{S^{#1}}
\newcommand{\gsset}[1]{S^{#1}}
\newcommand{\pointset}{\mathcal{X}}

\newcommand{\GC}{GC}

\newcommand{\colinstances}{\# instances}
\newcommand{\colcons}{cons}
\newcommand{\colobjsense}{obj}
\newcommand{\colprimal}{best primal}
\newcommand{\coldb}{DB}
\newcommand{\colmilp}{\MILP}
\newcommand{\colsurrogate}{S}
\newcommand{\coliter}{iter}
\newcommand{\colgapclosed}{gc}
\newcommand{\colrgapclosed}{rgc}
\newcommand{\colwins}{M}
\newcommand{\colloses}{L}

\newcommand{\floatsub}[2]{
    \FPsub\myres{#1}{#2}\FPeval\myres{round(myres:1)}\myres
}

\newcommand{\experimentRoot}{\texttt{ROOTGAP}}
\newcommand{\experimentAlgo}{\texttt{BENDERS}}
\newcommand{\experimentDual}{\texttt{DUALBOUND}}

\def\rootGapWorseKOneKTwo{173}

\def\rootGapWorseKTwoKThree{105}

\def\rootAllSize{633}
\def\rootGapAllKOne{18.4}
\def\rootGapAllKTwo{21.4}
\def\rootGapAllKThree{23.4}
\def\rootAffectedSize{469}
\def\rootGapAffectedKOne{35.0}
\def\rootGapAffectedKTwo{42.2}
\def\rootGapAffectedKThree{46.9}
\def\rootNconssTenSize{528}
\def\rootGapNconssTenKOne{14.6}
\def\rootGapNconssTenKTwo{16.9}
\def\rootGapNconssTenKThree{18.4}
\def\rootNconssTwentySize{391}
\def\rootGapNconssTwentyKOne{10.7}
\def\rootGapNconssTwentyKTwo{12.3}
\def\rootGapNconssTwentyKThree{13.5}
\def\rootNconssFiftySize{229}
\def\rootGapNconssFiftyKOne{7.1}
\def\rootGapNconssFiftyKTwo{7.9}
\def\rootGapNconssFiftyKThree{8.5}

\def\rootNconssTenAfSize{370}
\def\rootGapNconssTenAfKOne{30.1}
\def\rootGapNconssTenAfKTwo{36.0}
\def\rootGapNconssTenAfKThree{40.1}
\def\rootNconssTwentyAfSize{244}
\def\rootGapNconssTwentyAfKOne{26.2}
\def\rootGapNconssTwentyAfKTwo{31.5}
\def\rootGapNconssTwentyAfKThree{35.4}
\def\rootNconssFiftyAfSize{115}
\def\rootGapNconssFiftyAfKOne{23.9}
\def\rootGapNconssFiftyAfKTwo{28.0}
\def\rootGapNconssFiftyAfKThree{30.8}

\def\rootGapNotaffected{164}
\def\rootGapNoGapAllKOneKTwo{21}

\def\rootGapNoGapAllKTwoKThree{11}
 
\def\algoGroupSizeAll{457}
\def\algoWinsPlainAll{100}
\def\algoLosesPlainAll{36}
\def\algoRelgapPlainAll{86.9}
\def\algoWinsNotrustAll{40}
\def\algoLosesNotrustAll{31}
\def\algoRelgapNotrustAll{98.3}
\def\algoWinsNosuppAll{41}
\def\algoLosesNosuppAll{27}
\def\algoRelgapNosuppAll{98.4}
\def\algoWinsNoearlyAll{94}
\def\algoLosesNoearlyAll{40}
\def\algoRelgapNoearlyAll{88.2}
\def\algoGroupSizeNconssTen{346}
\def\algoWinsPlainNconssTen{90}
\def\algoLosesPlainNconssTen{29}
\def\algoRelgapPlainNconssTen{84.1}
\def\algoWinsNotrustNconssTen{34}
\def\algoLosesNotrustNconssTen{27}
\def\algoRelgapNotrustNconssTen{98.7}
\def\algoWinsNosuppNconssTen{38}
\def\algoLosesNosuppNconssTen{24}
\def\algoRelgapNosuppNconssTen{98.5}
\def\algoWinsNoearlyNconssTen{85}
\def\algoLosesNoearlyNconssTen{33}
\def\algoRelgapNoearlyNconssTen{85.4}
\def\algoGroupSizeNconssTwenty{222}
\def\algoWinsPlainNconssTwenty{72}
\def\algoLosesPlainNconssTwenty{10}
\def\algoRelgapPlainNconssTwenty{77.2}
\def\algoWinsNotrustNconssTwenty{25}
\def\algoLosesNotrustNconssTwenty{17}
\def\algoRelgapNotrustNconssTwenty{98.5}
\def\algoWinsNosuppNconssTwenty{32}
\def\algoLosesNosuppNconssTwenty{20}
\def\algoRelgapNosuppNconssTwenty{98.2}
\def\algoWinsNoearlyNconssTwenty{65}
\def\algoLosesNoearlyNconssTwenty{12}
\def\algoRelgapNoearlyNconssTwenty{79.7}
\def\algoGroupSizeNconssFifty{107}
\def\algoWinsPlainNconssFifty{35}
\def\algoLosesPlainNconssFifty{1}
\def\algoRelgapPlainNconssFifty{74.2}
\def\algoWinsNotrustNconssFifty{13}
\def\algoLosesNotrustNconssFifty{7}
\def\algoRelgapNotrustNconssFifty{98.9}
\def\algoWinsNosuppNconssFifty{14}
\def\algoLosesNosuppNconssFifty{12}
\def\algoRelgapNosuppNconssFifty{98.4}
\def\algoWinsNoearlyNconssFifty{32}
\def\algoLosesNoearlyNconssFifty{2}
\def\algoRelgapNoearlyNconssFifty{76.3}

 \newcommand{\algoDefault}{\texttt{DEFAULT}}
\newcommand{\algoPlain}{\texttt{PLAIN}}
\newcommand{\algoNoTR}{\texttt{NOSTAB}}
\newcommand{\algoNoSupp}{\texttt{NOSUPP}}
\newcommand{\algoNoEarly}{\texttt{NOEARLY}}

\def\treeExpNinstances{209}
\def\treeExpNinstancesBetter{53}
\def\treeExpAvgGapSCIP{284.3}
\def\treeExpAvgGapBenders{142.8}  
\usepackage{amsthm}
\newtheorem{proposition}{Proposition}
\newtheorem{theorem}{Theorem}
\newtheorem{remark}{Remark}

\newtheorem{definition}{Definition}
\newtheorem{example}{Example}

\newcommand{\myorcidlink}[1]{\,\href{https://orcid.org/#1}{\raisebox{-0.45ex}{\includegraphics[width=1.8ex]{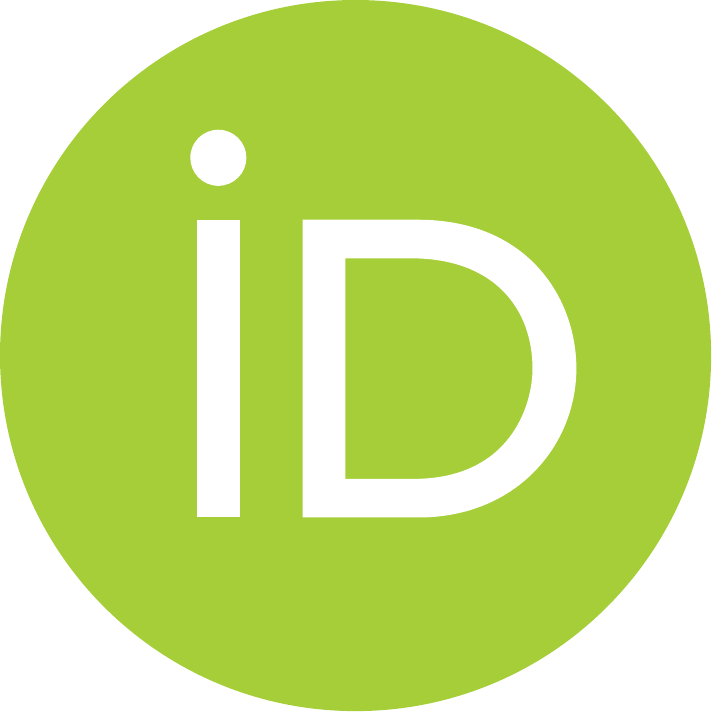}}}}

\newcommand{\myorcid}[1]{ORCID \href{https://orcid.org/#1}{#1}}

\begin{document}

\ZTPAuthor{
  Benjamin Müller\protect\myorcidlink{0000-0002-4463-2873},
  Gonzalo Mu\~{n}oz\protect\myorcidlink{0000-0002-9003-441X},
  Maxime Gasse\protect\myorcidlink{0000-0001-6982-062X},
  Ambros Gleixner\protect\myorcidlink{0000-0003-0391-5903},
  Andrea Lodi\protect\myorcidlink{0000-0001-9269-633X},
  Felipe Serrano\protect\myorcidlink{0000-0002-7892-3951}
}

\ZTPTitle{\bf \mytitle}
\ZTPInfo{This work has been supported by the Research Campus MODAL
  \emph{Mathematical Optimization and Data Analysis Laboratories} funded by the
  Federal Ministry of Education and Research (BMBF Grant~05M14ZAM).  All
  responsibility for the content of this publication is assumed by the authors.}
\ZTPNumber{19-55}
\ZTPMonth{December}
\ZTPYear{2019}

\title{\bf \mytitle}

\author{Benjamin Müller%
  \thanks{\itshape Zuse Institute Berlin, Berlin, Germany,
    E-mail \nolinkurl{benjamin.mueller@zib.de}, \newline\myorcid{0000-0002-4463-2873}\vspace{0.5ex}}
  \and
  Gonzalo Mu\~{n}oz%
  \thanks{\itshape Universidad de O’Higgins, Rancagua, Chile,
    E-mail \nolinkurl{gonzalo.munoz@uoh.cl}, \newline\myorcid{0000-0002-9003-441X}\vspace{0.5ex}}
  \and
  Maxime Gasse%
  \thanks{\itshape Polytechnique Montr\'{e}al, Montr\'{e}al, Canada,
    E-mail \nolinkurl{maxime.gasse@polymtl.ca}, \newline\myorcid{0000-0001-6982-062X}\vspace{0.5ex}}
  \and
  Ambros Gleixner%
  \thanks{\itshape Zuse Institute Berlin, Berlin, Germany,
    E-mail \nolinkurl{gleixner@zib.de}, \newline\myorcid{0000-0003-0391-5903}\vspace{0.5ex}}
  \and
  Andrea Lodi%
  \thanks{\itshape Polytechnique Montr\'{e}al, Montr\'{e}al, Canada,
    E-mail \nolinkurl{andrea.lodi@polymtl.ca}, \newline\myorcid{0000-0001-9269-633X}\vspace{0.5ex}}
  \and
  Felipe Serrano%
  \thanks{\itshape Zuse Institute Berlin, Berlin, Germany,
    E-mail \nolinkurl{serrano@zib.de}, \newline\myorcid{0000-0002-7892-3951}}
}


\hypersetup{pdftitle={\mytitle},
  pdfauthor={Benjamin Müller, Gonzalo Mu\~{n}oz, Maxime Gasse, Ambros Gleixner, Andrea Lodi, Felipe Serrano}}

\zibtitlepage
\maketitle

\begin{abstract}
  
The most important ingredient for solving mixed-integer nonlinear programs (\MINLPs) to global $\epsilon$-optimality
with spatial branch and bound is a tight, computationally tractable relaxation. Due to both theoretical and practical
considerations, relaxations of \MINLPs are usually required to be convex. Nonetheless, current optimization solver can often
successfully handle a moderate presence of nonconvexities, which opens the door for the use of potentially tighter
nonconvex relaxations.
In this work, we exploit this fact and make use of a nonconvex relaxation obtained via aggregation of constraints: a
\emph{surrogate} relaxation. These relaxations were actively studied for linear integer programs in the 70s and 80s, but
they have been scarcely considered since. We revisit these relaxations in an \MINLP setting and show
the computational benefits and challenges they can have. Additionally, we study a generalization of such
relaxation that allows for multiple aggregations simultaneously and present the first algorithm that is capable of
computing the best set of aggregations. We propose a multitude of computational enhancements for improving its practical
performance and evaluate the algorithm's ability to generate strong dual bounds through extensive computational
experiments.

 \end{abstract}

\section{Introduction}
\label{section:intro}

We consider a mixed-integer nonlinear program (\MINLP) of the form
\begin{equation}
\label{eq:minlp}
  \min_{x \in \minlprelaxset} \; \left\{ c^\T x \mid g_i(x) \le 0 \fa i \in \nlconssidx \right\},
\end{equation}
where~$\minlprelaxset := \{x \in \R^{\nvars-\nintvars} \times \Z^{\nintvars} \mid Ax \le b\}$ is a compact mixed-integer linear set, each
$g_i :\R^\nvars \to \R$ is a factorable continuous function~\cite{McCormick1976}, and $\nlconssidx := \{1,\ldots,\nnlconss\}$ denotes the index
set of nonlinear constraints. Such a problem is called \emph{nonconvex} if at least one $g_i$ is nonconvex, and \emph{convex} otherwise. Many real-world applications
are inherently nonlinear, and can be formulated as a~\MINLP. See, e.g.,~\cite{GrossmanSahinidis2002} for an overview.

The state-of-the-art algorithm for solving nonconvex \MINLPs to $\epsilon$-global optimality is spatial
branch and bound, see, e.g.,~\cite{HorstTuy1996,QuesadaGrossmann1995,RyooSahinidis1995}. The performance of spatial
branch and bound mainly depends on the tightness of the relaxations used, which are typically convex relaxations
constructed from the convexification of the nonlinear and integrality constraints. These convex relaxations are refined by branching,
cutting planes, and variable bound tightening, e.g., feasibility- and optimality-based bound
tightening~\cite{QuesadaGrossmann1993,BelottiCafieriLeeLiberti2012TR}.
As a result of the rapid progress during the last decades, current solvers can often handle a moderate presence of nonconvex constraints
efficiently. This opens the door for a practical use of potentially  tighter \emph{nonconvex} relaxations in \MINLP solvers.
One example are \MILP relaxations, see~\cite{Zhou2017,Misener2014,Burlacu2019}.
In this paper we go one step further and explore the concept of \emph{surrogate relaxations}~\cite{Glover1965}.
Even more aggressively, this leads to \MINLP relaxations that may contain both discrete and continuous nonlinearities.

\begin{definition}[Surrogate relaxation]
For a given $\lambda \in \R^\nnlconss_+$, we call the optimization problem
\begin{equation}
\label{eq:surrogate:opt}
\sfct(\lambda) := \min_{x \in \minlprelaxset} \; \left\{ c^\T x \mid \sum_{i \in \nlconssidx} \lambda_i g_i(x) \le 0 \right\}
\end{equation}
a \emph{surrogate relaxation} of~\eqref{eq:minlp}.
\end{definition}

Consider $\minlpfeasset \subseteq \R^\nvars$ the feasible region of~\eqref{eq:minlp}, and $\sset_\lambda \subseteq \R^\nvars$
that of~\eqref{eq:surrogate:opt}. Throughout the whole paper we assume that $\minlpfeasset$ is not empty. Clearly,
$\minlpfeasset \subseteq S_\lambda$ holds for every $\lambda \in \R^\nnlconss_+$, and as such~\eqref{eq:surrogate:opt}
provides a valid lower bound of~\eqref{eq:minlp}.
Moreover, solving~\eqref{eq:surrogate:opt} might be computationally more
convenient than solving the original problem~\eqref{eq:minlp}, since there is
only one nonconvex constraint in $\sfct(\lambda)$.
Note that one may turn $\sset_\lambda$ into a continuous relaxation of~\eqref{eq:minlp} by removing
the integrality restrictions. However, in this work we purposely choose to retain integrality in the relaxation
and compare directly to the optimal value of~\eqref{eq:minlp} as opposed to its continuous relaxation.

The quality of the bound provided by $\sfct(\lambda)$ may be highly
dependent on the value of $\lambda$, and therefore it is natural to
consider the \emph{surrogate dual} problem in order to obtain the tightest
surrogate relaxation.
\begin{definition}[Surrogate dual]
We call the optimization problem
\begin{equation}
\label{eq:surrogate:dual}
  \sup_{\lambda \in \R^\nnlconss_+} \; \sfct(\lambda) \\
\end{equation}
the \emph{surrogate dual} of~\eqref{eq:minlp}.
\end{definition}
\noindent The function $\sfct$ is lower semi-continuous~\cite{Greenberg1970}, which means that for every
sequence $\{\lambda^t\}_{t \in \N} \subseteq \R^\nnlconss_+$ with $\lim_{t \rightarrow \infty} \lambda^t = \lambda^*$, the inequality
\begin{equation*}
  \sfct(\lambda^*) \le \liminf_{t \rightarrow \infty} \sfct(\lambda^t)
\end{equation*}
holds. Thus, there might be no $\lambda \in \R^\nnlconss_+$
such that $\sfct(\lambda)$ is equal to~\eqref{eq:surrogate:dual}, as can be observed in Figure~\ref{fig:surrogate:plot}. Therefore,
it is not possible to replace the supremum in~\eqref{eq:surrogate:dual} by a maximum.

The surrogate dual is closely related to the well-known \emph{Lagrangian dual} $ \max_{\lambda \in \R^\nnlconss_+} L(\lambda)$, where
\begin{equation}
 L(\lambda) := \min_{x \in \minlprelaxset} \left\{ c^\T x + \sum_{i \in \nlconssidx} \lambda_i g_i(x) \right\}\text{,}
\end{equation}
but always results in a bound that is at least as good as the Lagrangian one~\cite{Greenberg1970,Karwan1979}, i.e.,
\begin{align}
 \sfct(\lambda) \ge L(\lambda) & \quad \fa \lambda \in \R^\nnlconss_+.
\end{align}

Figure~\ref{fig:surrogate:plot} shows the difference between $\sfct(\lambda)$ and $L(\lambda)$ on the two-dimensional
instance of Example~\ref{example:xy}. In contrast to $L : \R^\nnlconss_+ \rightarrow \R$, which is a continuous and concave
function, $\sfct : \R^\nnlconss_+ \rightarrow \R$ is only quasi-concave \cite{Greenberg1970} (i.e., the set
$\{\lambda \in \R^\nnlconss_+ \mid \sfct(\lambda) \geq \alpha\}$ is convex for all $\alpha$) and in general is discontinuous.
As it can be seen in Figure~\ref{fig:surrogate:plot}, the main difficulty in optimizing $\sfct(\lambda)$ for nonconvex \MINLPs is that
the function is most of the time ``flat'', meaning that it leads to nontrivial dual bounds for only a small subset of the $\lambda$-space.
To the best of our knowledge, this aspect has not received much attention in the development of algorithms that solve~\eqref{eq:surrogate:dual}
for general \MINLPs.

\begin{example} \normalfont \label{example:xy}
Consider the following nonconvex problem
\begin{equation*}
  \begin{aligned}
    \min \quad & -y \\
    \st \quad  & 2xy + x^2 - y^2 - x \le 0  , \\
               & -xy - 0.3x^2 - 0.2y^2 - 0.5x + 1.5y \le 0  , \\
               & (x,y)\in [0, 1]^2  ,
  \end{aligned}
\end{equation*}
which attains its optimal value $-0.37$ at $(x^*,y^*) \approx (0.52,0.37)$. The surrogate dual problem reads as
\begin{equation*}
  \sup_{(\lambda_1,\lambda_2) \in \R_+^2} \left\{\begin{aligned}
  \min \quad & -y \\
  \st \quad  & \lambda_1 (2xy + x^2 - y^2 - x) \\ & + \lambda_2 (-xy - 0.3x^2 - 0.2y^2 - 0.5x + 1.5y) \le 0 \\
             & (x,y)\in [0, 1]^2
  \end{aligned}
  \right\}  ,
\end{equation*}
whereas the Lagrangian dual problem is
\begin{equation*}
  \sup_{(\lambda_1,\lambda_2) \in \R_+^2} \left\{\begin{aligned}
  \min \quad & -y + \lambda_1 (2xy + x^2 - y^2 - x) \\ & + \lambda_2 (-xy - 0.3x^2 - 0.2y^2 - 0.5x + 1.5y) \\
  \st \quad  & (x,y)\in [0, 1]^2
  \end{aligned}
  \right\} .
\end{equation*}
Here the optimal solution value of the surrogate dual is $\sfct(0.56,0.44) \approx -0.38$, which is stronger
than that of the Lagrangian dual $L(0.67,0.82) \approx -0.78$. Note that for this
problem neither the surrogate nor the Lagrangian dual proves global optimality of $(x^*,y^*)$.
\end{example}

\begin{figure}[tb]
 \begin{center}
  \hspace{10ex}\includegraphics[width=0.55\textwidth]{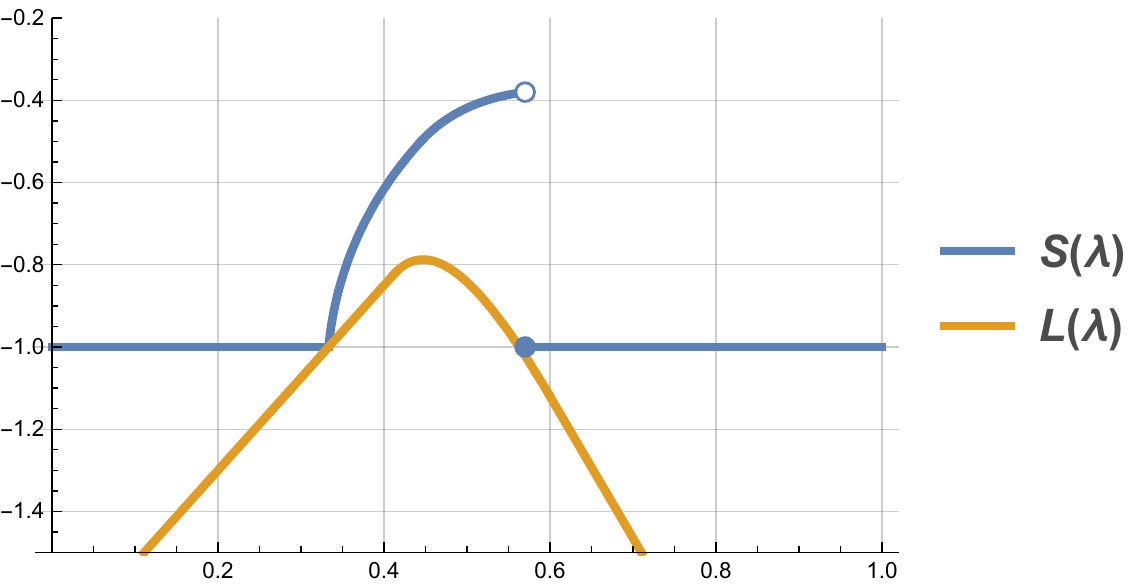}
     \caption{Plot of the surrogate and Lagrangian relaxations values for the \MINLP detailed in Example~\ref{example:xy}.
     Note that for display purposes we plot both relaxations with respect to
     $\lambda_1$ only. Since the surrogate relaxation is invariant to scaling, we can add a normalization constraint
     $\mnorm{\lambda}_1 = \alpha$ such that $\lambda_2 = \alpha - \lambda_1$ is not a free parameter any more. Additionally,
     we choose $\alpha$ such that $\max_\lambda L(\lambda)$ is attained.}
  \label{fig:surrogate:plot}
 \end{center}
\end{figure}

\paragraph{Contribution.}

In this paper, we revisit surrogate duality in the context of mixed-integer nonlinear programming. To the best of our knowledge,
surrogate relaxations have never been considered in practice for solving general \MINLPs.

The first contribution of the paper is an experimental study of a generalization of surrogate duality in the nonconvex
setting that allows for multiple aggregations of the nonlinear constraints. Second, based on a row-generation method, we present the first algorithm
to solve the corresponding generalized surrogate dual problem and prove its convergence.
Third, we present several computational enhancements to make the algorithm practical, which includes an effective way to integrate a \MINLP solver
into our algorithm. Our developed algorithm shows that the quality of the generalized surrogate relaxation can be significantly stronger than that of the
classic one. Finally, we provide a detailed computational analysis on publicly available benchmark instances.

\paragraph{Structure.}

The rest of the paper is organized as follows. In Section~\ref{section:background}, we present a literature review of
surrogate duality.
Section~\ref{section:simpledual} discusses an algorithm from the literature for solving the classic surrogate dual
problem and our new computational enhancements.
In Section~\ref{section:gendual}, we review a generalization of surrogate relaxations from the literature.
Afterwards, in Section~\ref{section:generalalgo}, we adapt an algorithm for the classic surrogate dual problem to the
general case and prove its convergence.
An exhaustive computational study using the \MINLP solver \scip on publicly available benchmark instances
is given in Section~\ref{section:experiments}.
Afterwards, Section~\ref{section:bandb} presents ideas for future work that exploits surrogate relaxations
in the tree search of spatial branch and bound.
Section~\ref{section:conclusion} presents concluding remarks.

\section{Background}
\label{section:background}

Surrogate constraints were first introduced by Glover~\cite{Glover1965} in the context of zero-one linear integer
programming problems. He defined the \emph{strength} of a surrogate constraint according to the dual bound achieved by
it ---the same notion we use in \eqref{eq:surrogate:dual} and throughout our work. He also showed how to obtain the best
multipliers for~\eqref{eq:surrogate:dual} in the case of two inequalities. Balas \cite{Balas1967} and Geoffrion
\cite{Geoffrion1967} extended the use of surrogate relaxations in zero-one linear programming. Their definitions
of strength of a surrogate relaxation, however, differed from that of Glover. Furthermore, their notions of strength
ignored integrality conditions. This allowed them to compute the best surrogate relaxation using a linear program. Later
on, Glover \cite{Glover1968} provided a unified view on the aforementioned approaches to surrogate relaxations
and proposed a generalization where only a subset of constraints are used for producing an aggregation, leaving the rest
explicitly enforced by the surrogate relaxation. We consider this variant via the set $\minlprelaxset$ in
\eqref{eq:surrogate:opt}.

A theoretical analysis of surrogate duality in a nonlinear setting was presented by Greenberg and
Pierskalla~\cite{Greenberg1970}. They showed that finding the best multipliers amounts to optimizing a quasi-concave and
in general discontinuous function and that the surrogate dual problem is at least as strong as the Lagrangian dual. They
also proposed a generalization using multiple disjoint aggregation constraints. A similar generalization allowing
multiple aggregations was later studied by Glover \cite{Glover1975} along with the \emph{composite dual}: a combination
of surrogate and Lagrangian relaxations. These generalizations were proposed without a computational evaluation.

Regarding the link between surrogate and Lagrangian duality, Karwan and Rardin~\cite{Karwan1979} presented necessary conditions
for having no gap between the Lagrangian and surrogate duals. They also gave empirical evidence on why having no such
gap is unlikely. As for the duality gap provided by the surrogate dual, and much like in Lagrangian duality, conditions
that ensure that the surrogate dual equals the optimal solution value (e.g., constraint qualification conditions) were
exhaustively studied, see \cite{Glover1975,Penot2004,Suzuki2011} and the references therein.

The first algorithmic method for finding the optimal value of~\eqref{eq:surrogate:dual} is attributed to
Banerjee~\cite{Banerjee1971}. In the context of integer linear programming, he proposed a Benders-type approach that
alternates between solving a linear program (the \emph{master} problem) and an integer linear program with a single
constraint (the \emph{sub-problem}). This approach is the one considered by us, which we describe in full detail in
Section \ref{section:simpledual} adapted to the \MINLP context, along with its convergence
guarantees. Karwan~\cite{Karwan1976} expanded on this approach, including a refinement of that of Banerjee and
subgradient-based methods. Independently, Dyer~\cite{Dyer1980} proposed similar methods to those of Karwan. Karwan and
Rardin~\cite{Karwan1979} argued in favor of Benders-based approaches for the search of multipliers, as opposed to
subgradient methods, by showing that a subgradient may not provide an ascent direction for the surrogate
dual. Nonetheless, a subgradient-like search procedure was proposed by Karwan and Rardin~\cite{Karwan1984} with
positive results in packing problems. The latter search method may also be viewed as a variant of the Benders approach
of Banerjee, with the LP master problem being replaced by a computationally more efficient multiplier update. Sarin et
al.~\cite{Sarin1987} then proposed a different multiplier search procedure based on consecutive Lagrangian dual searches
and tested it on randomly generated packing problems. Gavish and Pirkul~\cite{Gavish1985} proposed a heuristic
to find useful multipliers based on a sequential search over each multiplier separately while keeping the others
constant.  They presented computational experiments for their heuristic on packing instances as well. Kim and
Kim~\cite{Kim1998} built upon the approach by Sarin et al., and developed a more efficient exact algorithm for
finding the optimal multipliers. However, the guarantees of the latter hold only when the feasible set is finite.

From a different perspective, Karwan and Rardin~\cite{Karwan1981} described the interplay between the branch-and-bound
trees of an integer programming problem and its surrogate relaxations, to efficiently incorporate surrogate duals in
branch and bound. Later on, Sarin et al.~\cite{Sarin1988} showed how to integrate their Lagrangian-based multiplier
search proposed in~\cite{Sarin1987} into branch and bound.

From an application point of view, surrogate constraints were used in various ways. In~\cite{Glover1977}, Glover
presented a class of surrogate constraint primal heuristics for integer programming problems. Djerdjour et
al.~\cite{Djerdjour1988} presented a surrogate relaxation-based algorithm for knapsack problems with a quadratic
objective function. Fisher et al.~\cite{Fisher1983} used surrogate relaxations to construct algorithms that improve the
dual and primal bounds for the job shop problem. Narciso and Lorena~\cite{Narciso1999} used a surrogate
relaxation approach for tackling generalized assignment problems. We refer the reader to
\cite{Glover2003,Alidaee2014} for reviews on surrogate duality methods, including other applications and
alternative methods for generating surrogate constraints not based on aggregations.

To the best of our knowledge, the efforts for practical implementations of multiplier search methods have mainly focused
on \emph{linear} integer programs. This can be explained by the maturity of the computational optimization tools
available at the time most of these implementations were developed. We are only aware of two exceptions. First, the
entropy approach to nonlinear programming (see \cite{Templeman1987,Xingsi1991}) which uses a single
aggregation-based constraint to tackle nonlinear problems, but uses an entropy-based reformulation instead of a
weighted sum of the constraints. And second, the work by Nakagawa~\cite{Nakagawa2003} who considered
\emph{separable} nonlinear integer programming and presented a novel algorithm for solving the surrogate dual. However,
the author's approach is tailored for a limited family of nonlinear problems. Additionally, the algorithm relies
on performing, at each step, a potentially expensive enumerative procedure.

Regarding the generalization of the surrogate dual which considers multiple aggregated constraints (discussed in detail
in Section \ref{section:gendual}), we are not aware of any work considering a multiplier search method with provable
guarantees or a computational implementation of a heuristic approach for it. We are only aware of the discussion by
Karwan and Rardin~\cite{Karwan1980} regarding the searchability of multipliers for the surrogate dual generalizations
proposed by Greenberg and Pierskalla~\cite{Greenberg1970} and Glover \cite{Glover1975}. They argued that the lack of
desirable structures (such as quasi-concavity) may impair search procedures which are directly based on the original
surrogate dual. They showed, however, that simple heuristics can perform empirically well; although only for the
composite dual. The target of Section \ref{section:generalalgo} is to show that a Benders approach for the case of
multiple aggregations can also be used. Moreover, we prove that such an approach has similar convergence guarantees to
those of the single-aggregation surrogate dual.

\section{Surrogate duality in \MINLPs}
\label{section:simpledual}

While surrogate duality in its broader definition can be applied in theory to any \MINLP, to the best of our
knowledge, only mixed-integer \emph{linear} programming problems have been considered for practical applications. Much less
attention has been given to the general \MINLP case, due to the potential nonconvexity of the resulting problems.
Figure~\ref{fig:surrogate:example} illustrates the possible drawbacks and benefits of a nonconvex surrogate
relaxation, namely, potentially tight relaxations and potentially convex (Figure~\ref{fig:surrogate:example:3}),
nonconvex (Figure~\ref{fig:surrogate:example:1} and~\ref{fig:surrogate:example:4}) and even disconnected
(Figure~\ref{fig:surrogate:example:2}) feasible regions.

\begin{figure}[t]
 \begin{center}
  \begin{minipage}[t]{0.23\textwidth}
   \centering
   \includegraphics[width=1\textwidth]{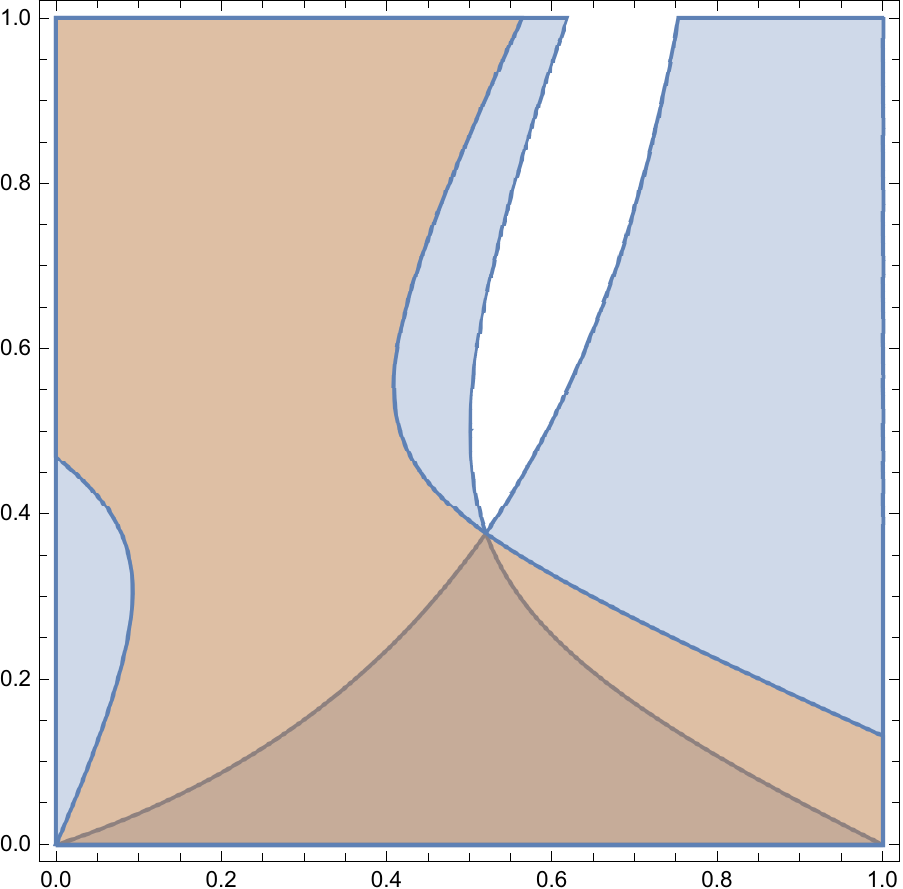}
   $\lambda = (\frac{3}{4},\frac{1}{4})$
   \subcaption{$\sset_{\lambda}$ nonconvex}
   \label{fig:surrogate:example:1}
  \end{minipage}
  \hspace{0.5ex}
  \begin{minipage}[t]{0.23\textwidth}
   \centering
   \includegraphics[width=1\textwidth]{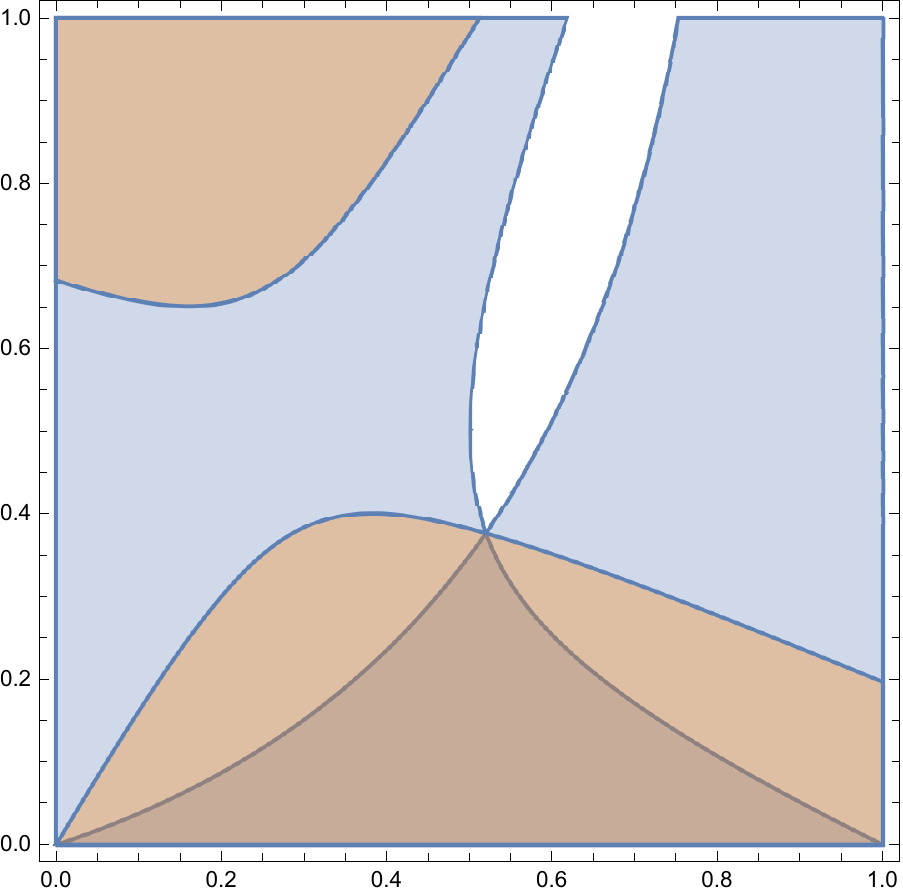}
   $\lambda = (\frac{2}{3},\frac{1}{3})$
   \subcaption{$\sset_{\lambda}$ disconnected}
   \label{fig:surrogate:example:2}
  \end{minipage}
  \hspace{0.5ex}
  \begin{minipage}[t]{0.23\textwidth}
   \centering
   \includegraphics[width=1\textwidth]{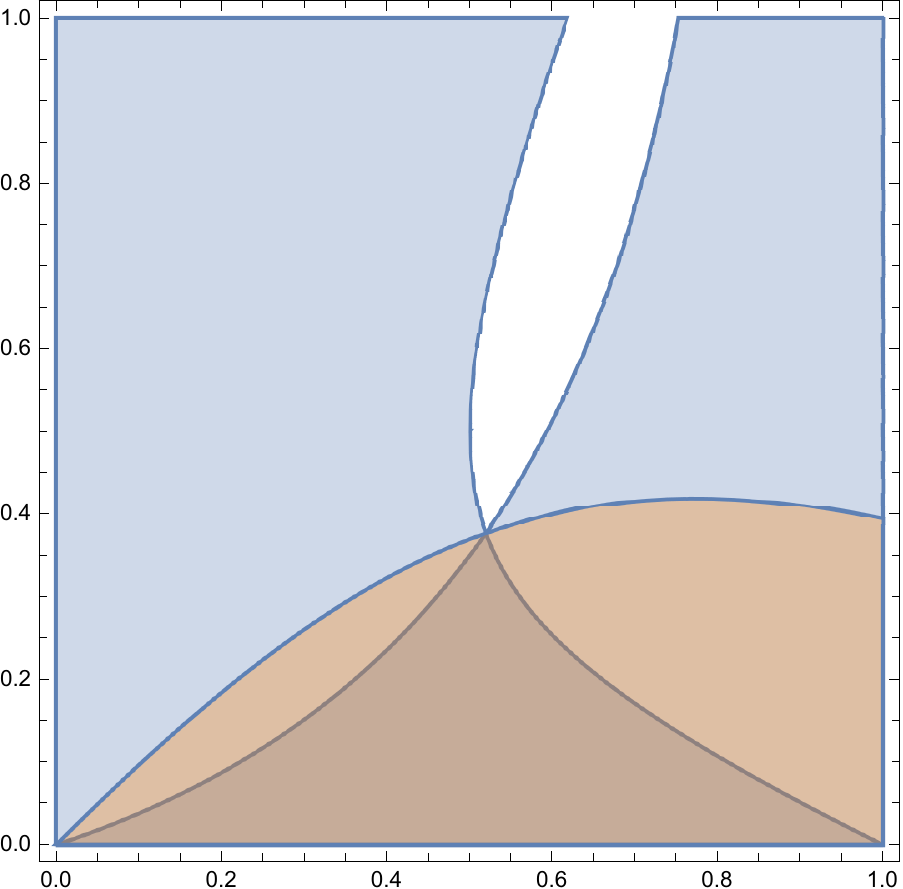}
   $\lambda = (\frac{1}{2},\frac{1}{2})$
   \subcaption{$\sset_{\lambda}$ convex}
   \label{fig:surrogate:example:3}
  \end{minipage}
  \hspace{0.5ex}
  \begin{minipage}[t]{0.23\textwidth}
   \centering
   \includegraphics[width=1\textwidth]{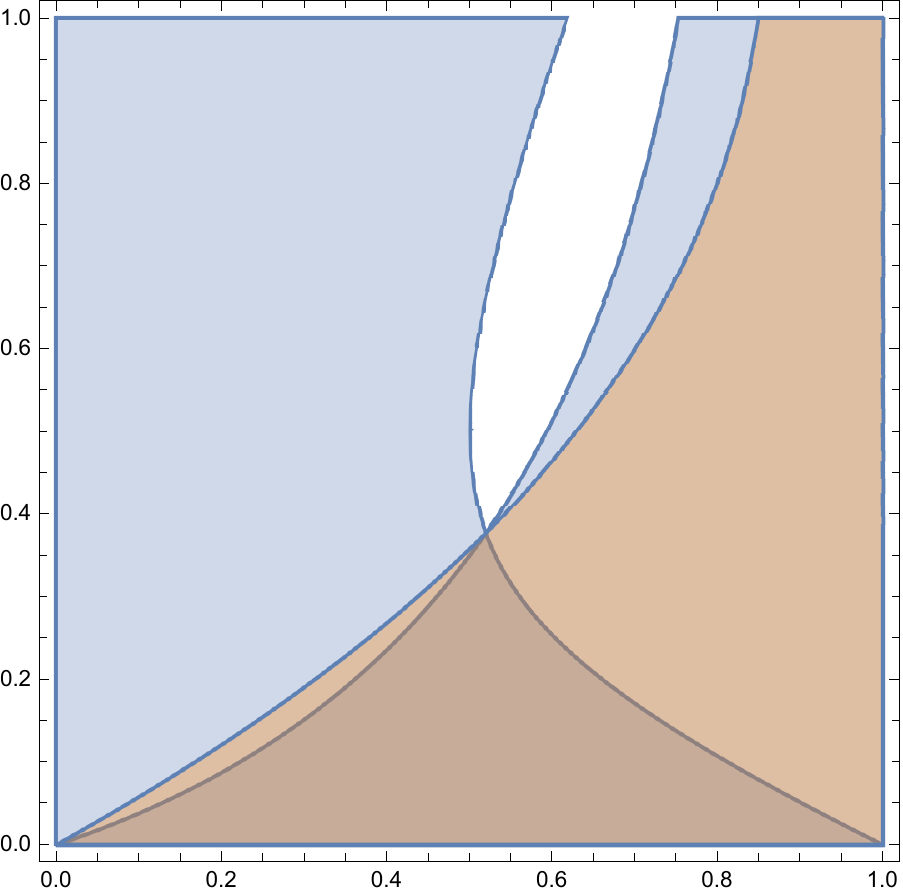}
   $\lambda = (\frac{1}{4},\frac{3}{4})$
   \subcaption{$\sset_{\lambda}$ complement of a convex set}
   \label{fig:surrogate:example:4}
  \end{minipage}
 \end{center}
 \caption{Different surrogate relaxations $\sset_{\lambda}$ for the nonconvex optimization problem in Example~\ref{example:xy}. The feasible region
  defined by each nonlinear constraint is blue and the set $\sset_{\lambda}$ is orange. Depending on the choice of $\lambda$, $\sset_{\lambda}$
  might be nonconvex, disconnected, convex, or reverse convex.}
 \label{fig:surrogate:example}
\end{figure}

We investigate the trade-off between the computational effort required to solve surrogate
relaxations and the quality of the resulting dual bounds. In this section, we show how to overcome the computational
difficulties faced when solving the surrogate dual with a Benders-type algorithm. This type of algorithm was presented
independently by Banerjee~\cite{Banerjee1971}, Karwan~\cite{Karwan1976}, and Dyer~\cite{Dyer1980}.

As we mentioned in Section \ref{section:background}, other algorithms for solving the surrogate dual exist, such as subgradient-based
algorithms~\cite{Karwan1976,Dyer1980,Sarin1987}. However, we use the Benders-type approach because its extension to the generalized surrogate dual problem (which we
discuss in Section~\ref{section:gendual}) is straightforward. It is unclear whether the subgradient-based algorithms
can be extended to work for the generalization, and if their convergence guarantees can be carried over.

\subsection{Solving the surrogate dual via Benders} \label{section:Benders}

In order to solve~\eqref{eq:surrogate:dual} (or at least find a good $\lambda$ multiplier),
we follow a known Benders-type algorithm, see~\cite{Karwan1976,Dyer1980}, which we
review here. The Benders algorithm is an iterative approach that alternates between solving a, so-called, master- and
sub-problem. The master problem searches for the next $\lambda$ aggregation and the sub-problem solves
$\sfct(\lambda)$. Note that the value of an optimal solution $\bar x$ of $\sfct(\lambda)$, i.e., $c^\T \bar x$, is a
valid dual bound for~\eqref{eq:minlp}. To ensure that the point $\bar x$ is not considered in later iterations, i.e.,
$\bar x \not\in \sset_\lambda$, the Benders algorithm uses the
master problem to compute a new vector $\lambda$ that ensures $\sum_{i \in \nlconssidx} \lambda_i g_i(\bar x) > 0$. This
can be done by maximizing constraint violation. More precisely, given $\pointset \subset \R^\nvars$ the set of
previously generated points of the sub-problems, the master problem reads as
\begin{equation}\label{eq:benders:master}
  \begin{aligned}
   \max \; & \Psi \\ \st & \sum_{i \in \nlconssidx} \lambda_i g_i(\bar x) \ge \Psi & & \fa \bar x \in \pointset, \\
   & \norm{\lambda}_1 \le 1, \\
   & \lambda \in \R^\nnlconss_+, \\
   & \Psi \in \R_+ .
  \end{aligned}
 \end{equation}
Due to the fact that each aggregation constraint is scaling invariant, it is necessary to add a
normalization, e.g., $\mnorm{\lambda}_1 \le 1$, to the master problem.
The resulting scheme, formalized in Algorithm~\ref{algo:benders}, terminates once the solution value
of~\eqref{eq:benders:master} is smaller than a fixed value $\epsilon > 0$. An illustration of the algorithm
for the nonconvex problem in Example~\ref{example:xy} is given in Figure~\ref{fig:benders}.

\begin{algorithm}[bt]
  \caption{Benders algorithm for the surrogate dual.}
  \label{algo:benders}
   \begin{algorithmic}[1]
   \REQUIRE{\MINLP of the form~\eqref{eq:minlp}, threshold $\epsilon > 0$}
   \ENSURE{optimal value $D \in \R$ of the surrogate dual}
   \STATE{initialize $\lambda \leftarrow 0 \in \R^{\nnlconss}_+$, $\Psi \leftarrow \infty$, $\pointset \leftarrow \emptyset$, $D \leftarrow -\infty$}
   \WHILE{$\Psi \ge \epsilon$}
     \STATE{$\bar x \leftarrow \argmin_x \{ c^\T x \mid x \in \sset_\lambda\}$ \label{step:subprobSol}}
     \STATE{$D \leftarrow \max \{ D,c^\T \bar x \}$ \label{step:dualBound}}
     \STATE{$\pointset \leftarrow \pointset \cup \{\bar x\}$}
     \STATE{$(\lambda,\Psi) \leftarrow $ optimal solution of~\eqref{eq:benders:master} for $\pointset$ \label{step:masterSol}}
   \ENDWHILE
   \RETURN{$D$}
   \end{algorithmic}
 \end{algorithm}

\begin{remark}
  Instead of finding an aggregation vector that maximizes the violation of all points in $\pointset$,
  Dyer~\cite{Dyer1980} uses an interior point for the polytope that is given by the so far found inequalities. This can
  be achieved by scaling $\Psi$ in each constraint of~\eqref{eq:benders:master} depending on the values $g_i(\bar x)$
  for each $i \in \nlconssidx$. In our experiments, however, we have observed that maximizing the violation significantly
  improved the quality of the computed dual bounds.
 \end{remark}

Although originally proposed for linear integer programming problems, Algorithm~\ref{algo:benders} can be attributed to
Banerjee~\cite{Banerjee1971}. Using his analysis, Karwan~\cite{Karwan1976} proved the following theorem for the
case of linear constraints.

\begin{theorem}\label{theorem:convergence}
 Denote by $\{(\lambda^t, \Psi^t)\}_{t \in \N}$ the sequence of values obtained in Step~\ref{step:masterSol} of Algorithm~\ref{algo:benders}
 for $\epsilon = 0$. If all $g_i$ are linear for all $i \in \nlconssidx$ then Algorithm~\ref{algo:benders} either
 \begin{itemize}
  \item terminates in $T$ steps, in which case $\max_{1 \le t \le T} \sfct(\lambda^t)$ is equal to~\eqref{eq:surrogate:dual}, or
  \item the sequence $\{\sfct(\lambda^t)\}_{t \in \N}$ has a sub-sequence converging to~\eqref{eq:surrogate:dual}.
 \end{itemize}
\end{theorem}

We prove a stronger version of this theorem in Section~\ref{section:generalalgo} that also works for nonlinear constraints.
Note that the convergence of the algorithm only relies on the solution of an LP and a nonconvex problem $S(\lambda)$, and
does not make any assumption on the nature of $S(\lambda)$.

\begin{figure}[t]
 \begin{center}
  \includegraphics[width=0.22\textwidth]{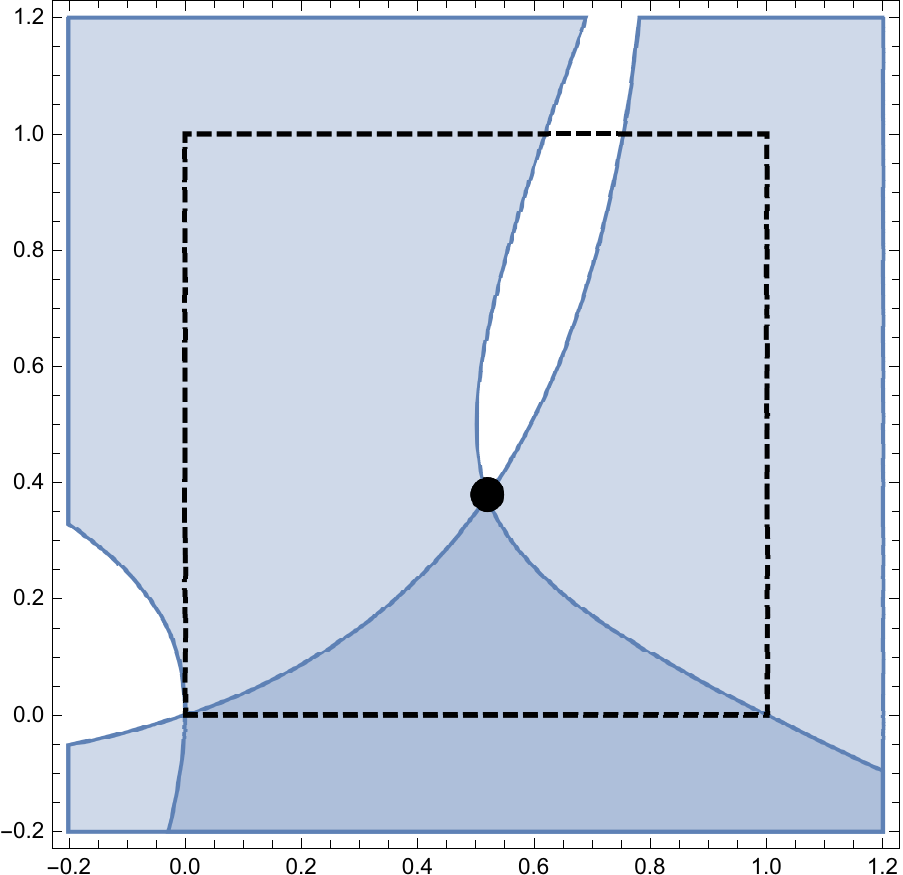}\hspace{2ex}
  \includegraphics[width=0.22\textwidth]{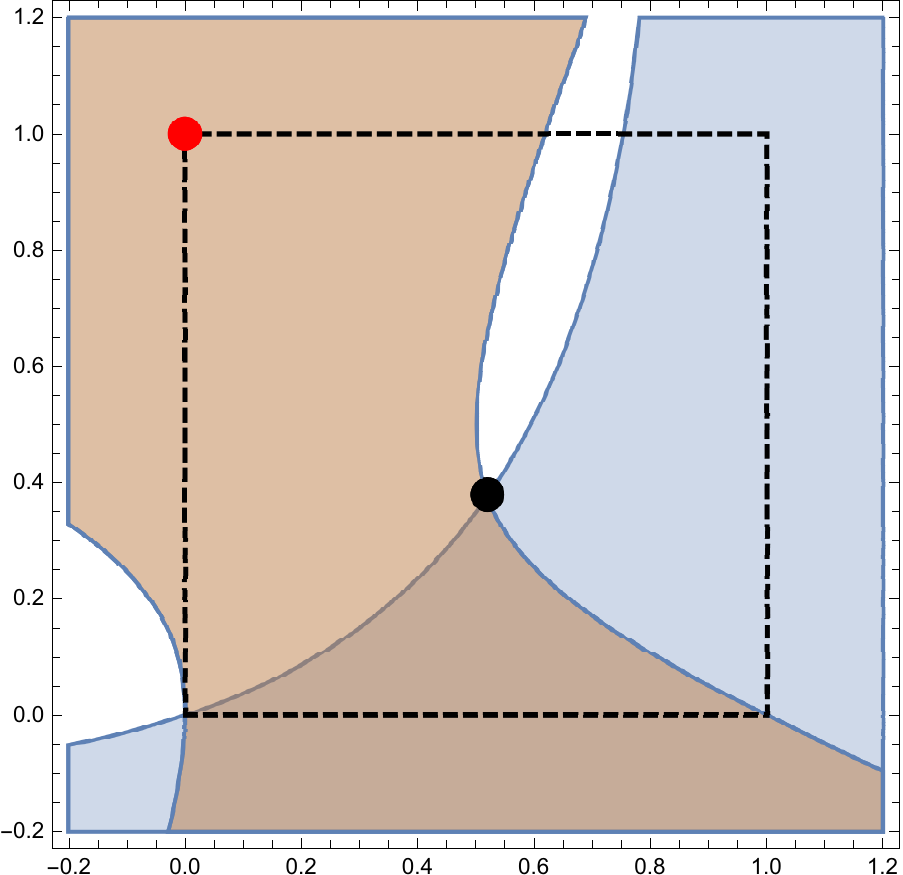}\hspace{2ex}
  \includegraphics[width=0.22\textwidth]{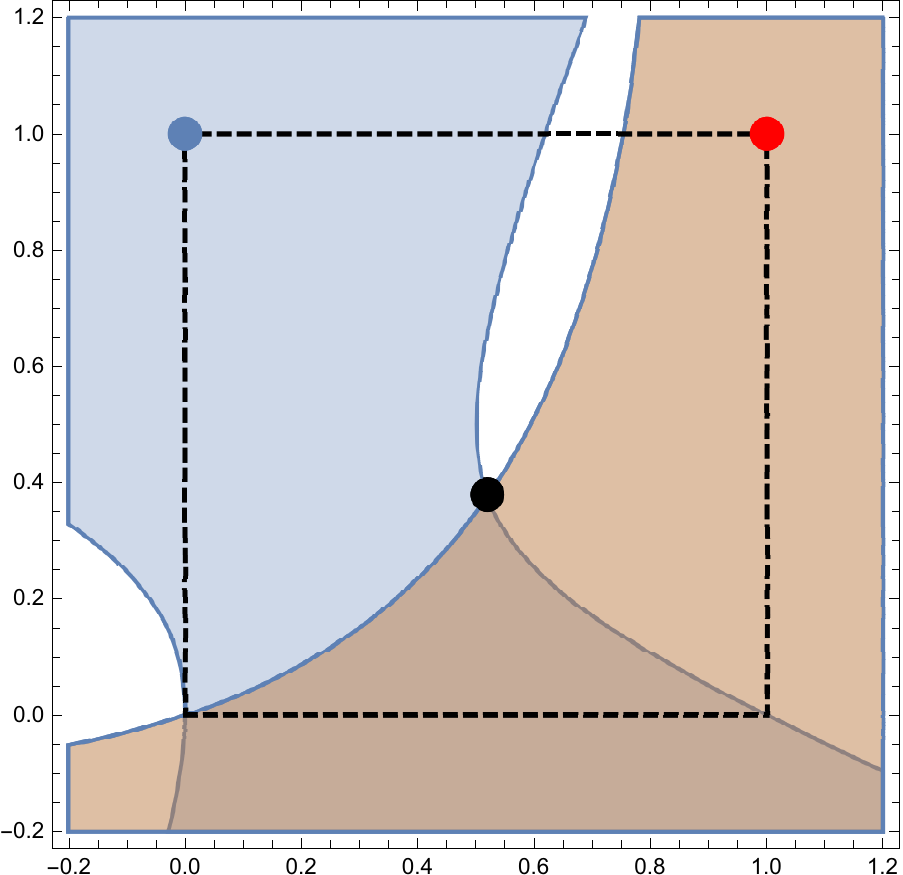}\hspace{2ex}
  \includegraphics[width=0.22\textwidth]{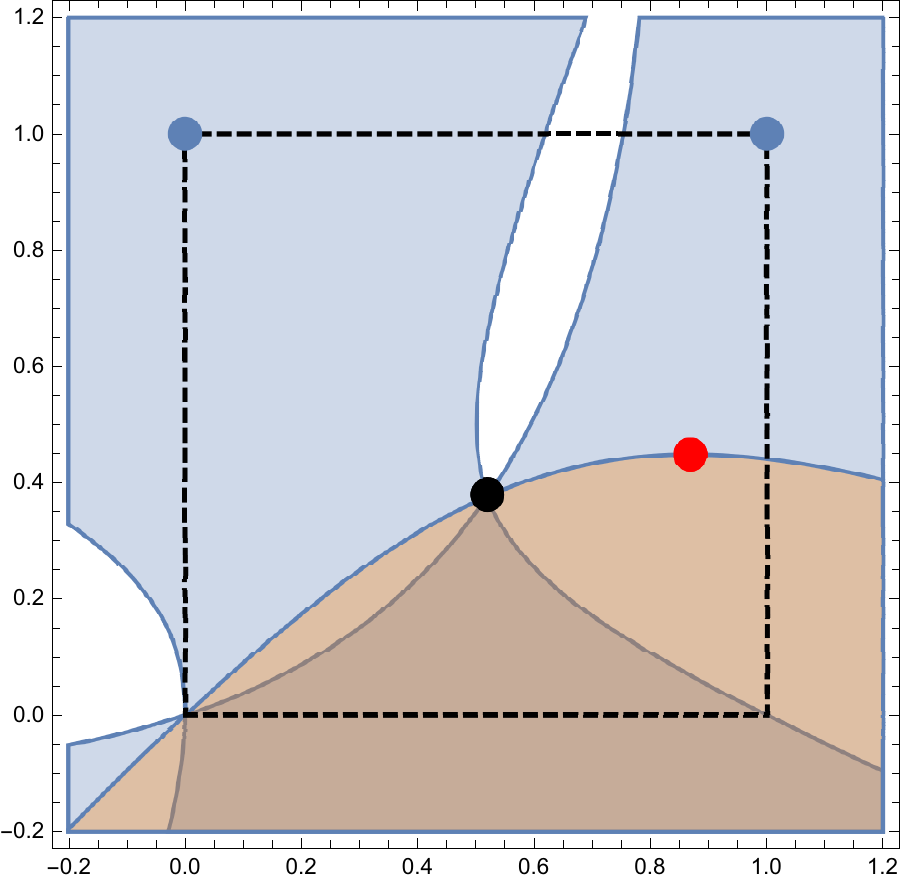}
 \end{center}

 \begin{center}
  \includegraphics[width=0.22\textwidth]{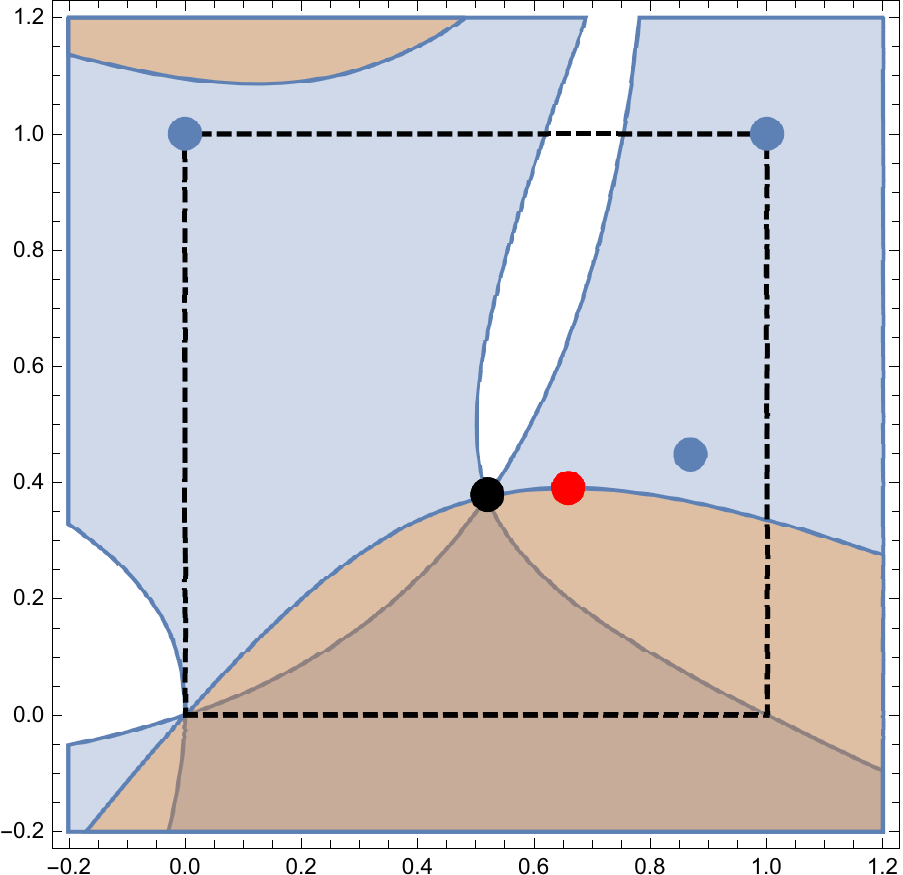}\hspace{2ex}
  \includegraphics[width=0.22\textwidth]{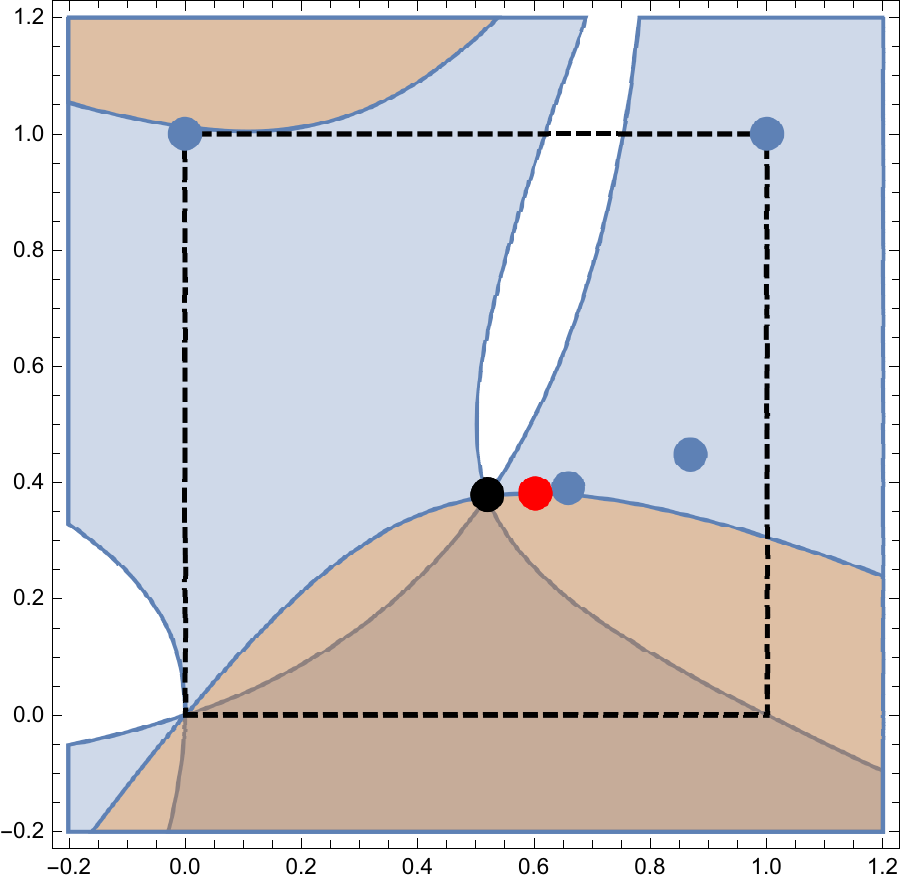}\hspace{2ex}
  \includegraphics[width=0.22\textwidth]{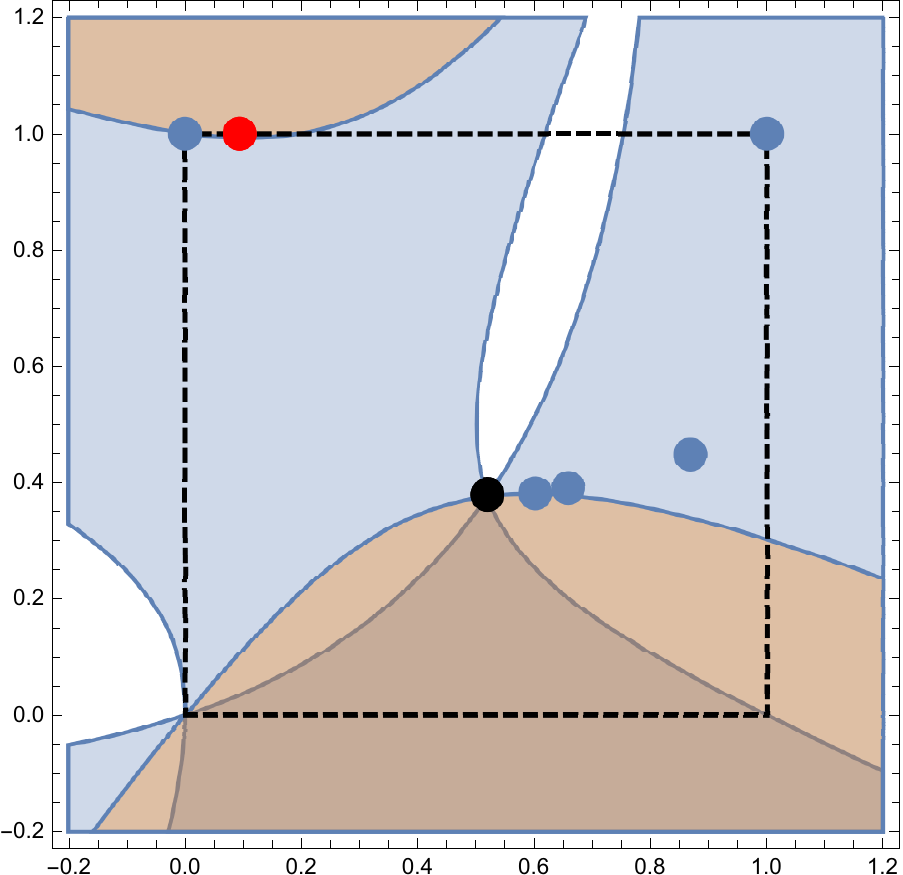}\hspace{2ex}
  \includegraphics[width=0.22\textwidth]{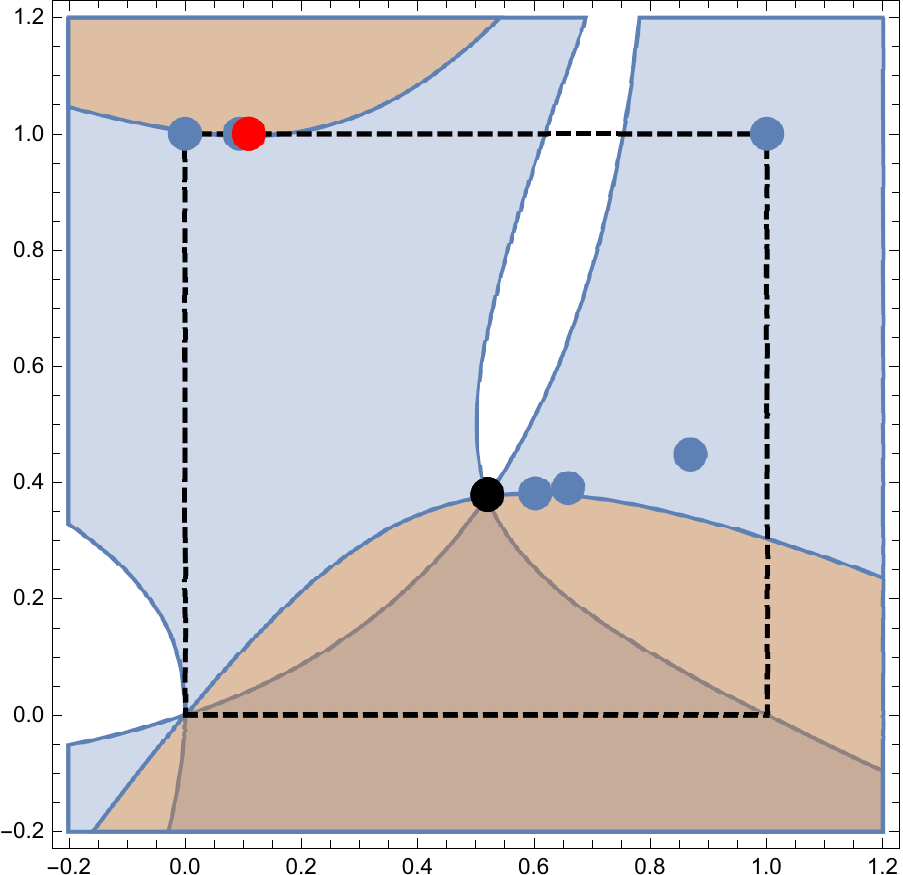}
 \end{center}

 \caption{An illustration of the Benders algorithm for the nonconvex optimization problem of Example~\ref{example:xy}. The black point
    is the optimal solution to the original problem, the dashed lines correspond to the variable bounds, the light blue regions are the
    feasible sets of each nonlinear constraint, and the orange region is $\sset_{\lambda}$ for the different values of
    $\lambda \in \R_+^2$ that are computed during the Benders algorithm. The red points are the optimal solutions at each iteration and any
	red point becomes blue in the next iteration as part of the set $\pointset$ of points to be separated. The algorithm converges after
	seven iteration (from left to write in both rows of the figure), whereas the best
    multiplier $\lambda^* \approx (0.56, 0.44)$ was found after five iterations.}
 \label{fig:benders}
\end{figure}

\subsection{Algorithmic enhancements}

In this section, we present computational enhancements that speed up Algorithm~\ref{algo:benders} and improve the
quality of the dual bound that can be achieved from~\eqref{eq:surrogate:dual}. For the sake of completeness, we also
include techniques that have been tested but did not improve the quality of the computed dual bounds significantly.

\subsubsection{Refined \MILP relaxation}\label{section:refinedmilp}

Instead of only using the initial linear constraints $Ax\le b$ of~\eqref{eq:minlp}, we exploit a linear programming
(\LP) relaxation of~\eqref{eq:minlp} that is available in \LP-based spatial branch and bound. This relaxation contains
$Ax \le b$ but also linear constraints that have been derived from, e.g., integrality restrictions of variables (e.g.,
MIR cuts~\cite{Nemhauser1990} and Gomory cuts~\cite{Gomory1960TR}), gradient cuts~\cite{Kelley1960}, \RLT
cuts~\cite{Sherali1999}, \SDP cuts~\cite{SheraliFraticelli2002}, or other valid underestimators for each $g_i$ with $i \in \nlconssidx$.
Using a linear relaxation $A'x \le b'$ with
\begin{equation}
  \minlprelaxset' := \{ x \in \minlprelaxset \mid A'x \le b' \}
\end{equation}
in the definition of $\sset_{\lambda}$ improves the value of~\eqref{eq:surrogate:dual} because a relaxed version of the
nonlinear constraint $g_i(x) \le 0$ is captured in $\sset_\lambda$ even if $\lambda_i$ is zero.

Another way to further strengthen the linear relaxation is to make use of objective cutoff information that is available
in spatial branch and bound. Suppose that there is a feasible, but not necessarily optimal, solution $x^*$
to~\eqref{eq:minlp}. Then, the linear relaxation can be strengthened by adding the inequality $c^\T x \le c^\T
x^*$. Adding this inequality preserves all optimal solutions of~\eqref{eq:minlp} and might
improve the optimal value of~\eqref{eq:surrogate:dual}.

In our experiments, we observed that utilizing the \LP relaxation that has been constructed in spatial branch and bound
is the most crucial ingredient to obtain strong dual bounds with surrogate relaxations, while the objective cutoff has
only a negligible impact on the quality of the computed dual bounds but helps in solving $\sfct(\lambda)$ faster.

\subsubsection{Dual objective cutoff in the sub-problem}

There is an undesired phenomenon present in Algorithm~\ref{algo:benders}: the sequence of dual bounds provided by
$c^\T x^*$ in Step~\ref{step:subprobSol} might not be monotone, i.e., the algorithm can spend several iterations
generating points that will not lead to an improvement in the dual bound $D$.

One way to overcome this problem is to add a \emph{dual objective cutoff} $c^\T x \geq D$ to the sub-problem $S(\lambda)$.
This enforces the sequence of dual bounds to be monotone. Adding such a constraint does not change the convergence/correctness
guarantees of Algorithm~\ref{algo:benders} and it can improve the
progress of the subsequent dual bounds. Moreover, such a cutoff can be used to filter
the set $\pointset$ and thus reduce the size of the \LP~\eqref{eq:benders:master}. Consider Figure~\ref{fig:benders} for
the effect of such a cutoff: the best dual bound is found at iteration five, meaning that the two last iterations
could be avoided. We also observed this behavior in other experiments, confirming the quality increase in the dual bounds
provided throughout the algorithm.

The dual objective cutoff has an unfortunate drawback. Adding a constraint that is parallel to the objective function
increases degeneracy. The degeneracy affects essential components of a branch-and-bound solver, e.g., pseudocost
branching~\cite{Benichou1971}, which typically makes the problem harder to solve. In the case of the Benders algorithm,
adding this cutoff significantly increases the time for solving the sub-problem, resulting in an overall negative
effect on the algorithm. We confirmed this with extensive computational experiments and decided not to include this
feature in our final implementation.

Fortunately, we can still carry dual information through different iterations and improve the performance of the
algorithm, without having to resort to a strict objective cutoff. We discuss this next.

\subsubsection{Early stopping in the sub-problem} \label{section:earlystop}

One important ingredient to speed up Algorithm~\ref{algo:benders}, proposed by Karwan~\cite{Karwan1976}
and Dyer~\cite{Dyer1980} independently, is an early stopping criterion while solving $S(\lambda)$. In our
setting, problem $S(\lambda)$ is the bottleneck of Algorithm~\ref{algo:benders}. This makes any technique
that can speed up the solving process of $S(\lambda)$ a crucial feature for Algorithm~\ref{algo:benders}.

Assume that Algorithm~\ref{algo:benders} proved a dual bound $D$ in some previous iteration. It is possible to stop the solving
process of $\sfct(\lambda)$ if a point $\bar x \in \sset_{\lambda}$ with $c^\T \bar x \le D$ has been found. The point
$\bar x$ both provides a new inequality for~\eqref{eq:benders:master} violated by $\lambda$ (as $\bar x \in
\sset_{\lambda}$) and shows $\sfct(\lambda) \le D$, i.e., $\lambda$ will not lead to a better dual bound. All
convergence and correctness statements regarding Theorem~\ref{theorem:convergence} remain valid after this
modification.

Furthermore, we can apply the same idea with any choice of $D$. In this scenario, $D$ would
act as a \emph{target} dual bound that we want to prove.
Due to the fact that the Benders-type algorithm is computationally expensive, one might require a minimum improvement
in the dual bound. Empirically, we observed that solving $\sfct(\lambda)$ to global optimality for difficult \MINLPs
requires a lot of time. However, finding a good quality solution for $\sfct(\lambda)$ is usually fast. This allows us
to early stop most of the sub-problems and only spend time on those sub-problems that will likely result in a dual
bound that is at least as good as the target value $D$.

In our computational study presented in Section~\ref{section:experiments}, we show that the early stopping technique is crucial
to prove significantly better dual bounds than the best known dual bounds in the literature on difficult \MINLPs.

\subsection{Empirical observations}

For the implementation of Algorithm~\ref{algo:benders}, we use the \MINLP solver \scip
\begin{enumerate}
\item to construct a linear relaxation $A'x \le b'$ for~\eqref{eq:minlp},
\item to find an objective cutoff $c^\T x \le c^\T x^*$, and
\item to use it as a black box to solve each $\sfct(\lambda)$ sub-problem.
\end{enumerate}
We provide more details of our implementation and the results in
Section~\ref{section:experiments}, but in order to provide an overall notion of the empirical impact of this algorithm
to the reader, we briefly summarize some important observations.

Our proposed algorithmic enhancements proved to be key for obtaining a practical algorithm for the
surrogate dual, especially the use of a refined \MILP relaxation. The achieved dual bounds by only using the initial
linear relaxation in Algorithm~\ref{algo:benders} were almost always dominated by the dual bounds obtained by the
refined \MILP relaxation. Thus, utilizing the refined \MILP relaxation seems mandatory for obtaining strong
surrogate relaxations.  Our computational study in Section~\ref{section:experiments} shows that our algorithmic
enhancements for Algorithm~\ref{algo:benders} allows us to compute dual bounds that close on average
\rootGapAffectedKOne\% more gap (w.r.t. the best known primal bound) than the dual bounds obtained by the refined
\MILP relaxations, i.e., $S(0)$, on \rootAffectedSize{} affected instances.

While the overall impact of this ``classic'' surrogate duality is positive, we observed that the dual bound
deteriorates with increasing number of nonlinear constraints. The reason is somewhat intuitive: aggregating a \emph{large}
number of nonconvex constraints into a single constraint may not capture the structure of the underlying \MINLP. For
this reason, we propose in the next Section to use generalized surrogate relaxations for solving \MINLPs, which include
multiple aggregation constraints. Even though the discussed relaxations are in general more difficult to solve,
they can provide significantly better dual bounds.

\section{Generalized surrogate duality}
\label{section:gendual}

In the following, we discuss a generalization of surrogate relaxations that has been introduced by~\cite{Glover1975}.
Instead of a single aggregation, it allows for $K \in \N$ aggregations of the nonlinear constraints of~\eqref{eq:minlp}.
The nonnegative vector
\begin{equation}
 \lambda = (\lambda^1, \lambda^2, \ldots, \lambda^K) \in \R^{K \nnlconss}_+
\end{equation}
encodes these $K$ aggregations
\begin{equation}
 \sum_{i \in \nlconssidx} \lambda^k_i g_i(x) \le 0, \quad k \in \{1,\ldots,K\}
\end{equation}
of the nonlinear constraints. Similar to~$\sset_\lambda$, for a vector $\lambda \in \R^{K \nnlconss}_+$ the feasible
region of the \emph{$K$-surrogate relaxation} is given by the intersection
\begin{equation}
 \gsset{K}_\lambda := \bigcap_{k=1}^{K} \sset_{\lambda^k} ,
\end{equation}
where $\sset_{\lambda^k}$ is the feasible region of the surrogate relaxation~$\sfct(\lambda^k)$ for $\lambda^k \in
\R^\nnlconss_+$. It clearly follows that $\gsset{K}_\lambda$ is a relaxation for~\eqref{eq:minlp}.
The best dual bound for~\eqref{eq:minlp} generated by a $K$-surrogate relaxation is given by
\begin{equation}\label{eq:gsd}
 \sup_{\lambda \in \R^{K \nnlconss}_+} \gsfct{K}(\lambda) ,
\end{equation}
which we call the \emph{$K$-surrogate dual}. Note that scaling each $\lambda^k \in \R^\nnlconss_+$ individually by a
positive scalar does not affect the value of $\gsfct{K}(\lambda)$, i.e.,
\begin{equation*}
  \gsfct{K}(\ldots, \lambda^k, \ldots) = \gsfct{K}(\ldots, \alpha \lambda^k, \ldots)
\end{equation*}
for any $\alpha > 0$. Therefore, it is possible to impose additional normalization constraints $\norm{\lambda^k}_1 \le 1$
for each $k \in \{1, \ldots, K\}$.

In~\cite{Greenberg1970}, a related generalization was proposed, although not computationally tested. The
paper considers a partition of constraints which are aggregated; equivalently, the support of sub-vectors $\lambda^k$
are assumed to be fixed and disjoint. Glover's generalization \cite{Glover1975} does not make any assumption on the structure of the
$\lambda^k$ sub-vectors. As we will see, this makes a significant difference for two reasons: (a) selecting
the ``best'' partition of constraints \emph{a-priori} is a challenging task and (b) restricting the support of sub-vectors
$\lambda^k$ to be disjoint can weaken the bound given by~\eqref{eq:gsd}. The reason is that the optimal
$\lambda$ multipliers might have to use the same constraints in multiple aggregations.

The function $\gsfct{K}$ remains lower semi-continuous for any choice of $K$. The idea of the proof
of the following proposition is similar to the one given by Glover~\cite{Glover1975} for the case of $K=1$.

\begin{proposition} \label{proposition:lsc}
  If $g_i$ is continuous for every $i \in \nlconssidx$ and $\minlprelaxset$ is compact then
  $\gsfct{K}: \R^{K \nnlconss}_+ \rightarrow \R$ is lower semi-continuous for any choice of $K$.
\end{proposition}

\begin{proof}
  Let $\{\lambda_t\}_{t \in \N} \subseteq \R^{K \nnlconss}_+$ a sequence that converges to $\lambda^*$ and denote with
  $x^t \in \minlprelaxset$ an optimal solution of $\gsfct{K}(\lambda^t)$.
  We need to show that $\gsfct{K}(\lambda^*) \leq \liminf_{t \to \infty} \gsfct{K}(\lambda^t)$.
  By definition, there exists a subsequence $\{\lambda^\tau\}_{\tau \in \N}$ of $\{\lambda^t\}_{t \in \N}$ such that
  $\gsfct{K}(\lambda^\tau) \to \liminf_{t \to \infty} \gsfct{K}(\lambda^t)$.
  Since $\minlprelaxset$ is compact, there exists a subsequence $\{x^l\}_{l \in \N}$ of
  $\{x^\tau\}_{\tau \in \N}$ such that $\lim_{l \rightarrow \infty} x^l = x^*$. 
  As $\{\lambda^l\}_{l \in \N}$ is a subsequence of $\{\lambda^\tau\}_{\tau \in \N}$, we have that
  $\lim_{l \rightarrow \infty} \gsfct{K}(\lambda^l) = \liminf_{t \to \infty} \gsfct{K}(\lambda^t)$. From
  $x^l \in \gsset{K}_{\lambda^l}$ it follows that
  \begin{equation*}
    \sum_{i \in \nlconssidx} \lambda_{k \nnlconss + i}^l \, g_i(x^l) \le 0
  \end{equation*}
  for every $k \in \{1, \ldots, K\}$, which is equivalent to
  \begin{equation*}
    \max_{1 \le k \le K} \sum_{i \in \nlconssidx} \lambda_{k \nnlconss + i}^l \, g_i(x^l) \le 0 \, .
  \end{equation*}
  Because the $g_i$ are continuous and the maximum of continuous functions is still continuous, it follows that
  \begin{equation*}
    \max_{1 \le k \le K} \sum_{i \in \nlconssidx} \lambda_{k \nnlconss
    + i}^* \, g_i(x^*) = \lim_{l \rightarrow \infty} \max_{1 \le k \le K}
    \sum_{i \in \nlconssidx} \lambda_{k \nnlconss + i}^l \, g_i(x^l) \le 0.
  \end{equation*}
  Hence, $x^*$ is feasible but not necessarily optimal for $\gsfct{K}(\lambda^*)$. Therefore,
  \begin{equation*}
    \gsfct{K}(\lambda^*) \le c^\T x^* = \lim_{l \rightarrow \infty} c^\T x^l = \lim_{l \rightarrow \infty} \gsfct{K}(\lambda^l) = \liminf_{t \rightarrow \infty} \gsfct{K}(\lambda^t) \, .
  \end{equation*}
\end{proof}

One important difference to the classic surrogate dual is that $\gsfct{K}(\lambda)$ is no longer quasi-concave. The
following example shows this even for the case of $K=2$ and two linear constraints.

\begin{example}\normalfont\label{example:counterexample}
 Let $K=2$ and consider the linear program
 \begin{equation*}
  \begin{aligned}
   \min \quad & y \\
   \st \quad  & 4x - 8y + 3.2 \le 0 , \\
              & 5x - y - 1.5 \le 0 , \\
              & (x,y) \in [0,1]^2 ,
  \end{aligned}
 \end{equation*}
 which contains two variables and two linear constraints. Due to the symmetry of the generalized surrogate dual,
 $\gsfct{2}(\lambda) = \gsset{2}(\mu) \approx 0.30$ holds for the aggregation vectors $\lambda := ((0.7,0.3),(0.3,0.7))$ and
 $\mu := ((0.3,0.7),(0.7,0.3))$. However, using the convex combination $\lambda / 2 + \mu / 2$ we have that
 $\gsfct{2}(\lambda / 2 + \mu / 2) \approx 0.19$, which is smaller than $\gsfct{2}(\lambda)$ and $\gsfct{2}(\mu)$ and
 thereby shows that $\gsfct{2}$ is not quasi-concave.
 See Figure~\ref{fig:counterexample} for an illustration of the counterexample.
\end{example}

\begin{figure}[t]
 \centering
 \begin{minipage}{0.30\textwidth}
  \centering
  \includegraphics[width=\textwidth]{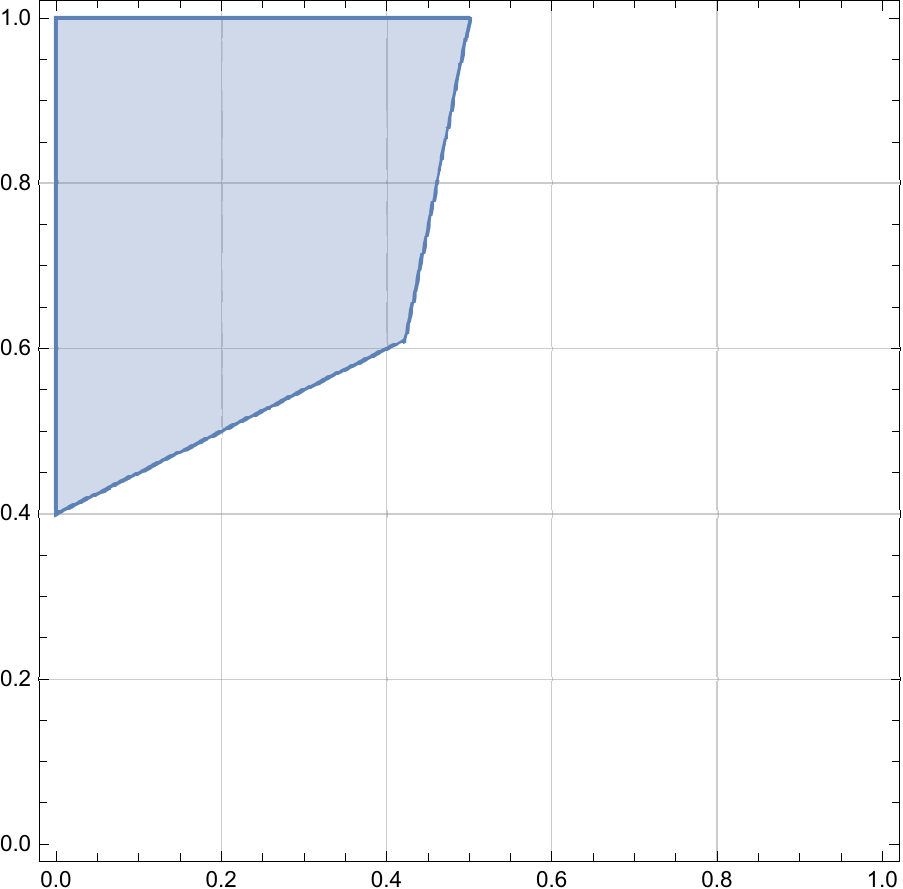}
 \end{minipage}
 \hfill
 \begin{minipage}{0.30\textwidth}
  \centering
  \includegraphics[width=\textwidth]{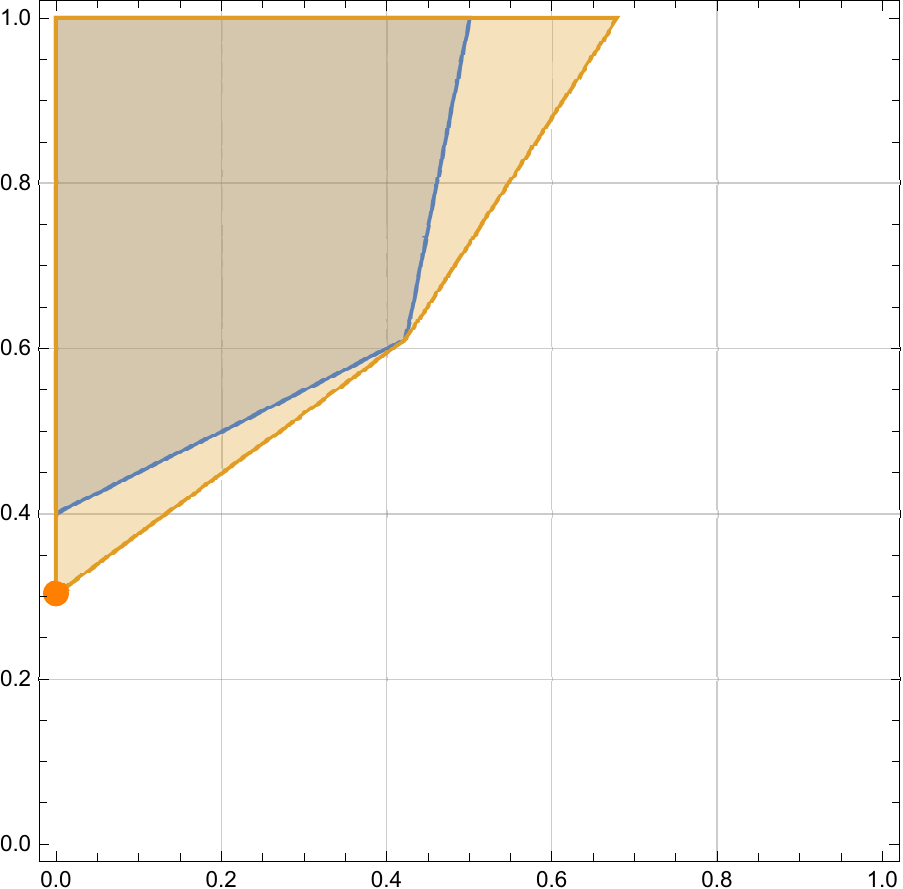}
 \end{minipage}
 \hfill
 \begin{minipage}{0.30\textwidth}
  \centering
  \includegraphics[width=\textwidth]{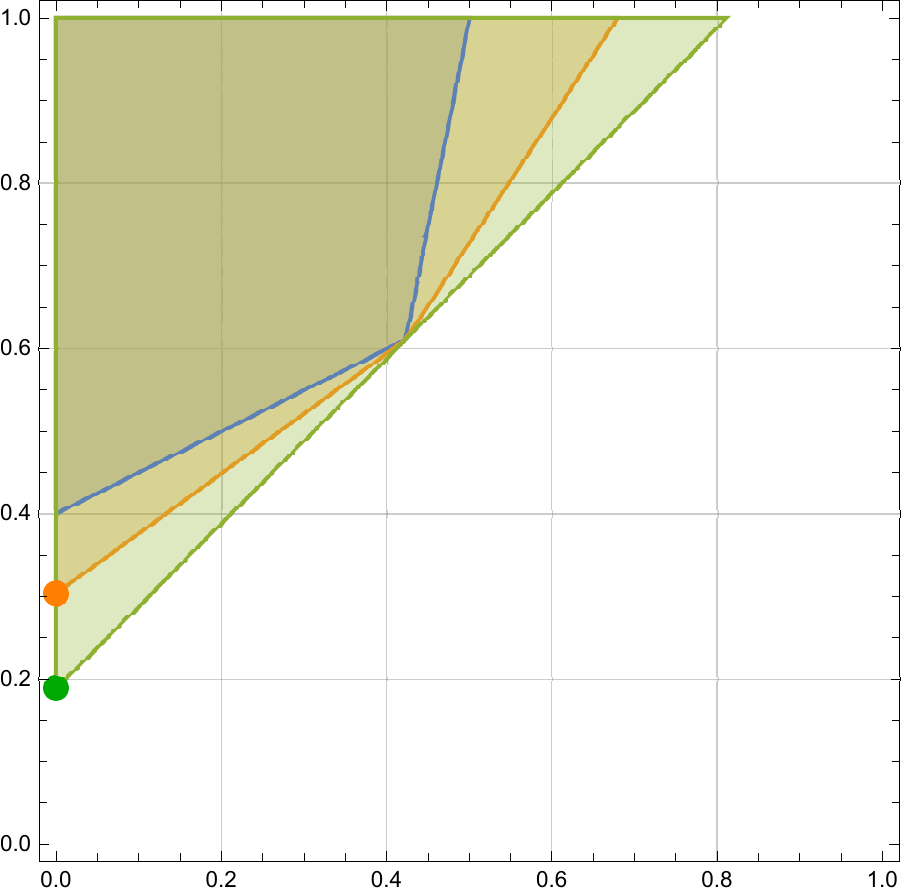}
 \end{minipage}
 \caption{A visualization of Example~\ref{example:counterexample} that shows that $\gsfct{K}$ is in general not
   quasi-concave for $K > 1$. The blue region is the feasible set defined by two original inequalities. The orange
   region depicts both $\gsset{2}_\lambda$ and $\gsset{2}_\mu$, while the green region is their convex combination
   $\gsset{2}_{(\lambda+\mu) / 2}$. The example shows $\gsfct{2}(\lambda) = \gsfct{2}(\mu) > \gsfct{2}((\lambda + \mu)/2)$,
   which proves that $\gsfct{2}$ is not quasi-concave.}
 \label{fig:counterexample}
\end{figure}

Due to the fact $\gsfct{K}$ is in general not quasi-concave, gradient descent-based algorithms for
optimizing~\eqref{eq:surrogate:dual}, as in~\cite{Karwan1976}, do not solve~\eqref{eq:gsd} to global
optimality.
Even though~\eqref{eq:gsd} is substantially more difficult to solve than~\eqref{eq:surrogate:dual}, the following
theorem shows that it might be beneficial to consider larger $K$ to obtain tight relaxations for~\eqref{eq:minlp}.

\begin{theorem}\label{theorem:gsd}
 The inequality
 \begin{equation}
  \sup_{\lambda \in \R^{K \nnlconss}_+} \gsfct{K}(\lambda) \le \sup_{\lambda \in \R^{(K+1) \nnlconss}_+} \gsfct{K+1}(\lambda)
 \end{equation}
 holds for any $K \in \N$. Furthermore, $\sup_{\lambda \in \R^{\nnlconss^2}_+} \gsfct{\nnlconss}(\lambda)$ is equal to the
 optimal solution value of~\eqref{eq:minlp}.
\end{theorem}

\begin{proof}
 Note that $\gsfct{K}(\lambda) = \gsfct{K+1}(\lambda,0)$ holds for any $\lambda \in \R^{K \nnlconss}_+$. The result follows from
 \begin{equation*}
  \sup_{\lambda \in \R^{K \nnlconss}_+} \gsfct{K+1}(\lambda,0) \le \sup_{\lambda \in \R^{(K+1) \nnlconss}_+} \gsfct{K+1}(\lambda) \ .
 \end{equation*}
 To prove the second part it is enough to see that the aggregation constraints for
 \begin{equation*}
  \lambda = (e_1, e_2, \ldots, e_\nnlconss) \in \R^{\nnlconss^2}_+,
 \end{equation*}
 with $e_i$ being the $i$-th $\nnlconss$-dimensional unit vector, are equal to the constraints of~\eqref{eq:minlp}.
\end{proof}

Theorem~\ref{theorem:gsd} shows the potential of generalized surrogate duality. Using a large enough $K$ implies that the
value of~\eqref{eq:gsd} is equal to the optimal value of the \MINLP. The following example shows that
going from $K=1$ to $K=2$ can have a tremendous impact on the quality of the surrogate relaxation:

\begin{example}\normalfont Consider the following \NLP with four nonlinear constraints and four unbounded variables:
 \begin{align*}
  \min \quad & -x - y \\
  \st \quad  & x^3 - z \le 0 \\
             & x^3 + z \le 0 \\
             & y^3 + w \le 0 \\
             & y^3 - w \le 0 \\
             & (x,y,z,w) \in \R^4
\end{align*}
 It is easy to see that $(0,0,0,0)$ is the optimal solution.
 First, note that the classic surrogate dual, i.e., when only a single aggregation is allowed, is unbounded. For an
 aggregation $\lambda \in \R^4$, the sole constraint in the corresponding surrogate relaxation is
 \[(\lambda_1 + \lambda_2)x^3 + (-\lambda_1 + \lambda_2)z + (\lambda_3 + \lambda_4)y^3 + (\lambda_3 - \lambda_4)w \leq
 0.\] If either $\lambda_1 \neq \lambda_2$ or $\lambda_3 \neq \lambda_4$, then the relaxation is clearly unbounded, as
 $z$ and $w$ are free variables. If $\lambda_1 = \lambda_2$ and $\lambda_3 = \lambda_4$, the aggregation reads
 $2\lambda_1 x^3 + 2\lambda_3 y^3 \le 0$, which also yields an unbounded surrogate relaxation.

 Consider the two aggregation vectors $\lambda = (\lambda^1,\lambda^2)$ with $\lambda^1 = (1/2,1/2,0,0)$ and $\lambda^2 =
 (0,0,1/2,1/2)$.  Using the $2$-surrogate relaxation obtained from $\lambda$ immediately implies tighter variable
 bounds $x \le 0$ and $y \le 0$, which proves optimality of $(0,0,0,0)$.
\end{example}

\section{An algorithm for the $K$-surrogate dual}
\label{section:generalalgo}

Even though~\eqref{eq:gsd} yields a strong relaxation for sufficiently large $K$, it is computationally
more challenging to solve than~\eqref{eq:surrogate:dual}. To the best of our knowledge, there is no algorithm in the literature known that can solve~\eqref{eq:gsd}.
Due to the missing quasi-concavity property of $\gsfct{K}$, it is not
possible to adjust each of the $K$ aggregation vectors independently and thus an alternating-type method based on the $K=1$ case could provide weak bounds.

In this section, we present the first algorithm for solving~\eqref{eq:gsd}. The idea of the algorithm is the
same as before: a master problem will generate an aggregation vector $(\lambda^1, \ldots, \lambda^K)$ and the sub-problem will solve the $K$-surrogate relaxation
corresponding to $(\lambda^1, \ldots, \lambda^K)$. The only differences to Algorithm~\ref{algo:benders} are that we replace the \LP master problem by a \MILP
master problem and solve $\gsfct{K}(\lambda^1,\ldots,\lambda^K)$ instead of $\sfct(\lambda)$.

\paragraph{Generalizing the Benders-type algorithm.}

Assume that we have found a solution $\bar x$ after solving $\gsfct{K}(\lambda^1, \ldots, \lambda^K)$.  In the next
iteration, we need to make sure that the point $\bar x$ is infeasible for at least one of the aggregated constraints.
This can be written as a disjunctive constraint
\begin{equation} \label{eq:gsd:logicor}
  \bigvee_{k=1}^K \left( \sum_{i \in \nlconssidx} \lambda^k_i g_i(\bar x) > 0 \right)
 \end{equation}
that contains $K$ many inequalities. As in~\eqref{eq:benders:master}, we replace the strict inequality by maximizing the
activity of $\sum_{i \in \nlconssidx} \lambda^k_i g_i(\bar x)$ for all $k \in \{1,\ldots,K\}$. The master problem for
the generalized Benders algorithm then reads as
\begin{equation}\label{eq:gbenders:master:disj}
  \begin{aligned}
   \max \quad & \Psi \\
   \st \quad & \bigvee_{k=1}^K \left( \sum_{i \in \nlconssidx} \lambda^k_i g_i(\bar x) \ge \Psi \right) && \fa \bar x \in \pointset  , \\
   & \norm{\lambda^k}_1 \le 1,\, \lambda^k \in \R_+^\nnlconss && \fa k \in \{1,\ldots,K\}  ,
  \end{aligned}
 \end{equation}
where $\pointset \subseteq \minlprelaxset$ is the set of generated points of the sub-problems.
One way to exactly solve~\eqref{eq:gbenders:master:disj} is to enumerate and solve all possible \LPs that are being
encoded by the disjunctions. Each \LP is constructed by choosing exactly one of the linear constraints of each
disjunction. However, following this approach is clearly prohibitively expensive because there are $K^{|\pointset|}$
many \LPs.

Instead,
we present an equivalent \MILP formulation that enables us to solve~\eqref{eq:gbenders:master:disj} more efficiently by
exploiting heuristics and symmetry breaking techniques that have been exclusively developed for \MILPs.

\paragraph{Solving the master problem.}

Modeling the master problem with a, so-called, big-M formulation solves orders of magnitudes faster. An
equivalent \MILP formulation of~\eqref{eq:gbenders:master:disj} reads as
\begin{equation}\label{eq:gbenders:master:milp}
  \begin{aligned}
   \max \quad & \Psi \\
   \st \quad  & \sum_{i \in \nlconssidx} \lambda^k_i g_i(\bar x) \ge \Psi - M (1 - z_k^{\bar x}) && \fa k \in \{1,\ldots,K\},\, \bar x \in \pointset  , \\
              & \sum_{k = 1}^{K} z_k^{\bar x} = 1 && \fa \bar x \in \pointset  , \\
              & z_k^{\bar x} \in \{0,1\} && \fa k \in \{1,\ldots,K\}, \, \bar x \in \pointset  , \\
              & \norm{\lambda^k}_1 \le 1,\, \lambda^k \in \R_+^\nnlconss && \fa k \in \{1,\ldots,K\}  , \\
  \end{aligned}
\end{equation}
where $M$ is a large constant. A binary variable $z_k^{\bar x}$ indicates if the $k$-th disjunction
of~\eqref{eq:gbenders:master:disj} is used to cut off the point $\bar x \in \pointset$. Due to the normalization
$\norm{\lambda^k}_1 \le 1$, it is possible to bound $M$ by $\max_{i \in \nlconssidx} |g_i(\bar x)|$. Even more,
since the optimal $\Psi$ values of~\eqref{eq:gbenders:master:milp} are non-increasing, we could use the optimal
$\Psi_{prev}$ of the previous iteration as a bound on $M$. Thus, it is possible to bound $M$ by $\min\{\max_{i \in
\nlconssidx} |g_i(\bar x)|, \Psi_{prev}\}$.

\begin{remark}
 Big-M formulations are typically not considered strong in \MILPs, given their usual weak \LP
 relaxations. Other formulations in extended spaces can yield better theoretical guarantees when solving problems
 like~\eqref{eq:gbenders:master:milp}, see, e.g.,~\cite{Balas1998},~\cite{Vielma2018b}, and~\cite{Bonami2015}.
 The drawback of these extended formulations is that they
 require to add copies of the
 $\lambda$ variables depending on the number of disjunctions. In~\cite{Vielma2018}, the author proposes an
 alternative that does not create variable copies, but that can be costly to construct unless special structure is
 present. In our case, however, as we will discuss in Section~\ref{section:generalalgo:earlystop}, we do not require a tight
 \LP relaxation of~\eqref{eq:gbenders:master:disj} and thus we opted to use~\eqref{eq:gbenders:master:milp}.
\end{remark}

The whole algorithm for the $K$-surrogate dual problem is stated in Algorithm~\ref{algo:genbenders}. Even
though~\eqref{eq:gbenders:master:milp} is more difficult to solve than~\eqref{eq:benders:master}, the following
example shows that Algorithm~\ref{algo:genbenders} can compute significantly better dual bounds than
Algorithm~\ref{algo:benders}.

\begin{algorithm}[bt]
  \caption{Benders algorithm for the $K$-surrogate dual.}
  \label{algo:genbenders}
   \begin{algorithmic}[1]
   \REQUIRE{\MINLP of the form~\eqref{eq:minlp}, threshold $\epsilon > 0$, $K \in \N$ aggregations}
   \ENSURE{optimal value $D \in \R$ of the $K$-surrogate dual}
   \STATE{initialize $\lambda \leftarrow 0 \in \R^{K \nnlconss}_+$, $\Psi \leftarrow \infty$, $\pointset \leftarrow \emptyset$, $D \leftarrow -\infty$}
   \WHILE{$\Psi \ge \epsilon$}
     \STATE{$\bar x \leftarrow \argmin_x \{ c^\T x \mid x \in \gsset{K}_\lambda\}$}
     \STATE{$D \leftarrow \max \{ D,c^\T \bar x \}$}
     \STATE{$\pointset \leftarrow \pointset \cup \{\bar x\}$}
     \STATE{$(\lambda,\Psi) \leftarrow $ optimal solution of~\eqref{eq:gbenders:master:disj} for $\pointset$}
   \ENDWHILE
   \RETURN{$D$}
   \end{algorithmic}
 \end{algorithm}

\begin{example}\label{ex:gbenders:example}\normalfont We briefly discuss the results of Algorithm~\ref{algo:genbenders} for the instance
 \texttt{genpooling\_lee1} from the \minlplib. The instance consists of 20 nonlinear, 59 linear constraints, 9 binary, and 40 continuous
 variables after preprocessing. The classic surrogate dual, i.e., $K=1$, could be solved to optimality, whereas for
 $K=2$ and $K=3$ the algorithm hit the iteration limit. Nevertheless, the dual bound $-5064.2$ achieved for $K=2$ and
 the dual bound $-4973.2$ for $K=3$ are significantly better than the dual bound of $-5246.0$ for $K=1$, see
 Figure~\ref{fig:genpooling_lee1}.
\end{example}

\begin{figure}[t]
 \centering
 \begin{minipage}{0.45\textwidth}
  \centering
  \includegraphics[height=3.5cm]{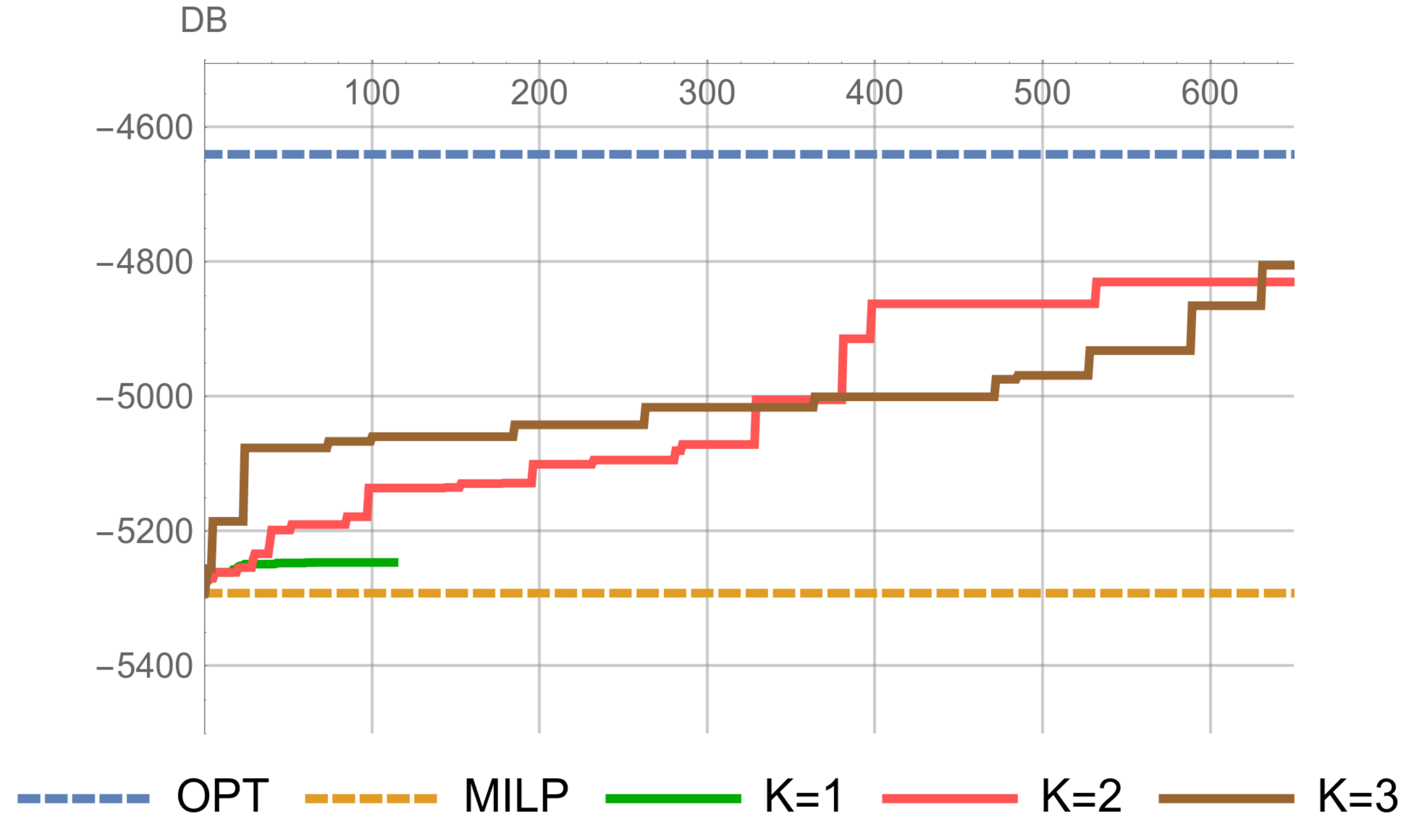}
 \end{minipage}
 \begin{minipage}{0.45\textwidth}
  \centering
  \includegraphics[height=3.5cm]{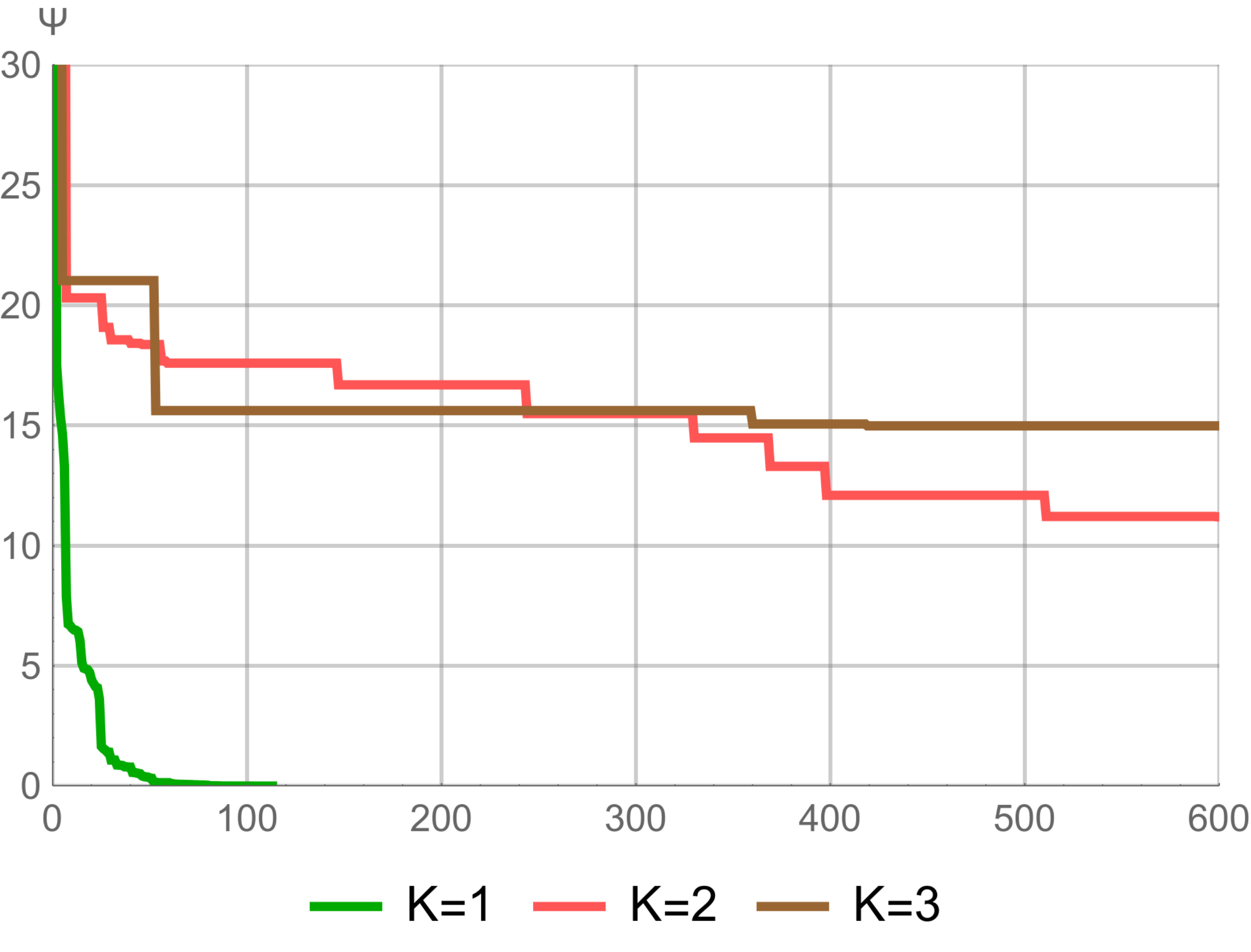}
 \end{minipage}
 \caption{Results of Algorithm~\ref{algo:genbenders} for the instance \texttt{genpooling\_lee1} for different choices
  of $K$. The first plot shows the progress of the proved dual bound and the second plot the value of $\Psi$ for the first $600$
  iterations. The blue line is the optimal solution value of the \MINLP and the yellow line that of the \MILP
  relaxation.}
 \label{fig:genpooling_lee1}
\end{figure}

\subsection{Convergence}

In the following, we show that the dual bounds obtained by Algorithm~\ref{algo:genbenders} converges to the optimal
value of the $K$-surrogate dual. The idea of the proof is similar to the one presented by~\cite{Karwan1984} for
the case of $K=1$ and linear constraints.

\begin{theorem} \label{theorem:generalalgo}
  Denote by $\{(\lambda^t, \Psi^t)\}_{t \in \N}$ the sequence of values obtained after
  solving~\eqref{eq:gbenders:master:disj} in Algorithm~\ref{algo:genbenders} for $\epsilon = 0$. The algorithm either
  \begin{itemize}
   \item[\textbf{(a)}] terminates in $T$ steps, i.e., $\Psi^T = 0$, in which case $\max_{1 \le t \le T}
     \gsfct{K}(\lambda^t)$ is equal to~\eqref{eq:gsd}, or
   \item[\textbf{(b)}] $\sup_{t \geq 1} \gsfct{K}(\lambda^t)$ is equal to~\eqref{eq:gsd}.
  \end{itemize}
\end{theorem}

\begin{proof}
  Let $OPT$ be the optimal value of~\eqref{eq:gsd} and let $x^t \in \minlprelaxset$ be an optimal solution obtained from solving
  $\gsfct{K}(\lambda^t)$ at iteration $t$.

  \begin{itemize}[listparindent=0.6cm]
    \item[\textbf{(a)}] If the algorithm terminates after $T$ iterations, i.e., $\Psi^T = 0$,
    then there is at least one point $x^1,\ldots,x^T$ that is feasible for $\gsfct{K}(\lambda)$ for any choice
    $\lambda \in \R^{K \nnlconss}$. This implies $OPT=\max_{1 \le t \le T} \{ c^\T x^t \}$.
    \item[\textbf{(b)}] Now assume that the algorithm does not converge in
      a finite number of steps, i.e., $\Psi^t > 0$ for all $t \ge 1$. Then,
    there are converging subsequences
    \begin{itemize}
      \item $\{\Psi^l\}_{l \in \N} \subseteq \{ \Psi^t \}_{t \in \N}$ such that $\lim_{l \rightarrow \infty} \Psi^l = \Psi^* \ge 0$ because $\Psi^t \ge \Psi^{t+1}$ and $\Psi^t \ge 0$ hold for all $t$,
      \item $\{ \lambda^l \}_{l \in \N} \subseteq \{ \lambda^t \}_{t \in \N}$ such that $\lim_{l \rightarrow \infty} \lambda^l = \lambda^*$ because $\norm{\lambda^t}_1 \le 1$, and
      \item $\{ x^l \}_{l \in \N} \subseteq \{ x^t \}_{t \in \N}$ such that $\lim_{l \rightarrow \infty} x^l = x^*$ because $\{x^t\} \subseteq \minlprelaxset$, which is assumed to be compact.
    \end{itemize}

    First, we show $\Psi^* = 0$. Note that $x^l$ is an optimal solution to $\gsfct{K}(\lambda^l)$. This means that $x^l$ satisfies all
    aggregation constraints, i.e., $\sum_{i \in \nlconssidx} \lambda_{k\nnlconss + i}^l \, g_i(x^l) \le 0$ for all $k = 1, \ldots, K$,
    which is equivalent to the inequality $\max_{1 \le k \le K} \sum_{i \in \nlconssidx} \lambda_{k\nnlconss + i}^l \, g_i(x^l) \le 0$. After
    solving~\eqref{eq:gbenders:master:disj}, we know that $\Psi^l$ is equal to the minimum violation of the disjunction constraints for
    the points $x^1,\ldots,x^{l-1}$. This implies the inequality
    \begin{equation*}
      \Psi^l = \min_{1 \le t \le l-1} \max_{1 \le k \le K} \sum_{i \in \nlconssidx} \lambda_{k\nnlconss + i}^l \, g_i(x^t) \le \max_{1 \le k \le K} \sum_{i \in \nlconssidx} \lambda_{k\nnlconss + i}^l \, g_i(x^{l-1})  ,
    \end{equation*}
    which uses the fact that the minimum over all points $x^1,\ldots,x^{l-1}$ is bounded by the value for $x^{l-1}$. Both
    inequalities combined show that
    \begin{equation*}
      \max_{1 \le k \le K} \sum_{i \in \nlconssidx} \lambda_{k\nnlconss + i}^l \, g_i(x^l) \le 0 < \psi^l \le \max_{1 \le k \le K} \sum_{i \in \nlconssidx} \lambda_{k\nnlconss + i}^l \, g_i(x^{l-1})
    \end{equation*}
    for all $l \ge 0$. Using the continuity of $g_i$ and the fact that the maximum of finitely many continuous functions is
    continuous, we obtain
    \begin{equation*}
      \max_{1 \le k \le K} \sum_{i \in \nlconssidx} \lambda_{k\nnlconss + i}^* \, g_i(x^*) \le 0 < \psi^* \le \max_{1 \le k \le K} \sum_{i \in \nlconssidx} \lambda_{k\nnlconss + i}^* \, g_i(x^*) \,
    \end{equation*}
    which shows $\Psi^* = 0$.

    Next, we show that $\sup_{t \geq 1} \gsfct{K}(\lambda^t) = OPT$.
    Clearly, $\sup_{t \geq 1} \gsfct{K}(\lambda^t) \leq OPT$.
    Let us now prove that $\sup_{t \geq 1} \gsfct{K}(\lambda^t) \geq OPT$.

    Take any $\epsilon > 0$ and let $\bar \lambda$ be such that $\gsfct{K}(\bar
    \lambda) \ge OPT - \epsilon$ and $\|\bar \lambda^k\| \leq 1$ for all $k \in
    \{1, \ldots, K\}$.
    By definition,
    \begin{equation*}
      \Psi^l \geq \min_{1 \le t \le l-1} \max_{1 \le k \le K} \sum_{i \in
      \nlconssidx} \bar \lambda_{k\nnlconss + i} \, g_i(x^t).
    \end{equation*}
    Computing the limit when $l$ goes to infinity, we obtain
    \begin{equation*}
      0 \geq \inf_{1 \le t} \max_{1 \le k \le K} \sum_{i \in
      \nlconssidx} \bar \lambda_{k\nnlconss + i} \, g_i(x^t).
    \end{equation*}
    Let $\bar x$ be $x^{t_0}$ if the infimum is achieved at $t_0$ or $x^*$ if
    the infimum is not achieved.
    Notice that
    \begin{equation*}
      \max_{1 \le k \le K} \sum_{i \in \nlconssidx} \bar
      \lambda_{k\nnlconss + i} \, g_i(\bar x) \leq 0.
    \end{equation*}
    This last inequality implies that $\bar x$ is feasible for $\gsfct{K}(\bar \lambda)$.
    Hence,
    \[
      OPT - \epsilon \leq \gsfct{K}(\bar \lambda) \leq c^\T \bar x \leq \sup_{t \geq 1} \gsfct{K}(\lambda^t).
    \]
    Since $\epsilon > 0$ is arbitrary, we conclude that $\sup_{t \geq 1} \gsfct{K}(\lambda^t) \geq OPT$.
  \end{itemize}
\end{proof}

The proof of Theorem~\ref{theorem:generalalgo} shows that $\{\Psi^t\}_{t \in \N}$ always converge to zero. A direct
consequence of this fact is that the Algorithm~\ref{algo:genbenders} converges in finite steps for any $\epsilon > 0$.

We now discuss computational enhancements meant for improving the performance of the proposed algorithm to solve the
$K$-surrogate dual. As in the case $K=1$, we also report techniques that we did not include in our final implementation.

\subsection{Multiplier symmetry breaking}
\label{section:generalalgo:symmetry}

One difficulty of optimizing the $K$-surrogate dual is that~\eqref{eq:gbenders:master:disj}
and~\eqref{eq:gbenders:master:milp} might contain many equivalent solutions.
For example, any permutation $\pi$ of the set $\{1,\ldots,K\}$ implies that the sub-problem $\gsfct{K}(\lambda)$ with
$\lambda = (\lambda^1,\ldots,\lambda^K)$ is equivalent to $\gsfct{K}(\lambda^{\pi})$ with $\lambda^\pi =
(\lambda^{\pi_1},\ldots,\lambda^{\pi_K})$. This symmetry slows down Algorithm~\ref{algo:genbenders}, as it heavily
impacts the solution time of the master problem. We refer to~\cite{Margot2009} for an overview of symmetry in integer
programming.

One way to overcome the problem of equivalent solutions is to explicitly break symmetry in the
$\lambda^1,\ldots,\lambda^K$ vectors. One way is to add the constraints
\begin{equation}\label{eq:master:lexorder}
 \lambda^1 \succeq_{\text{lex}} \lambda^2 \succeq_{\text{lex}} \ldots \succeq_{\text{lex}} \lambda^K
\end{equation}
that enforce a lexicographical order on $\lambda^1,\ldots,\lambda^K$ in~\eqref{eq:gbenders:master:milp}.

Enforcing a lexicographical order $\lambda^1 \succeq_{\text{lex}} \lambda^2$ on continuous vectors $\lambda^1$ and
$\lambda^2$ can be modeled using the following constraints
\begin{equation}\label{eq:master:lexorder:symcons1}
 \begin{aligned}
   & \lambda^1_1 \ge \lambda^2_1 \text{,} \\
   \left(\lambda^1_1 = \lambda^2_1\right) \;\Rightarrow\; & \lambda^1_2 \ge \lambda^2_2 \text{,} \\
   \left(\lambda^1_1 = \lambda^2_1 \land \lambda^1_2 = \lambda^2_2\right) \;\Rightarrow\; & \lambda^1_3 \ge \lambda^2_3 \text{,} \\
   \ldots
 \end{aligned}
\end{equation}
which can be reformulated linearly with additional binary variables and big-M constraints. However, we observed that
adding~\eqref{eq:master:lexorder:symcons1} increases the complexity of~\eqref{eq:gbenders:master:milp} so much that it
is not possible anymore to solve it in a reasonable amount of time.
For this reason, we use only simple linear inequalities to partially break symmetry in the master problem. We propose
two alternative ways. First, the constraints
\begin{equation}\label{eq:master:lexorder:symcons2}
 \lambda^1_1 \ge \lambda^2_1 \ge \ldots \ge \lambda^K_1
\end{equation}
enforce that $\lambda^1,\ldots,\lambda^K$ are sorted with respect to the first component, i.e., the first nonlinear
constraint. The drawback of this sorting is that if $\lambda^k_1 = 0$ for all $k \in \{1,\ldots,K\}$, i.e., if all
aggregations in a given iteration ignore the first constraint, then~\eqref{eq:master:lexorder:symcons2} does not break
any of the symmetry of~\eqref{eq:gbenders:master:milp}.

Our second idea for breaking symmetry is to use
\begin{equation}\label{eq:master:lexorder:symcons3}
 \begin{aligned}
  \lambda^1_1 & \ge \lambda^k_1 & & \fa k \in \{2,\ldots,K\}  , \\
  \lambda^2_2 & \ge \lambda^k_2 & & \fa k \in \{3,\ldots,K\}  , \\
  & \ldots \\
  \lambda^{K-1}_{K-1} & \ge \lambda^K_{K-1}  ,
 \end{aligned}
\end{equation}
which has a natural interpretation if the vectors $\lambda^1,\ldots,\lambda^K$ are written as columns of a matrix
$\Lambda \subseteq \R_+^{\nnlconss \times K}$. The constraints~\eqref{eq:master:lexorder:symcons3} enforce that the
diagonal entries $\Lambda_{k,k}$ are not smaller than $\Lambda_{k,k'}$ for any $k' > k$.

In our experiments, we used the Benders algorithm for $K \in \{1,2,3\}$. We observed that for these small choices of
$K$, slightly better dual bounds could be computed when using~\eqref{eq:master:lexorder:symcons2} instead
of~\eqref{eq:master:lexorder:symcons3}.  Furthermore, we also observed that both symmetry breaking inequalities had only an
impact on the obtained dual bounds if the first nonlinear constraint was used in the best found solution of the Benders
algorithm.

\subsection{Early stopping of the master problem}
\label{section:generalalgo:earlystop}

Solving~\eqref{eq:gbenders:master:milp} to optimality in every iteration of the Benders algorithm is computationally
expensive for $K\geq 2$. On the one hand, the true optimal value of $\Psi$ is needed to decide whether the algorithm
terminated, i.e., $\Psi \le \epsilon$. On the other hand, to ensure progress of the Benders algorithm it is enough to
only compute a feasible point $(\Psi,\lambda^1,\ldots,\lambda^K)$ of~\eqref{eq:gbenders:master:milp} with $\Psi > 0$. We
balance these two opposing forces with the following early stopping method.

Given that~\eqref{eq:gbenders:master:milp} is a \MILP, we use branch and bound to solve it. During the tree search of
this algorithm, we have access to both a valid dual bound $\Psi_{d}$ and primal bound $\Psi_{p}$ such that the optimal
$\Psi$ is contained in $[\Psi_p,\Psi_d]$.
Note that the primal bound can be assumed to be nonnegative as the vector of zeros is always feasible
for~\eqref{eq:gbenders:master:milp}.
Furthermore, let $\Psi^t_d$ and $\Psi^t_p$ be the primal and dual bounds obtained from the master problem in iteration
$t$ of the Benders algorithm. We stop the master problem in iteration $t+1$ as soon as $\Psi^{t+1}_p \ge \alpha
\Psi^t_d$ holds for a fixed $\alpha \in (0,1]$. The parameter $\alpha$ controls the trade-off between proving a good
  dual bound $\Psi^{t+1}_d$ and saving time for solving the master problem. On the one hand, $\alpha = 1$ implies
\begin{equation*}
 \Psi^{t+1}_p \ge \alpha \Psi^{t}_d \ge \alpha \Psi^{t+1}_d = \Psi^{t+1}_d  ,
\end{equation*}
which can only be true if $\Psi^{t+1}_p = \Psi^{t+1}_d$ holds. This equality proves optimality of the master problem in
iteration $t+1$. On the other hand, setting $\alpha$ close to zero means that we would stop as soon as a feasible
solution to the master problem has been found. In our experiments, we observed that setting $\alpha$ to $0.2$ performs
well.

\subsection{Constraint filtering}
\label{section:generalalgo:constraintfiltering}

Even though it is not necessary to solve the master problem in every iteration to global optimality, its complexity
grows exponentially since a disjunction constraint of the form~\eqref{eq:gsd:logicor} is added in every iteration of the
algorithm. One way to alleviate this problem is to reduce the set of nonlinear constraints to only those that are needed
for a good quality solution of~\eqref{eq:gsd}. This set of constraints is unknown in advance and challenging to compute
because of the nonconvexity of the \MINLP.

We tested different filtering heuristics to preselect nonlinear constraints. We used the violation of the constraints
with respect to the \LP, \MILP, and convex \NLP relaxation of the \MINLP, as measures of ``importance'' of nonlinear
constraints. We also used the connectivity of nonlinear constraints in the variable-constraint graph\footnote{Bipartite
graph where each variable and each constraint are represented as nodes, and edges are included when a variable appears
in a constraint.} for discarding some constraints. Unfortunately, we could not identify a good filtering rule that
selects few nonlinear constraints and results in strong bounds for~\eqref{eq:gsd}.

However, we developed a way of capturing the idea of reducing the number of constraints considered in the master problem
without having to impose such a strong \emph{a-priori} filter on the constraints: an adaptive filtering, which we call
\emph{support stabilization}. This allows to improve the performance of the master problem without compromising the
quality of the generated dual bounds. We specify this next.

\subsection{Support stabilization}
\label{section:generalalgo:supportstablilization}

Direct implementations of Benders-based algorithms, much like column generation approaches, are known to suffer from
convergence issues. Deriving ``stabilization'' techniques that can avoid oscillations of the $\lambda$ variables and
tailing-off effects, among others, are a common goal for improving performance, see,
e.g.,~\cite{Merle1999},~\cite{Amor2009}, and~\cite{Ackooij2016}.

In the following, we present a \emph{support stabilization} technique to address the exponential increase in complexity
of the master problem~\eqref{eq:gbenders:master:milp} and to prevent the oscillations of the $\lambda$ variables.
Since restricting the support on the aggregation vectors allows us to solve the master problem orders of magnitudes
faster, we use the following strategy: once the Benders algorithm finds a multiplier vector that improves the overall
dual bound, we restrict the support to that of the improving dual multiplier. This restricts the search space and
improves solution times. Once stalling is detected (which corresponds to finding a local optimum of~\eqref{eq:gsd}),
we remove the support restriction until another multiplier vector that improves the dual bound is found.

This technique enables us to solve the master problem substantially faster and, at the same time, compute better bounds
on~\eqref{eq:gsd} in fewer iterations due to its \emph{stabilization} interpretation.

\subsection{Trust-region stabilization}
\label{section:generalalgo:trustregion}

In the previous section, we presented a form of stabilization for our algorithm, meant for both alleviating some of the
computational burden when solving the master problem and preventing the support of subsequent variables to
deviate. Nonetheless, the non-zero entries of the $\lambda$ vectors can (and do, in practice) vary significantly from
iteration to iteration.
To remedy this, we incorporated a classic stabilization technique: a \emph{box trust-region}
stabilization, see~\cite{Conn2000}. Given a reference solution $(\hat{\lambda}^1, \ldots, \hat{\lambda}^k)$, we impose the
following constraint in~\eqref{eq:gbenders:master:milp}
\begin{equation*}
  \| (\lambda^1, \ldots, \lambda^k) - (\hat{\lambda}^1, \ldots, \hat{\lambda}^k) \|_\infty \leq \delta
\end{equation*}
for some parameter $\delta$. This prevents the $\lambda$ variables from oscillating excessively, and carefully updating
$(\hat{\lambda}^1, \ldots, \hat{\lambda}^k)$ and $\delta$ can maintain the convergence guarantees of the algorithm
proven in Theorem~\ref{theorem:generalalgo}. In our implementation, we maintain a fixed $(\hat{\lambda}^1, \ldots,
\hat{\lambda}^k)$ until we obtain a bound improvement or the algorithm stalls. When any of this happens, we remove the
box and compute a new $(\hat{\lambda}^1, \ldots, \hat{\lambda}^k)$ with~\eqref{eq:gbenders:master:milp} without any
stabilization added.

\begin{remark}
  In our experiments, we used another stabilization technique inspired by column generation's
  \emph{smoothing} by~\cite{Wentges1997} and~\cite{Neame2010}. Let $\lambda^{best}$ be the best found primal solution so far
  and let $\lambda^{new}$ be the solution of the current master problem. Instead of using $\lambda^{new}$ as a new
  multiplier vector, we choose as next aggregation vector a convex combination between $\lambda^{best}$ and
  $\lambda^{new}$.  This way we can control the distance between the new aggregation vector and $\lambda^{best}$. While
  this stabilization technique improved the performance of the Benders algorithm with respect to the algorithm with no
  stabilization, it performed significantly worse than the trust-region stabilization. Therefore, we did not include it
  in our final implementation.
\end{remark}

\section{Computational experiments}
\label{section:experiments}

In this section, we present a computational study of the classic and generalized surrogate duality on publicly
available instances of the \minlplib~\cite{MINLPLIB}.
We conduct three main experiments to answer the following questions:
\begin{enumerate}
  \item \experimentRoot{}: How much of root gap with respect to the \MILP relaxation can be closed by using the classic
    and $K$-surrogate dual? For how many instances could the classic and generalized Benders algorithm successfully
    terminate?
  \item \experimentAlgo{}: How much do the ideas of Section~\ref{section:generalalgo} improve the performance of the
    generalized Benders algorithm?
  \item \experimentDual{}: Can the generalized Benders algorithm improve on the dual bounds obtained by the \MINLP solver
    \scip?
\end{enumerate}

Our ideas are embedded in the \MINLP solver \scip~\cite{SCIP}. We refer
to~\cite{Achterberg2007a,Vigerske2013,Vigerske2017} for an overview of the general solving algorithm and \MINLP features
of \scip.

\subsection{Experimental setup}

All three experiments use Algorithm~\ref{algo:genbenders} to compute a tighter dual bound in the root node. As
discussed in Section~\ref{section:intro}, the quality of the surrogate relaxation strongly depends on the constructed
linear relaxation of~\eqref{eq:minlp}. Therefore, the Benders algorithm is called after the root node has been
completely processed by \scip. All generated and initial linear inequalities are added to $\sset_\lambda$.

For the \experimentRoot{} experiment, we run Algorithm~\ref{algo:genbenders} for one hour for each choice of $K \in
\{1,2,3\}$. To measure how much more root gap can be closed by using $K+1$ instead of $K$, we use the best found
aggregation vector of $K$ as an initial point for $K+1$. This ensures that Algorithm~\ref{algo:genbenders} always
finds a dual bound for $K+1$ that is at least as good as the one for $K$.

In contrast to the first experiment, in the \experimentAlgo{} experiment we focus on $K=3$ and do not start with an
initial point for the aggregation vector. Considering only one $K$ allows us to more easily analyze the impact of each
component of the Benders algorithm. We compare the following settings:
\begin{itemize}
  \item \algoDefault{}: Benders algorithm applying all techniques that have been presented in
    Section~\ref{section:generalalgo}.
  \item \algoPlain{}: Plain version of the Benders algorithm. It uses none of the techniques of
    Section~\ref{section:generalalgo}.
  \item \algoNoTR{}: Same as \algoDefault{} but without using the trust-region of
    Section~\ref{section:generalalgo:trustregion} and support stabilization
    of Section~\ref{section:generalalgo:supportstablilization}.
  \item \algoNoSupp{}: Same as \algoDefault{} but without using the support stabilization.
  \item \algoNoEarly{}: Same as \algoDefault{} but without using early termination for the master problem, described in
    Section~\ref{section:generalalgo:earlystop}.
\end{itemize}
Each of the five settings uses a time limit of one hour.

Finally, in the \experimentDual{} experiment we evaluate how much the dual bounds obtained by \scip with default
settings can be improved by the Algorithm~\ref{algo:genbenders}. First, we collect the dual bounds for all instances that
could not be solved by \scip within three hours. Afterward, we apply Algorithm~\ref{algo:genbenders} for $K=3$, a
time limit of three hours, and set a target dual bound (see Section~\ref{section:earlystop}) of
\begin{equation*}
  D + (P-D) \cdot 0.2,
\end{equation*}
where $D$ is the dual bound obtained by default \scip and $P$ be the best known primal bound reported in the
\minlplib. This means that we aim for a gap closed reduction of at least $20$\% and early stop each sub-problem in
Algorithm~\ref{algo:genbenders} that will provably lead to a smaller reduction.

During all three experiments, we use a gap limit of $10^{-4}$ for each sub-problem of the Benders algorithm to reduce
the impact of tailing-off effects. Additionally, we chose a dual feasibility tolerance of $10^{-8}$ (\scip's default is
$10^{-7}$) and a primal feasibility tolerance of $10^{-7}$ (\scip's default is $10^{-6}$).

\paragraph{Implementation.}

We extended \scip by a (relaxator) plug-in that solves the $K$-surrogate dual problem after the root node has been completely
processed by \scip, i.e., no more cutting planes or variable bound tightenings could be found.

The trust-region and support stabilization have been implemented as follows. Both stabilization methods are applied once
an improving aggregation $\lambda^*$ could be found. Each entry $\lambda_i$ with $\lambda_i^* = 0$ is fixed to
zero. Otherwise, the domain of $\lambda_i$ is restricted to the interval
\begin{equation*}
  \left[\max\{0,\lambda_i^* - 0.1\},\min\{1,\lambda_i^* + 0.1\}\right].
\end{equation*}
Once a new improving solution has been found, we update the trust region accordingly. We remove the trust region and
support stabilization in case no improving solution could be found for $20$ iterations.

\paragraph{Test set.}

We used the publicly available instances of the \minlplib~\cite{MINLPLIB}, which at time of the experiments
contained~$1683$ instances.
This includes among others instances from the first \minlplib, the nonlinear programming library \globallib, and the
CMU-IBM initiative \href{www.minlp.org}{minlp.org}~\cite{MINLPDOTORG}.
We selected the instances that were available in OSiL format and consisted of nonlinear expressions that could be
handled by \scip, in total $1671$~instances.

\paragraph{Gap closed.}

We use the following measure to compare dual bounds relative to a given primal bound. Let $d_1 \in \R$ and $d_2 \in \R$
be two dual bounds for~\eqref{eq:minlp} and $p \in \R$ a reference primal bound, e.g., the optimal solution value
of~\eqref{eq:minlp}, that is reported in the \minlplib. The function $\GC : \R^3 \rightarrow [-1,1]$ defined as
\begin{equation*}
  \GC(p,d_1,d_2) :=
  \begin{cases}
    0, & \text{ if } d_1 = d_2 \\ +1 - \frac{p - d_1}{p - d_2}, & \text{ if } d_1 > d_2 \\ -1 + \frac{p - d_2}{p - d_1},
    & \text{ if } d_1 < d_2
  \end{cases}
\end{equation*}
measures the \emph{gap closed} improvement.

\paragraph{Performance evaluation.}

To evaluate algorithmic performance over a large test set of benchmark instances, we compare geometric means, which
provide a measure for relative differences. This avoids results being dominated by outliers with large absolute values
as is the case for the arithmetic mean. In order to also avoid an over-representation of differences among very small
values, we use the shifted geometric mean.
The \textit{shifted geometric mean} of values $v_1, \ldots, v_N \ge 0$ with shift $s \ge 0$ is defined as
\begin{equation*}
 \left( \prod_{i=1}^N (v_i + s) \right)^{1/N} - s.
\end{equation*}
See also the discussion in~\cite{Achterberg2007a,AchterbergWunderling2013,Hendel2014}.
As shift values we use 10~seconds for averaging over running time and $5$\% for averaging over gap closed values.

\paragraph{Hardware and software.}

The experiments were performed on a cluster of 64bit Intel Xeon X5672~CPUs at 3.2\,GHz with 12\,MB cache and 48\,GB main
memory.
In order to safeguard against a potential mutual slowdown of parallel processes, we ran only one job per node at a time.
We used a development version of \scip with \cplex~12.8.0.0 as \LP~solver~\cite{Cplex}, the algorithmic differentiation
code \cppad~20180000.0~\cite{CppAD}, the graph automorphism package \bliss~0.73~\cite{bliss} for detecting \MILP symmetry,
and \ipopt~3.12.11 with \mumps~4.10.0~\cite{Mumps} as \NLP solver~\cite{WachterBiegler2006,Ipopt}.

\subsection{Computational results}

In the following, we present results for the above described \experimentRoot{}, \experimentAlgo{}, and \experimentDual{}
experiments.

\paragraph{\experimentRoot{} Experiment.}

From all instances of \minlplib, we filter those for which \scip's \MILP relaxation proves optimality in the root node,
no primal solution is known, or \scip aborted due to numerical issues in the \LP solver. This leaves $\rootAllSize$
instances for the \experimentRoot{} experiment.

Figure~\ref{fig:root:scatter} visualizes the achieved gap closed values via scatter plots. The plots show that for the
majority of the instances we can close significantly more gap than the \MILP relaxation. There are
$\rootGapWorseKOneKTwo$ instances for which $K=2$ closes at least $1$\% more gap than $K=1$, and even more gap can be
closed using $K=3$.
There are $\rootGapNoGapAllKOneKTwo$ instances for which $K=1$ could not close any gap, but $K=2$ could close some. On
$\rootGapNoGapAllKTwoKThree$ additional instances $K=3$ could close gap, which was not possible with $K=2$.
Finally, comparing $K=2$ and $K=3$ shows that on $\rootGapWorseKTwoKThree$ instances $K=3$ could close at least $1$\%
more gap than $K=2$. Interestingly, for most of these instances $K=2$ could already close at least $50\%$ of the root
gap.

Aggregated results are reported in Table~\ref{table:root:gapclosed} and we refer to Table~\ref{table:root:detailed} in
the appendix for detailed instance-wise results. First, we observe an average gap reduction of $\rootGapAllKOne$\% for
$K=1$, $\rootGapAllKTwo$\% for $K=2$, and $\rootGapAllKThree$\% for $K=3$, respectively. The same tendency is true when
considering groups of instances that are defined by a bound on the minimum number of nonlinear constraints. For example,
for the $\rootNconssTwentySize$ instances with at least $20$ nonlinear constraints after preprocessing, $K=2$ and $K=3$
close $\floatsub{\rootGapNconssTwentyKTwo}{\rootGapNconssTwentyKOne}$\% and
$\floatsub{\rootGapNconssTwentyKThree}{\rootGapNconssTwentyKOne}$\% more gap than $K=1$,
respectively. Table~\ref{table:root:gapclosed} also reports results when filtering out the $\rootGapNotaffected$
instances for which less than $1$\% gap was closed by Algorithm~\ref{algo:genbenders}. We consider these instances
\emph{unaffected}. On the $\rootAffectedSize$ affected instances we close on average up to $\rootGapAffectedKThree$\% of
the gap, and we see that $K=3$ closes $\floatsub{\rootGapAffectedKTwo}{\rootGapAffectedKOne}$\% more gap than $K=2$ and
$\floatsub{\rootGapAffectedKThree}{\rootGapAffectedKOne}$\% more than $K=1$.

Our results show that using surrogate relaxations has a tremendous impact on reducing the root gap. Additionally, we
observe that using the generalized surrogate dual for $K=2$ and $K=3$ reduces significantly more gap in the root node
than the classic surrogate dual.

\begin{figure}[tb]
 \centering
 \begin{minipage}{0.40\textwidth}
  \centering \includegraphics[width=\textwidth]{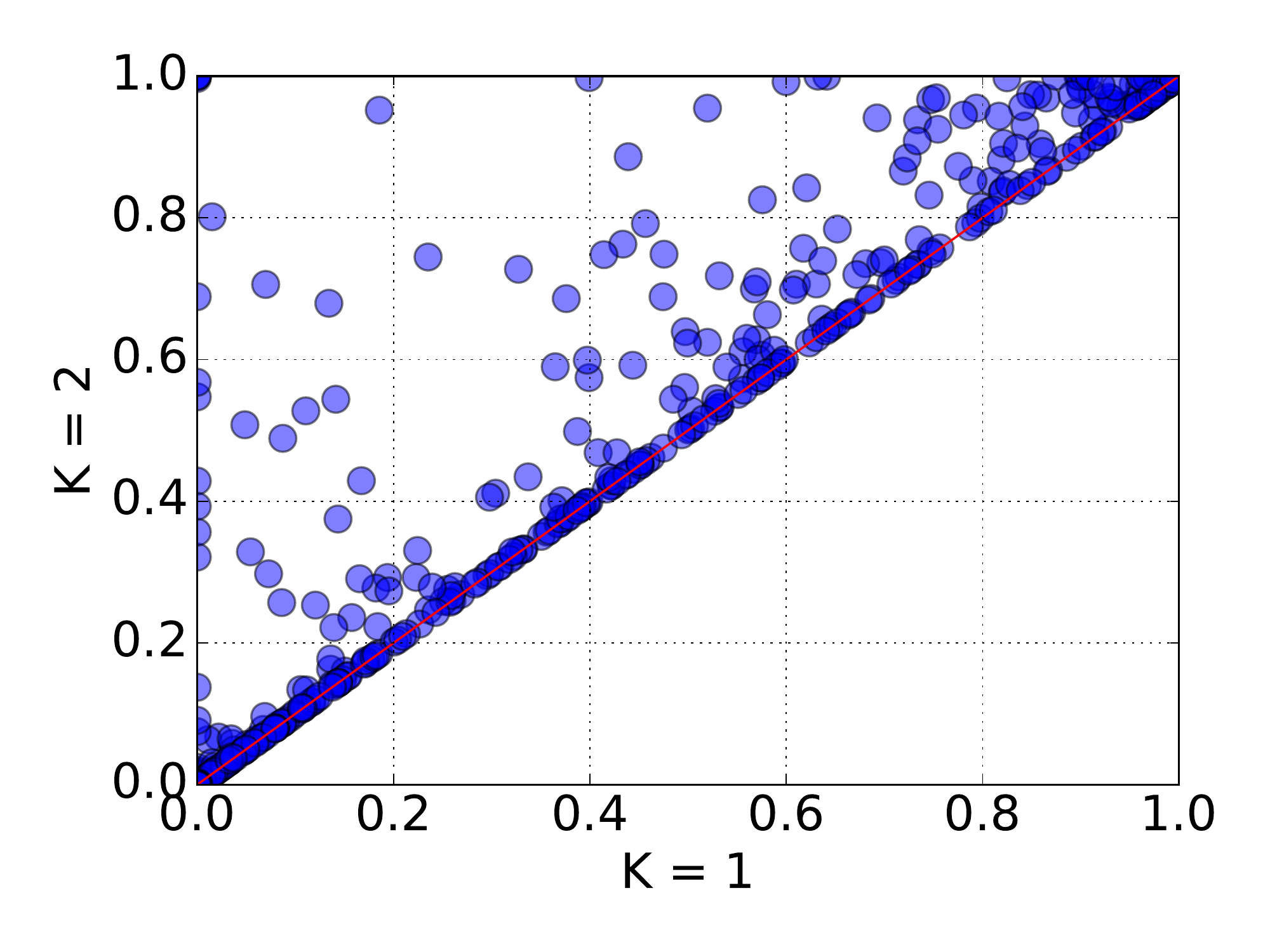}
 \end{minipage}
 \begin{minipage}{0.40\textwidth}
  \centering \includegraphics[width=\textwidth]{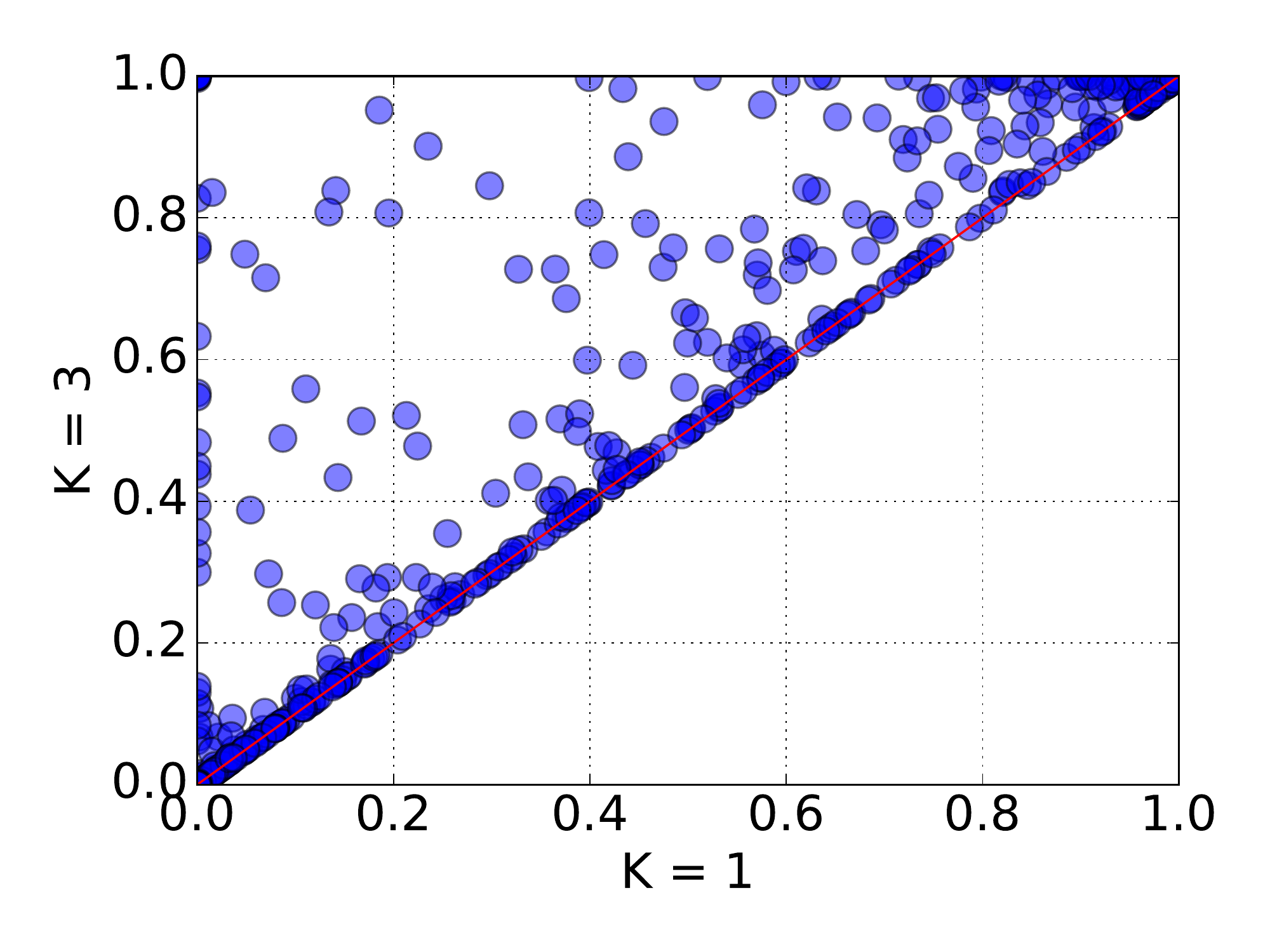}
 \end{minipage}

 \begin{minipage}{0.40\textwidth}
  \centering \includegraphics[width=\textwidth]{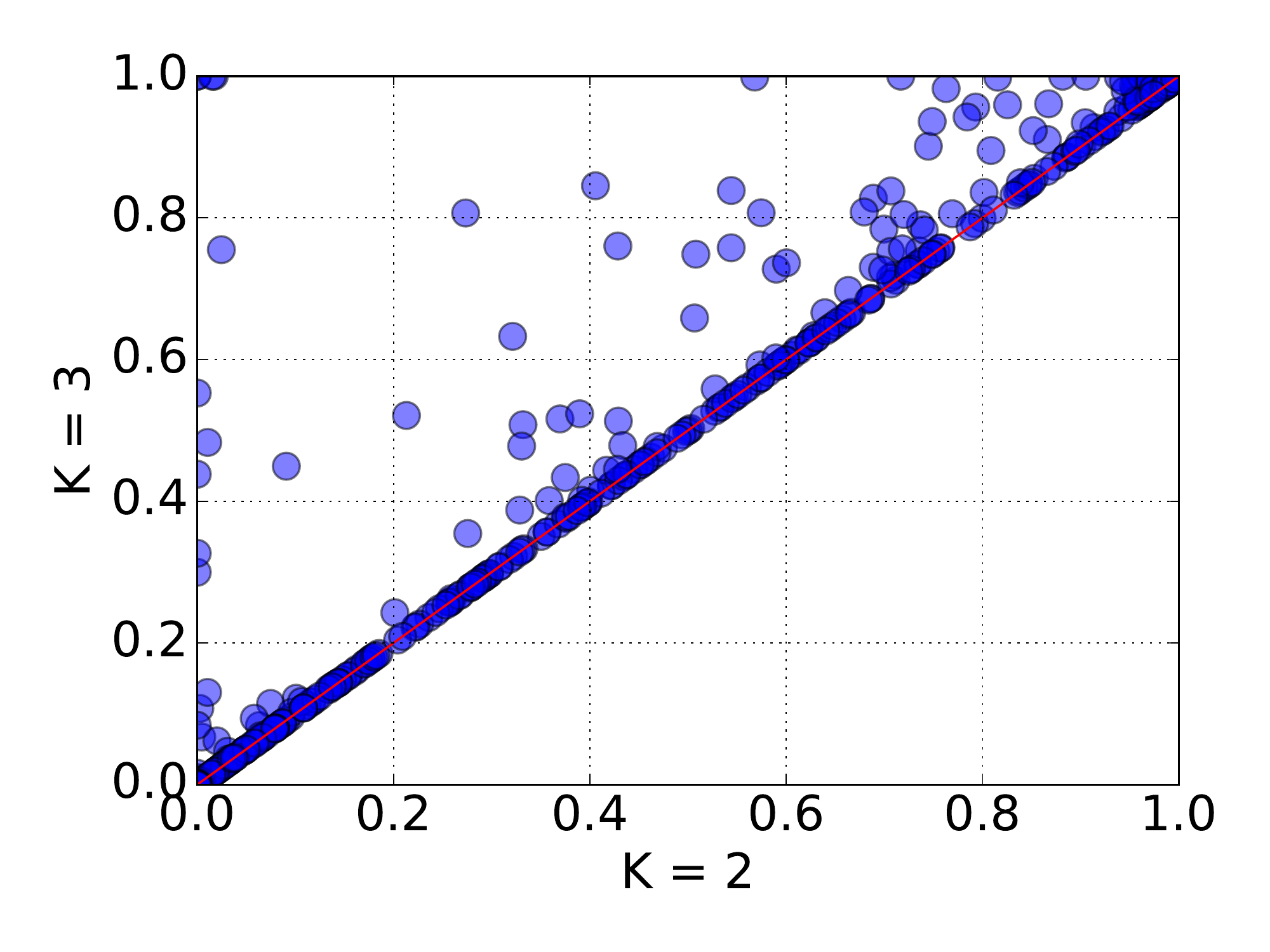}
 \end{minipage}
 \caption{Scatter plot comparing the root gap closed values of the \experimentRoot{} experiment comparing $(K=1,K=2)$,
   $(K=1,K=3)$, and $(K=2,K=3)$.  For example, each point $(x,y) \in [0,1]^2$ in the top left plot corresponds to an
   instance for which $100x$\% of the gap is closed by $K=1$ and $100y$\% closed by $K=2$.}
 \label{fig:root:scatter}
\end{figure}

\begin{table}
 \centering
 \begin{tabular*}{0.7\textwidth}{@{\extracolsep{\fill}}rrrrr}
  \toprule group & \colinstances & $K=1$ & $K=2$ & $K=3$ \\ \midrule \texttt{ALL} & $\rootAllSize$ & $\rootGapAllKOne$\%
  & $\rootGapAllKTwo$\% & $\rootGapAllKThree$\% \\ $\nnlconss \ge 10$ & $\rootNconssTenSize$ &
  $\rootGapNconssTenKOne$\% & $\rootGapNconssTenKTwo$\% & $\rootGapNconssTenKThree$\% \\ $\nnlconss \ge 20$ &
  $\rootNconssTwentySize$ & $\rootGapNconssTwentyKOne$\% & $\rootGapNconssTwentyKTwo$\% & $\rootGapNconssTwentyKThree$\%
  \\ $\nnlconss \ge 50$ & $\rootNconssFiftySize$ & $\rootGapNconssFiftyKOne$\% & $\rootGapNconssFiftyKTwo$\% &
  $\rootGapNconssFiftyKThree$\% \\ \midrule \texttt{AFFECTED} \\ \midrule \texttt{ALL} & $\rootAffectedSize$ &
  $\rootGapAffectedKOne$\% & $\rootGapAffectedKTwo$\% & $\rootGapAffectedKThree$\% \\ $\nnlconss \ge 10$ &
  $\rootNconssTenAfSize$ & $\rootGapNconssTenAfKOne$\% & $\rootGapNconssTenAfKTwo$\% & $\rootGapNconssTenAfKThree$\%
  \\ $\nnlconss \ge 20$ & $\rootNconssTwentyAfSize$ & $\rootGapNconssTwentyAfKOne$\% & $\rootGapNconssTwentyAfKTwo$\%
  & $\rootGapNconssTwentyAfKThree$\% \\ $\nnlconss \ge 50$ & $\rootNconssFiftyAfSize$ & $\rootGapNconssFiftyAfKOne$\%
  & $\rootGapNconssFiftyAfKTwo$\% & $\rootGapNconssFiftyAfKThree$\% \\ \bottomrule
 \end{tabular*}
 \caption{Aggregated results for the \experimentRoot{} experiment. A row $\nnlconss \ge x$ considers all instances
   that have at least $x$ many nonlinear constraints. The second part of the table only considers instances for which at
   least one setting closes at least $1$\% of the root gap.}
 \label{table:root:gapclosed}
\end{table}

\paragraph{\experimentAlgo{} Experiment.}

Table~\ref{table:algo:gapclosed} reports aggregated results for the \experimentAlgo{} experiment, which, similar to
Figure~\ref{fig:root:scatter}, are visualized in Figure~\ref{fig:algo:scatter}. We refer to
Table~\ref{table:algo:detailed} in the appendix for detailed instance-wise results.

First, we observe that the \algoDefault{} performs significantly better than
\algoPlain{}. Table~\ref{table:algo:gapclosed} shows that on $\algoWinsPlainAll$ of the $\algoGroupSizeAll{}$ affected
instances \algoDefault{} closes at least 1\% more gap than \algoPlain{}.  Only on $\algoLosesPlainAll$ instances
\algoPlain{} closes more gap, but over all instances it closes on average $\floatsub{100.0}{\algoRelgapPlainAll}\%$ less
gap than \algoDefault{}. On instances with a larger number of nonlinear constraints, \algoDefault{} performs even
better: on the $\algoGroupSizeNconssFifty$ instances with at least $50$ nonlinear constraints, \algoDefault{} computes
$\algoWinsPlainNconssFifty$ times a better and only $\algoLosesPlainNconssFifty$ time a worse dual bound than
\algoPlain{}.  For these $\algoGroupSizeNconssFifty$ instances, \algoPlain{} closes
$\floatsub{100.0}{\algoRelgapPlainNconssFifty}$\% less gap than \algoDefault{}. Interestingly,
Figure~\ref{fig:algo:scatter} shows that there are instances for which \algoPlain{} could not close any gap but
\algoDefault{} could. There is no instance for which the opposite is true.

Next, we analyze which components of the Benders algorithm are responsible for the significantly better performance of
\algoDefault{} compared to \algoPlain{}.  Table~\ref{table:algo:gapclosed} shows that \algoDefault{} dominates
\algoNoTR{}, \algoNoSupp{}, and \algoNoEarly{} with respect to the average gap closed and the difference between the
number of wins and the number of losses on each subset of the instances.
The most important component is the early termination of the master problem. By disabling this feature, the Benders
algorithm closes $\floatsub{100.0}{\algoRelgapNoearlyAll}$\% less gap on all instances and even
$\floatsub{100.0}{\algoRelgapNoearlyNconssFifty}$\% on those which have at least $50$ nonlinear constraints.

Even though Table~\ref{table:algo:gapclosed} suggests that the trust-region and support stabilization are not crucial
for closing a significant portion of the root gap, both techniques are important to exploit the $\lambda$ space in a
more structured way.
Once an improving $\lambda$ vector is found, it is likely that there are even better vectors in its neighborhood. The
proposed stabilization methods help us to explore this neighborhood and to converge to a local optimum. Overall, this
helps us to find better aggregation vectors faster.
To visualize this, we use the instance \texttt{genpooling\_lee1} from Example~\ref{ex:gbenders:example}.
Figure~\ref{fig:algo:genpooling_lee1} shows the achieved dual bounds and the sparsity pattern of the $\lambda$ vector in
each iteration of the Benders algorithm for \algoDefault{} and \algoNoTR{}. Both settings run with an iteration limit
of~$600$.

First, we observe that the achieved dual bound of $-4775.26$ with \algoDefault{} is significantly better than the dual
bound of $-5006.95$ when using \algoNoTR{}. The best dual bound is found after $97$ iterations by \algoDefault{} and
after $494$ iterations with \algoNoTR{}. To understand this behavior, we analyze the computed aggregation vectors in
each iteration. After \algoDefault{} finds an aggregation that improves the dual bound, it fixes the support of the
aggregation vector and tries to improve the dual bound by finding a better aggregation vector for that fixed
support. This happens at the beginning of the solving process and after iteration $62$, which is visible in the bottom
left plot of Figure~\ref{fig:algo:genpooling_lee1}. After iteration $112$ the algorithm removed the trust region and
support fixation and no further dual bound improvement could be found. Due to the nature of the Benders algorithm, the
algorithm frequently oscillates in the $\lambda$ space if no stabilization is used. This is displayed in the ``noisy''
parts of the $\lambda$ plots of \algoDefault{} and \algoNoTR{} in Figure~\ref{fig:algo:genpooling_lee1}.
In the iterations where no stabilization is used by \algoDefault{}, we do not observe any pattern indicating which of
the two settings finds a better dual bound ---the behavior seems rather random. This type of randomness and the large
time limit used explain the similar results for the achieved gap closed values for \algoDefault{} and \algoNoTR{} that
are reported in Table~\ref{table:algo:gapclosed}.
The important observation is that using the presented stabilization methods allows us to reach the final dual bound much
faster than without using stabilization.

\begin{table}[tb]
  \centering
  \footnotesize
  \begin{tabular*}{1.0\textwidth}{@{\extracolsep{\fill}}rr|rrcrrcrrcrrc}
   \toprule
   & & \multicolumn{3}{c}{\algoPlain} & \multicolumn{3}{c}{\algoNoTR} & \multicolumn{3}{c}{\algoNoSupp} & \multicolumn{3}{c}{\algoNoEarly} \\
   group & $|P|$ & \colwins & \colloses & \colrgapclosed & \colwins & \colloses & \colrgapclosed & \colwins & \colloses & \colrgapclosed & \colwins & \colloses & \colrgapclosed \\
   \midrule
   \texttt{ALL}              & $\algoGroupSizeAll$          & $\algoWinsPlainAll$          & $\algoLosesPlainAll$          & $\algoRelgapPlainAll$          & $\algoWinsNotrustAll$          & $\algoLosesNotrustAll$          & $\algoRelgapNotrustAll$          & $\algoWinsNosuppAll$          & $\algoLosesNosuppAll$          & $\algoRelgapNosuppAll$          & $\algoWinsNoearlyAll$          & $\algoLosesNoearlyAll$          & $\algoRelgapNoearlyAll$\\
   \texttt{$\nnlconss \ge 10$} & $\algoGroupSizeNconssTen$    & $\algoWinsPlainNconssTen$    & $\algoLosesPlainNconssTen$    & $\algoRelgapPlainNconssTen$    & $\algoWinsNotrustNconssTen$    & $\algoLosesNotrustNconssTen$    & $\algoRelgapNotrustNconssTen$    & $\algoWinsNosuppNconssTen$    & $\algoLosesNosuppNconssTen$    & $\algoRelgapNosuppNconssTen$    & $\algoWinsNoearlyNconssTen$    & $\algoLosesNoearlyNconssTen$    & $\algoRelgapNoearlyNconssTen$\\
   \texttt{$\nnlconss \ge 20$} & $\algoGroupSizeNconssTwenty$ & $\algoWinsPlainNconssTwenty$ & $\algoLosesPlainNconssTwenty$ & $\algoRelgapPlainNconssTwenty$ & $\algoWinsNotrustNconssTwenty$ & $\algoLosesNotrustNconssTwenty$ & $\algoRelgapNotrustNconssTwenty$ & $\algoWinsNosuppNconssTwenty$ & $\algoLosesNosuppNconssTwenty$ & $\algoRelgapNosuppNconssTwenty$ & $\algoWinsNoearlyNconssTwenty$ & $\algoLosesNoearlyNconssTwenty$ & $\algoRelgapNoearlyNconssTwenty$\\
   \texttt{$\nnlconss \ge 50$} & $\algoGroupSizeNconssFifty$  & $\algoWinsPlainNconssFifty$  & $\algoLosesPlainNconssFifty$  & $\algoRelgapPlainNconssFifty$  & $\algoWinsNotrustNconssFifty$  & $\algoLosesNotrustNconssFifty$  & $\algoRelgapNotrustNconssFifty$  & $\algoWinsNosuppNconssFifty$  & $\algoLosesNosuppNconssFifty$  & $\algoRelgapNosuppNconssFifty$  & $\algoWinsNoearlyNconssFifty$  & $\algoLosesNoearlyNconssFifty$  & $\algoRelgapNoearlyNconssFifty$\\
   \bottomrule
  \end{tabular*}
  \caption{Table shows aggregated results for the \experimentAlgo{} experiment. The column ``\colwins{}''/``\colloses{}'' reports the
    number of instances for which \algoDefault{} could close at least 1\% more/less root gap than the settings of
    the corresponding column. Column ``\colrgapclosed{}'' reports the average root gap closed relative to our default
    settings (in \%). Instances for which no setting could close at least 1\% of the root gap are filtered out.}
  \label{table:algo:gapclosed}
 \end{table}

\begin{figure}[tb]
 \centering
 \begin{minipage}{0.40\textwidth}
  \centering
  \includegraphics[width=\textwidth]{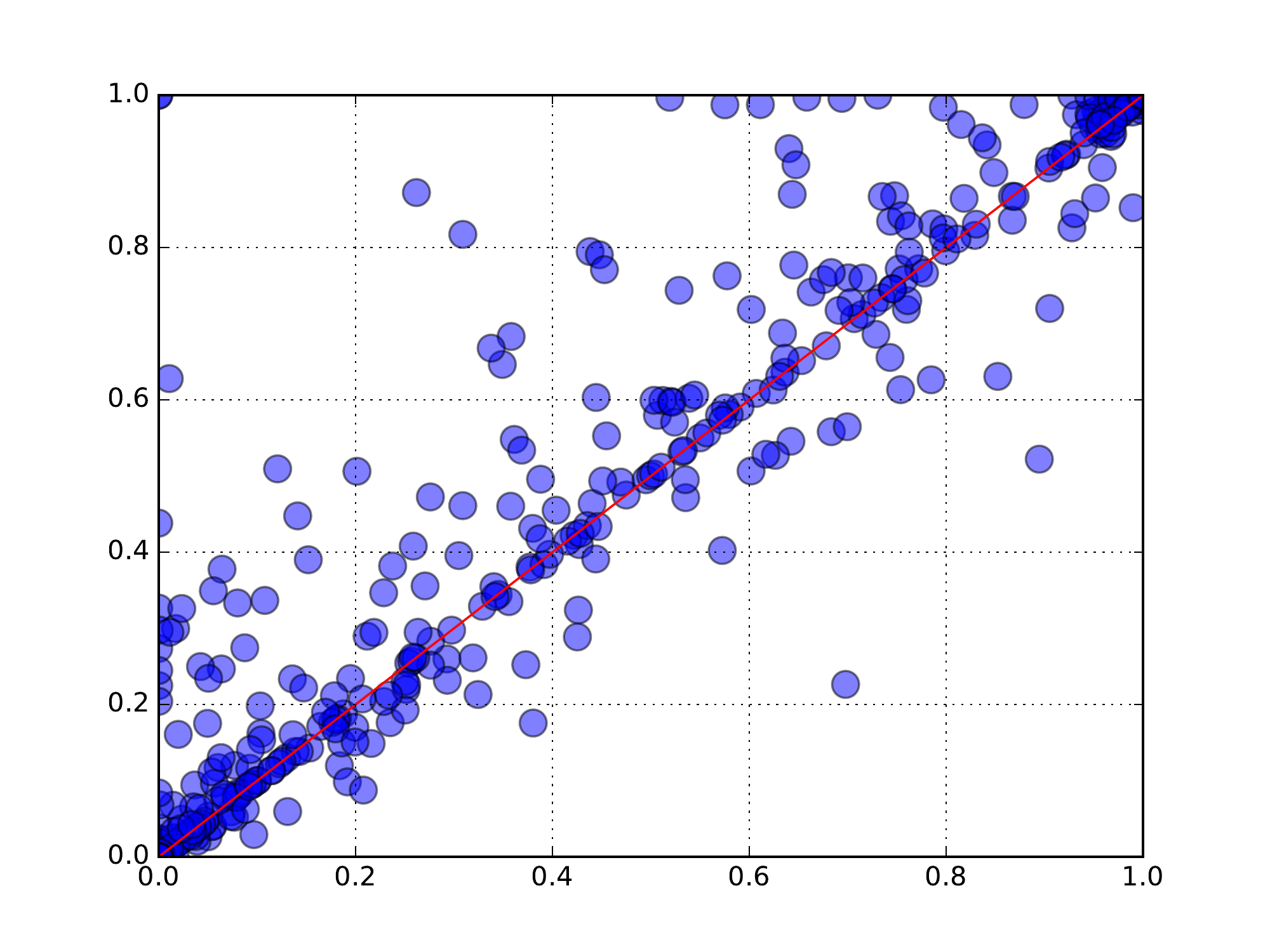}
  \subcaption{\algoPlain}
 \end{minipage}
 \begin{minipage}{0.40\textwidth}
  \centering
  \includegraphics[width=\textwidth]{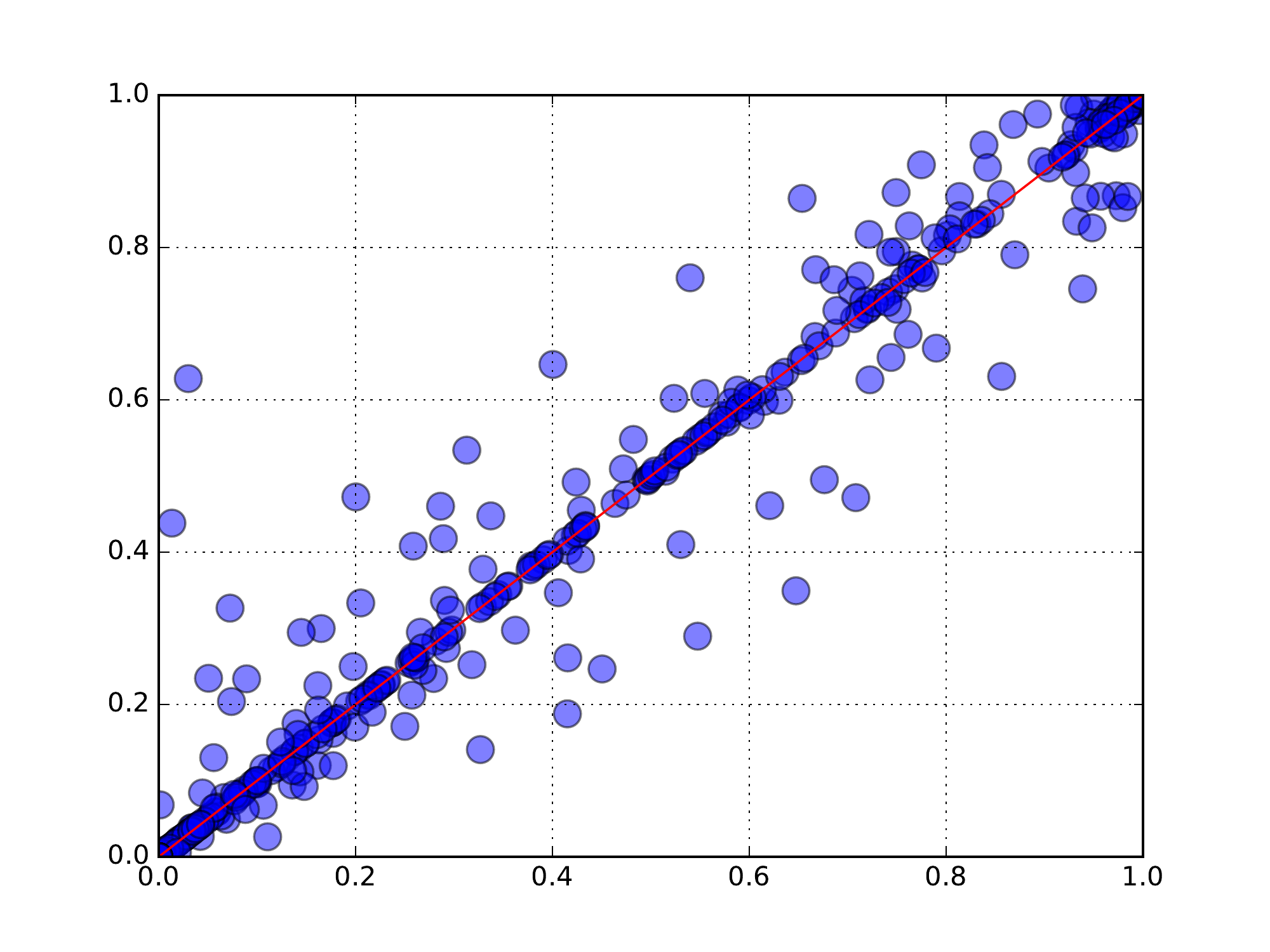}
  \subcaption{\algoNoTR}
 \end{minipage}

 \begin{minipage}{0.40\textwidth}
  \centering
  \includegraphics[width=\textwidth]{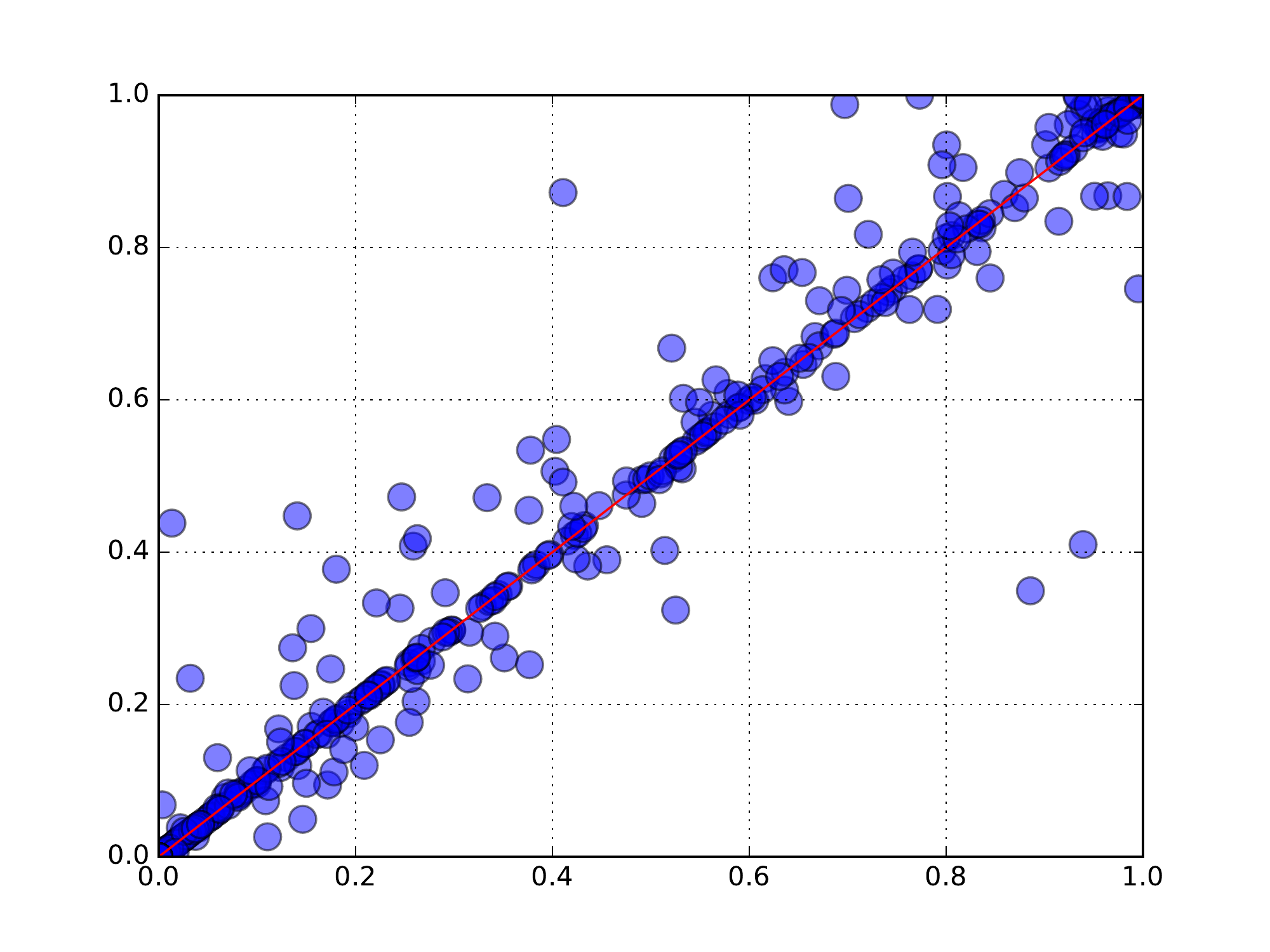}
  \subcaption{\algoNoSupp}
 \end{minipage}
 \begin{minipage}{0.40\textwidth}
  \centering
  \includegraphics[width=\textwidth]{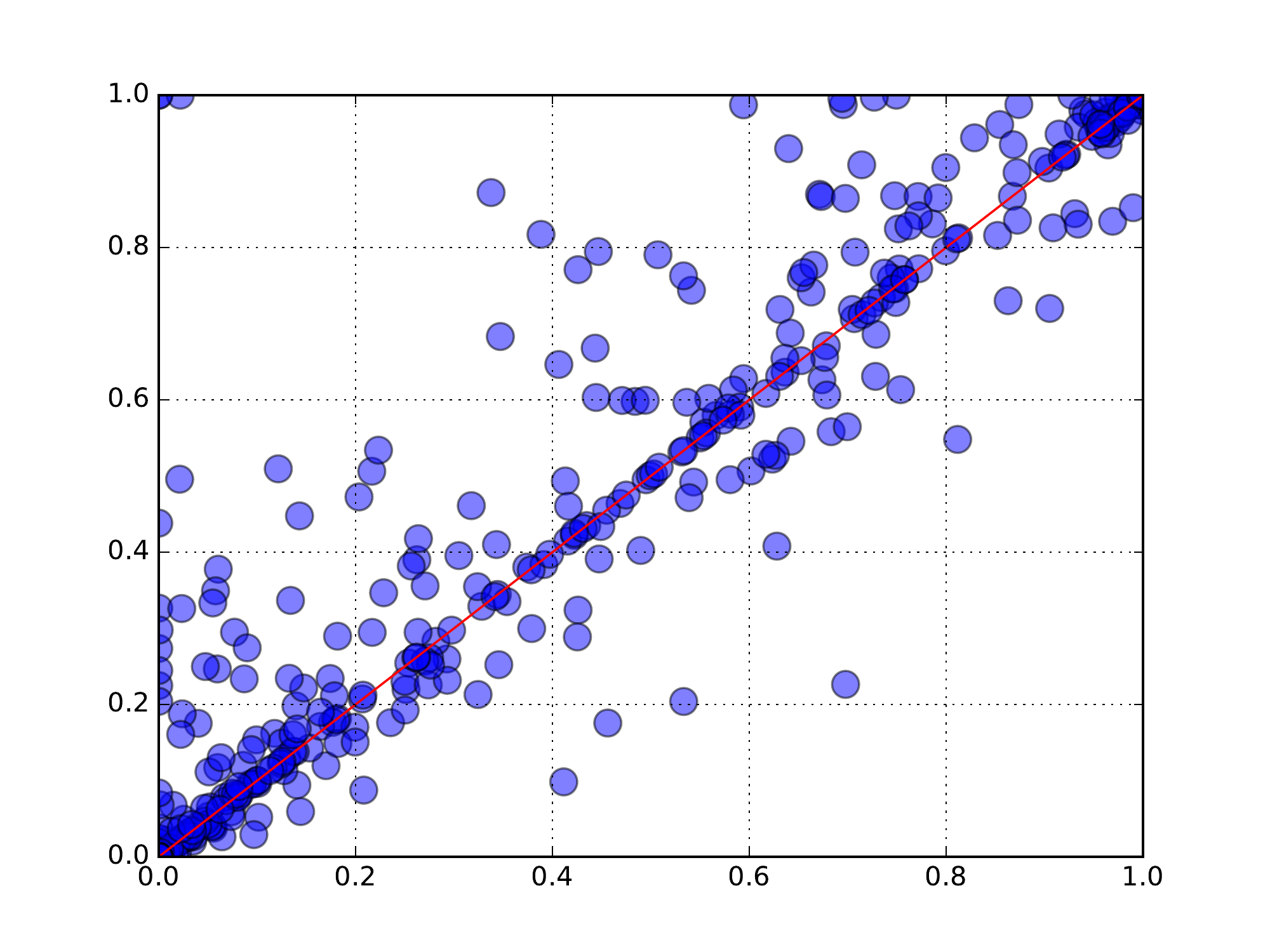}
  \subcaption{\algoNoEarly}
 \end{minipage}
 \caption{Scatter plots comparing the achieved root gap closed values for different settings of the Benders
   algorithm. The y-axis reports the gap closed values for \algoDefault{} and the x-axis the gap closed values
   for the mentioned setting.}
 \label{fig:algo:scatter}
\end{figure}

\begin{figure}[tb]
  \centering
  \begin{minipage}{0.48\textwidth}
    \centering \includegraphics[width=\textwidth]{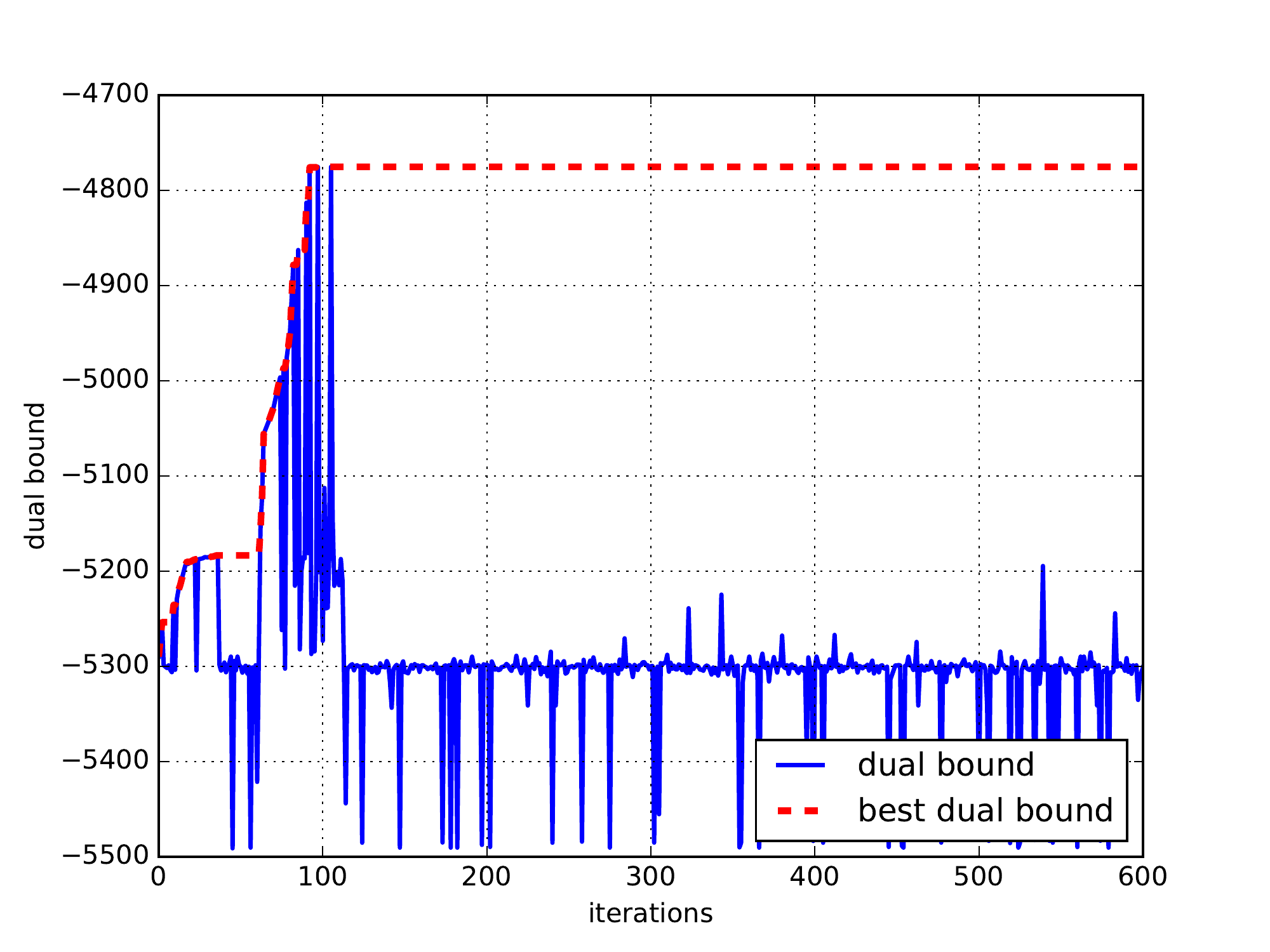}
  \end{minipage}
  \begin{minipage}{0.48\textwidth}
    \centering \includegraphics[width=\textwidth]{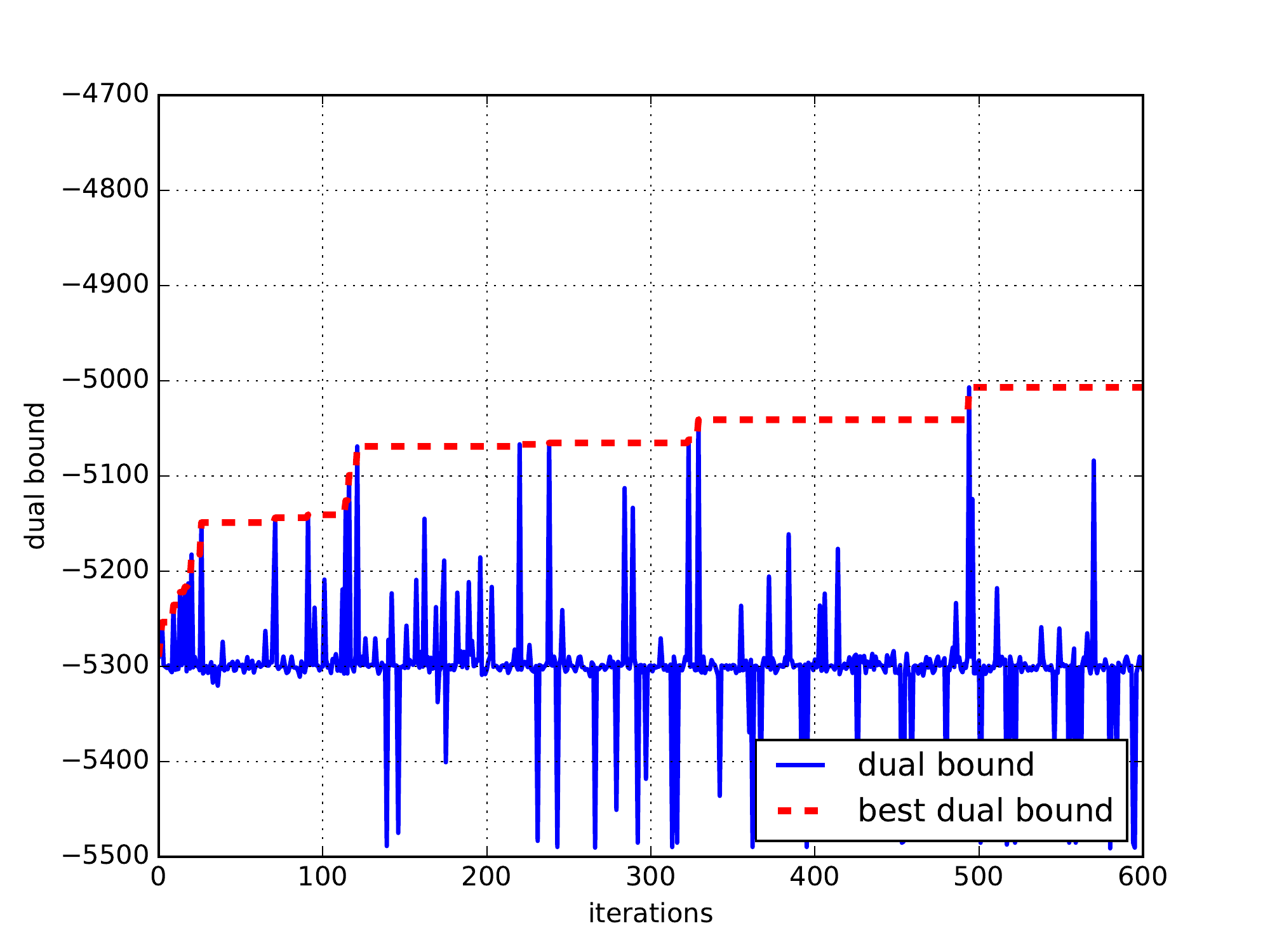}
  \end{minipage}

  \begin{minipage}{0.48\textwidth}
    \centering \includegraphics[width=0.85\textwidth]{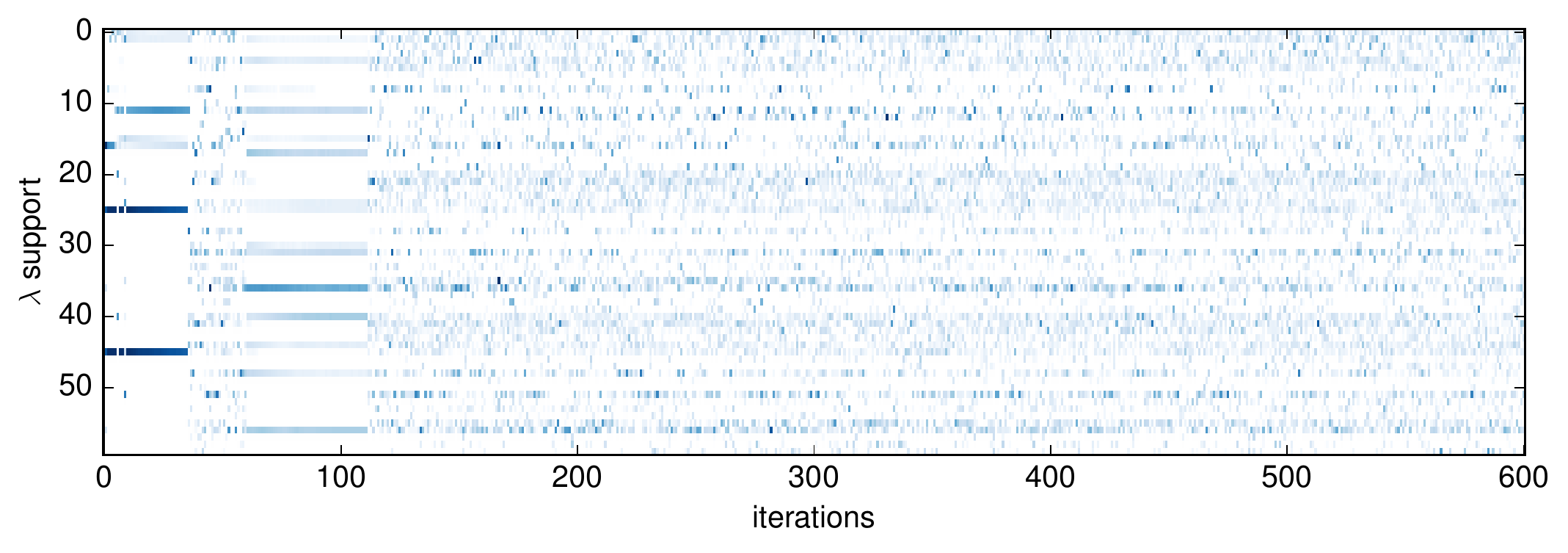}
  \end{minipage}
  \begin{minipage}{0.48\textwidth}
    \centering \includegraphics[width=0.85\textwidth]{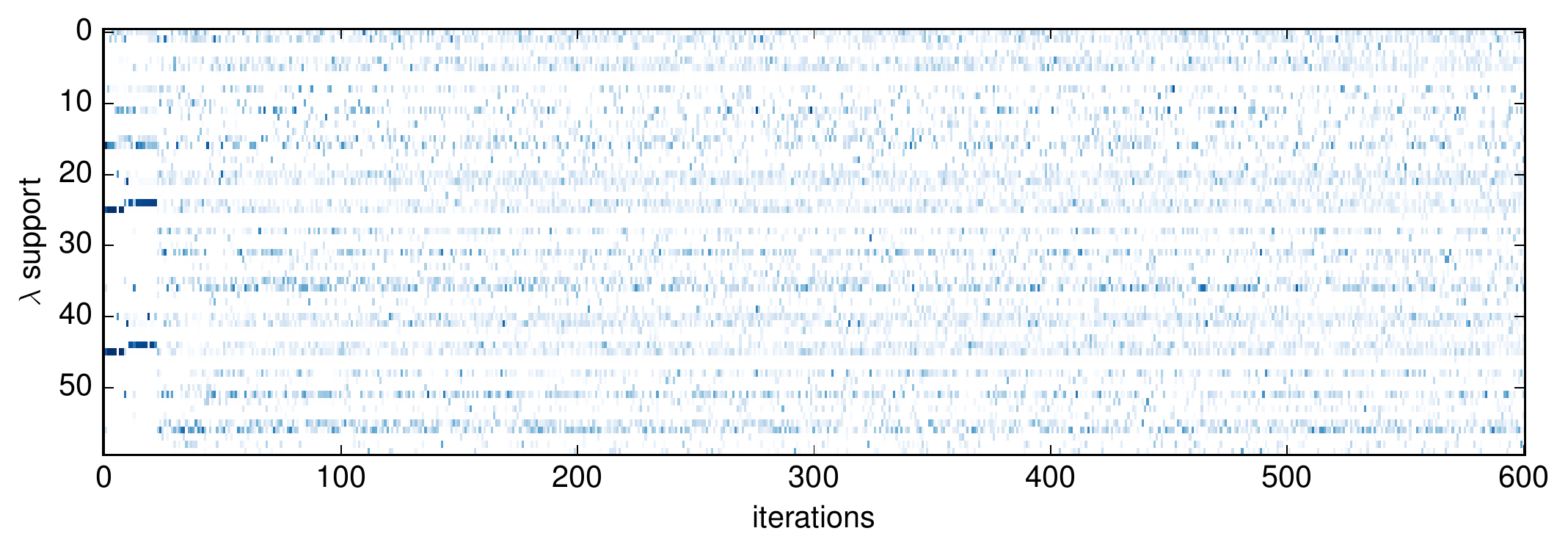}
  \end{minipage}
  \caption{Comparison of the achieved dual bounds and the obtained aggregation vectors during Algorithm~\ref{algo:genbenders}
    for \algoDefault{} (left) and \algoNoTR{} (right) on \texttt{genpooling\_lee1} using $K=3$. The red dashed
    curve shows the best found dual bound so far, whereas the blue curve shows the computed dual bound in every
    iteration. The pictures at the bottom visualize the $\lambda$ vector. White means that an entry $\lambda_i$ is zero
    and blue means that it is one.}
  \label{fig:algo:genpooling_lee1}
\end{figure}

\paragraph{\experimentDual{} Experiment.}

For this experiment, we include all instances which could not be solved by \scip with default settings within three
hours, have a final gap of at least ten percent, terminate without an error, and contain at least four nonlinear
constraints. To compute gaps we use the best known primal bounds from the \minlplib as reference values. This leaves in
total~$\treeExpNinstances$ instances for the \experimentDual{} experiment.
Table~\ref{table:algo:detailed} in the appendix reports detailed results on the subset of instances for which the
Algorithm~\ref{algo:genbenders} was able to improve on the bound obtained by \scip with default settings, which was the case for
$\treeExpNinstancesBetter$ of the~$\treeExpNinstances$ instances. On these instances, the average gap of
$\treeExpAvgGapSCIP$\% for \scip with default settings could be reduced to an average gap of $\treeExpAvgGapBenders$\%.

Two interesting subsets of instances are the \texttt{rsyn*} and \texttt{syn*} instances. These instances contain a large
number of integer variables and linear constraints but all nonlinear constraints are convex. Note that in this case the
aggregation constraints remain convex and thus each sub-problem in the Algorithm~\ref{algo:genbenders} is a convex
optimization problem with integrality constraints. The advantage of using a surrogate relaxation for these problems is
that such relaxation is able to capture the important nonlinear structure of the problem using a small \emph{convex}
problem that can be solved substantially faster with branch and bound. This explains the better dual bounds compared to
default \scip after three hours.

Algorithm~\ref{algo:genbenders} computes strong dual bounds on difficult nonconvex \MINLPs: For example, for all
\texttt{polygon*} instances and four \texttt{facloc*} instances, we find better bounds than the reported best known dual bounds from the
\minlplib, as shown in Table~\ref{table:algo:betterdbs}.

\begin{table}[ht]
  \centering
  \begin{tabular*}{1.0\textwidth}{@{\extracolsep{\fill}}lrrrr}
    \toprule
    instance & \colprimal & \coldb{} (\minlplib) & \coldb{} (\scip) & \experimentDual{} \\
    \midrule
    \texttt{polygon25}  & $-0.78$ &  $-5.80$ &  $-4.24$ & $-3.94$  \\
    \texttt{polygon50}  & $-0.78$ & $-15.27$ & $-10.78$ & $-8.72$  \\
    \texttt{polygon75}  & $-0.78$ & $-24.87$ & $-16.82$ & $-13.55$ \\
    \texttt{polygon100} & $-0.78$ & $-34.00$ & $-24.37$ & $-19.03$ \\
    \texttt{facloc1\_3\_95} & $12.30$ & $4.46$ & $5.50$ & $5.70$ \\
    \texttt{facloc1\_4\_80} &  $7.88$ & $0.16$ & $0.09$ & $0.41$ \\
    \texttt{facloc1\_4\_90} & $10.46$ & $0.48$ & $0.49$ & $1.18$ \\
    \texttt{facloc1\_4\_95} & $11.18$ & $0.79$ & $1.40$ & $2.40$ \\
    \bottomrule
  \end{tabular*}
  \caption{Comparison between the dual bounds computed by \scip, Algorithm~\ref{algo:genbenders}, and the best known dual
    bounds reported in the \minlplib for all \texttt{polygon*} and four \texttt{facloc*} instances.}
  \label{table:algo:betterdbs}
\end{table}

In general, we have observed the following behavior of Algorithm~\ref{algo:genbenders}. Due to the target dual bound,
the first iterations are processed quickly because the master problem~\eqref{eq:gbenders:master:milp} is easy to solve
and \scip rapidly finds feasible solutions for the sub-problems, which trigger the early termination criterion. After
this first phase, the Benders algorithm finds a promising aggregation vector, i.e., \scip does not find a feasible
solution with an objective value below the target dual bound.
Interestingly, we observe that $\gsfct{K}(\lambda)$ cannot be solved to global optimality within the time limit for most
of the~$\treeExpNinstances$ instances of the \experimentDual{}. However, the dual bound obtained by optimizing
$\gsfct{K}(\lambda)$ is often significantly better than the dual bound obtained by optimizing~\eqref{eq:minlp} when
using the same working limits.

\section{Surrogate duality during the tree search}
\label{section:bandb}

In the previous sections, we focused on developing computational techniques that can improve the performance of a dual
bounding procedure based on surrogate duality. While the obtained dual bounds are strong, in general complex instances
will still require branching in order to solve them to provable optimality. Additionally, even though the presented
computational techniques improve the running time of Algorithm~\ref{algo:genbenders}, it is still too costly to be
used in every node of a branch-and-bound tree.

In this section, we present a technique that incorporates Algorithm~\ref{algo:genbenders} into spatial branch and
bound. The technique focuses on extracting information of a \emph{single} execution of Algorithm~\ref{algo:genbenders}
in the root node, and reuses this information during spatial branch and bound.

Let $\Lambda := \{\lambda_1, \ldots, \lambda_L\} \subseteq \R^{K \nnlconss}$ be the set of aggregation vectors that
have been computed during Algorithm~\ref{algo:genbenders} in the root node of the branch-and-bound tree. We consider
only those aggregations that imply a tighter dual bound than the \MILP relaxation. Instead of using the generalized
Benders algorithm in a local node $v$, we select the most promising aggregation vector $\lambda$ from $\Lambda$ and
solve $\gsfct{K}_v(\lambda)$, which is equal to $\gsfct{K}(\lambda)$ except that the global linear relaxation is
replaced with a linear relaxation that is only locally valid in $v$.

We propose the following procedure. If $\gsfct{K}_v(\lambda)$ results in a better dual bound than the local \MILP
relaxation, i.e., $\gsfct{K}_v(0)$, then we skip the remaining aggregation vectors in $\Lambda$ and continue with the
tree search. If the dual bound does not improve, then we discard $\lambda$ in the sub-tree with root $v$. The intuition
behind discarding aggregations as we search down the tree is twofold. First, since the aggregations are computed in the
root node, their ability to provide good dual bounds is expected to deteriorate with the increasing depth of an explored
node. Second, we would like to alleviate the computational load of checking for too many aggregations as the
branch-and-bound tree-size increases. The idea is stated in Algorithm~\ref{algo:localapprox}.

\begin{algorithm}[tb]
  \caption{Surrogate approximation}
  \label{algo:localapprox}
  \begin{algorithmic}[1]
    \REQUIRE{node $v$, parent node $p$, parent aggregation candidates $C_{p} \subseteq \Lambda$}
    \ENSURE{$D \in \R$ valid dual bound for $v$, aggregation candidates $C_{v} \subseteq C_{p}$}
    \STATE{initialize $D \leftarrow \gsfct{K}_v(0)$ \label{algo:localapprox:milprelax}}
    \STATE{initialize $C_{v} \leftarrow C_{p}$}
    \FOR{$\lambda \in C_{v}$}
      \IF{$D < \gsfct{K}_v(\lambda)$}
        \RETURN{$(\gsfct{K}_v(\lambda),C_v)$ \label{algo:localapprox:found}}
      \ELSE
        \STATE{$C_v \leftarrow C_v \, \backslash \, \{\lambda\}$ \label{algo:localapprox:filter}}
      \ENDIF
    \ENDFOR
    \RETURN{$(D,\emptyset)$ \label{algo:localapprox:end}}
  \end{algorithmic}
 \end{algorithm}

\begin{figure}[tb]
  \centering
  \includegraphics[width=1.0\textwidth]{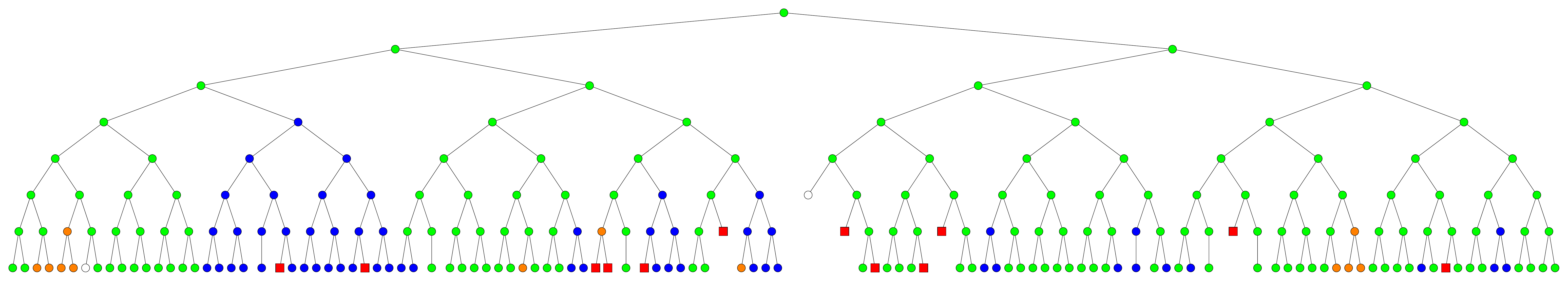}
  \caption{A visualization of Algorithm~\ref{algo:localapprox} for the instance \texttt{himmel16} of the \minlplib. The
    size of the aggregation pool has been limited to three and we used bread-first-search as the node selection
    strategy. The colors determine $|C_v|$ at each node $v$: green for three, blue for two, orange for one, white for
    zero aggregations. Red square shaped nodes could be pruned by Algorithm~\ref{algo:localapprox}, i.e., the proven
    dual bound exceeds the value of an incumbent solution.}
  \label{fig:localapprox:himmel16}
\end{figure}

Let $C_v \subset \Lambda$ denote the candidate aggregations in node $v$. The algorithm assumes $C_r = \Lambda$ for the
root node $r$. First, the value of the \MILP relaxation of node $p$ is computed in
Step~\ref{algo:localapprox:milprelax}. Each candidate aggregation $\lambda \in C_p$ is used to compute a tighter bound
for $v$. If $\gsfct{K}_v(\lambda)$ improves upon the \MILP relaxation, then the algorithm terminates in
Step~\ref{algo:localapprox:found}. Otherwise, $\lambda$ is discarded from the set of candidates in $v$ (see
Step~\ref{algo:localapprox:filter}).  In case no $\lambda \in C_p$ leads to a tighter dual bound, the algorithm returns
in Step~\ref{algo:localapprox:end} the value of the \MILP relaxation and the empty set as the set of aggregation
candidates for $v$.

As illustrated for the instance \texttt{himmel16} in Figure~\ref{fig:localapprox:himmel16},
Algorithm~\ref{algo:localapprox} might lead to stronger dual bounds in local nodes of the branch-and-bound tree, which
could result in a smaller tree. However, for the challenging instances of the \experimentDual{} experiment we observed
that solving $\gsfct{K}_v(\lambda)$ is too costly and almost always runs into the time limit. In these cases, we cannot
improve the dual bound.
An exception is instance \texttt{multiplants\_mtg1c}. The instance contains $28$ nonlinear constraints, $193$ continuous
variables, and $104$ integer variables. \scip with default settings proves a dual bound of $4096.04$, which is improved
by Algorithm~\ref{algo:genbenders} to $3161.13$. Algorithm~\ref{algo:localapprox} can further improve the dual bound
to $2935.58$ in the beginning of the search tree. Overall, there were only seven nodes for which
Algorithm~\ref{algo:localapprox} could improve a local dual bound. Afterward, all aggregation candidates have been
filtered out and \scip processed $491860$ nodes in total. During the exploration of these nodes, \scip could not improve
the dual bound further.

\section{Conclusion}
\label{section:conclusion}

In this article, we studied theoretical and computational aspects of surrogate relaxations for \MINLPs. We developed
the first algorithm to solve a generalization of the surrogate dual problem that allows multiple aggregations of nonlinear
constraints.
To this end, we adapted a Benders-type algorithm for solving the classic surrogate dual problem to solve its generalization and proved that the algorithm
always converges. Besides computational enhancements for solving the classic and generalized surrogate dual problem, we discussed
how to exploit surrogate duality in a spatial branch-and-bound solver to obtain strong dual bounds for difficult
nonconvex \MINLPs.

Our extensive computational study on the heterogeneous set of publicly available instances of the \minlplib, which used an
implementation in the \MINLP solver \scip, showed that exploiting surrogate duality can lead to significantly better dual
bounds than using \scip with default settings. Concretely, solving the classic surrogate dual problem led to a root gap reduction of $\rootGapAllKOne$\% on all
$\rootAllSize$ instances and $\rootGapAffectedKOne$\% on $\rootAffectedSize$ affected instances. The presented
generalization of surrogate duality reduced the root gap further, namely by $\rootGapAllKThree$\% on all
instances and by $\rootGapAffectedKThree$\% on the affected instances.
Additionally, our experiments showed that the presented computational enhancements are important to obtain good dual bounds
for problems with a large number of nonlinear constraints. On the
$\algoGroupSizeNconssFifty$ instances with at least $50$ nonlinear constraints, our implementation of the generalized
Benders algorithm closed $\floatsub{100.0}{\algoRelgapPlainNconssFifty}$\% more root gap when our proposed trust region
stabilization, support stabilization, and early termination criterion are used.
Finally, our tree experiments showed that using the result of Algorithm~\ref{algo:genbenders} during the tree search can
lead to significantly better dual bounds than solving \MINLPs with standard spatial branch and bound. On very difficult
\MINLPs, we achieved an average gap reduction from $\treeExpAvgGapSCIP$\% to $\treeExpAvgGapBenders$\%.

Finally, we want to highlight two out of many open questions that remain related to generalized surrogate duality and its application in branch-and-bound solvers.
First, consider the case that each constraint of~\eqref{eq:minlp} is quadratic, i.e., $g_i(x) = x^\T Q_i x + q_i^\T x +
b_i$ for each $i \in \nlconssidx$.  Note that adding the constraints
\begin{equation*}
  \sum_{i \in \nlconssidx} \lambda^k_i Q_i \psd 0
\end{equation*}
for all $k \in \{1,\ldots,K\}$ to the master problem~\eqref{eq:gbenders:master:milp} enforces that each sub-problem is a
convex mixed-integer quadratically constrained program. This increases the complexity of the master problem but, at the
same time, reduces the complexity of the sub-problems. A computational study of surrogate relaxations using this
modification would be interesting on instances for which solving the sub-problems is currently too expensive.

Second, it remains an open question how a pure surrogate-based spatial branch-and-bound approach could perform in
practice. All generated points $\pointset \subseteq \minlprelaxset$ of a parent node can be used as an initial set of points
in a child node, which could be considered as a warm-start strategy. However, it is not clear how branching decisions
would affect the dual bounds obtained by solving a surrogate relaxation. Future work might design a branching rule that
tries to improve the dual bounds obtained by solving surrogate relaxations.
 
\section*{Acknowledgments}
This work has been supported by the Research Campus MODAL \emph{Mathematical
  Optimization and Data Analysis Laboratories} funded by the Federal Ministry of
Education and Research (BMBF Grant~05M14ZAM).  All responsibilty for the content
of this publication is assumed by the authors.
The described research activities are funded by the Federal Ministry for Economic Affairs
and Energy within the project EnBA-M (ID: 03ET1549D).
The authors thank the Schloss Dagstuhl – Leibniz Center for Informatics for hosting
the Seminar 18081 "Designing and Implementing Algorithms for Mixed-Integer Nonlinear
Optimization" for providing the environment to develop the ideas in this paper.
 
\bibliographystyle{spmpsci}
\bibliography{surrogate}

\section*{Appendix}
\begin{landscape}
    \scriptsize
\begin{longtable}{@{\extracolsep{\fill}}lrrrr|rrr|rrr|rrr}
        \caption{Detailed results for \experimentRoot{} experiments comparing the achieved root gap closed values for $K \in \{1,2,3\}$ for all instances of the \minlplib~\cite{MINLPLIB}.
            The generalized Benders algorithm is called after processing the root node of the branch-and-bound tree. The reported gap closed values are relative to the root \MILP
            relaxation value and the value of best known primal solution.}
        \label{table:root:detailed}\\[1ex]
        \tabledefline{\colcons}{total number of nonlinear constraints}
        \tabledefline{\colobjsense}{objective sense}
        \tabledefline{\colprimal}{best known primal solution}
        \tabledefline{\colmilp}{\MILP relaxation value}
        \tabledefline{\colsurrogate}{surrogate relaxation value}
        \tabledefline{\coliter}{total number of surrogate master iterations (``tilim''/``iterlim'' for the time/iteration limit}
        \tabledefline{\colgapclosed}{gap closed by \colsurrogate{} w.r.t. ``\colmilp{}'' and ``\colprimal{}'' column}
        \toprule
        & & & &  & \multicolumn{3}{c|}{K=1} & \multicolumn{3}{c|}{K=2} & \multicolumn{3}{c}{K=3}\\
        instance & \colcons & \colobjsense & \colprimal & \colmilp & \colsurrogate & \coliter & \colgapclosed & \colsurrogate & \coliter & \colgapclosed & \colsurrogate & \coliter & \colgapclosed \\
        \toprule
        \endfirsthead
        \textbf{Table~\ref{table:root:detailed}} continued\\
        \toprule
        & & & &  & \multicolumn{3}{c|}{K=1} & \multicolumn{3}{c|}{K=2} & \multicolumn{3}{c}{K=3}\\
        instance & \colcons & \colobjsense & \colprimal & \colmilp & \colsurrogate & \coliter & \colgapclosed & \colsurrogate & \coliter & \colgapclosed & \colsurrogate & \coliter & \colgapclosed \\
        \toprule
        \endhead
           10bar1A &         10 &        min &       1674 &       1638&       1674 &         72 & 1.00 &       1674 &         60 & 1.00 &       1674 &         57 & 1.00 \\
   10bar1B &         10 &        min &    1623.09 &    1587.09&    1622.73 &         65 & 0.99 &    1623.09 &         66 & 1.00 &    1623.09 &        121 & 1.00 \\
   10bar1C &         10 &        min &    1623.09 &    1485.26&     1556.9 &         38 & 0.52 &    1571.31 &         41 & 0.62 &    1571.31 &         37 & 0.62 \\
   10bar1D &         10 &        min &    1623.09 &    1485.26&    1536.18 &        269 & 0.37 &    1536.18 &         26 & 0.37 &     1556.4 &         32 & 0.52 \\
    10bar2 &         20 &        min &    1954.23 &    1816.92&    1817.28 &         12 & 0.00 &    1817.28 &          5 & 0.00 &    1831.68 &         21 & 0.11 \\
   4stufen &         33 &        min &     116330 &     100935&     101307 &        128 & 0.02 &     101307 &       2532 & 0.02 &     101307 &        948 & 0.02 \\
     72bar &        144 &        min &    75.9049 &    71.5849&    71.5849 &        249 & 0.00 &    71.5849 &        255 & 0.00 &    71.5849 &        249 & 0.00 \\
     90bar &        180 &        min &    97.5374 &    91.5374&    91.5374 &        190 & 0.00 &    91.5374 &        192 & 0.00 &    91.5374 &        178 & 0.00 \\
\end{longtable}
\end{landscape}

\begin{landscape}
    \scriptsize
    \begin{longtable}{@{\extracolsep{\fill}}lrrrr|rrrrrr}
    \caption{Detailed results for \experimentAlgo{} experiments comparing the achieved root gap closed values for $K=3$ when using different settings for the Benders algorithm. The reported
        values are relative to the value of \scip's \MILP relaxation after processing the root node and the best known primal solution in the \minlplib~\cite{MINLPLIB}.}
    \label{table:algo:detailed}\\[1ex]
    \tabledefline{\colcons}{total number of nonlinear constraints}
    \tabledefline{\colobjsense}{objective sense}
    \tabledefline{\colprimal}{best known primal solution}
    \tabledefline{\colmilp}{\MILP relaxation value}
    \tabledefline{\algoDefault}{root gap closed for \algoDefault{} setting}
    \tabledefline{\algoPlain}{root gap closed for \algoPlain{} setting}
    \tabledefline{\algoNoTR}{root gap closed for \algoNoTR{} setting}
    \tabledefline{\algoNoSupp}{root gap closed for \algoNoSupp{} setting}
    \tabledefline{\algoNoEarly}{root gap closed for \algoNoEarly{} setting}
    \toprule
    instance & \colcons & \colobjsense & \colprimal & \colmilp & \algoDefault & \algoPlain & \algoNoTR & \algoNoSupp & \algoNoEarly \\
    \toprule
    \endfirsthead
    \textbf{Table~\ref{table:algo:detailed}} continued\\
    \toprule
    instance & \colcons & \colobjsense & \colprimal & \colmilp & \algoDefault & \algoPlain & \algoNoTR & \algoNoSupp & \algoNoEarly \\
    \toprule
    \endhead
       10bar1A & 10 &   min &       1674 &       1638 &  1.00 &  1.00 &  1.00 &  1.00 &  1.00\\
   10bar1B & 10 &   min &    1623.09 &    1587.09 &  1.00 &  0.99 &  1.00 &  1.00 &  0.99\\
   10bar1C & 10 &   min &    1623.09 &    1485.26 &  0.26 &  0.26 &  0.26 &  0.26 &  0.26\\
   10bar1D & 10 &   min &    1623.09 &    1485.26 &  0.41 &  0.26 &  0.26 &  0.26 &  0.63\\
    10bar2 & 20 &   min &    1954.23 &    1816.92 &  0.26 &  0.26 &  0.26 &  0.26 &  0.26\\
    10bar3 & 10 &   min &    5156.64 &    1883.97 &  0.06 &  0.09 &  0.09 &  0.06 &  0.06\\
    10bar4 & 20 &   min &    5647.06 &    2170.74 &  0.00 &  0.00 &  0.00 &  0.00 &  0.00\\
     25bar & 50 &   min &    387.073 &    268.495 &  0.23 &  0.05 &  0.05 &  0.03 &  0.13\\
\end{longtable}
\end{landscape} 
\end{document}